\documentclass[12pt]{article}
\usepackage{amsmath, amsthm, amsfonts}
\usepackage[margin=1 in]{geometry}
\usepackage{amssymb}
\usepackage{natbib}
\usepackage{hyperref}
\usepackage{float}
\usepackage{graphicx}
\usepackage{cleveref}
\usepackage{authblk}
\usepackage{xr,bm}
\usepackage{mathrsfs}
\usepackage{enumitem}
\usepackage[labelfont=bf]{caption}

\usepackage{xcolor}
\newcommand{\be} {\begin{eqnarray*}}
\newcommand{\ee} {\end{eqnarray*}}

\def\d{{ \mbox{d} }}
\theoremstyle{definition}
\newtheorem{definition}{Definition}[section]

\newcommand{\1}{\\[1ex]}

\newtheorem{theorem}{Theorem}[section]
\newtheorem*{theorem*}{Theorem}
\newtheorem{lemma}[theorem]{Lemma}
\newtheorem{ass}[theorem]{Assumption}

\newtheorem{proposition}[theorem]{Proposition}
\newtheorem{corollary}[theorem]{Corollary}
\newtheorem{rem}{Remark}[section]

\def\*#1{\bm{#1}}

\title{Sieve Estimation of Optimal Transport Maps from\\ Paired Data in Gaussian Spaces}

\author{
Xin Jin$^{1}$,
Kit Chan$^{2}$,
and Riddhi Pratim Ghosh$^{2}$\\
{\small $^{1}$Department of Mathematics, The University of Tampa}\\
{\small $^{2}$Department of Mathematics and Statistics, Bowling Green State University}
}
\date{}

\begin{document}
\maketitle
\begin{abstract}
We study the estimation of infinite-dimensional optimal transport maps from noisy paired observations. The population map pushes a Gaussian reference measure forward to a target probability measure on a function space and takes the Cameron--Martin gradient form $T=I+\nabla_{\mathcal H}\phi$. Our estimator uses cylindrical gradient sieves based on finitely many Cameron--Martin coordinates, thereby reducing the problem to finite-dimensional empirical risk minimization. A local nonasymptotic oracle inequality separates cylindrical approximation and statistical estimation errors while accounting for the conditioning of the parametrization. The approximation analysis relies on regularity conditions governing coordinate decay and dependence on omitted input coordinates. For diagonal Gaussian and nonlinear block-interaction classes, we derive matching upper and lower bounds that establish the minimax rate $N^{-s/(2s+1)}$ in expected norm, where $s$ measures weighted coordinate regularity. We further analyze a continuous two-groups model with Gaussian--Laplace mixtures and derive a prediction-risk bound for the resulting transport-map estimator.
\end{abstract}
	
\vspace{1cm}


\newpage

\section{Introduction}
\label{sec:intro}

In uncertainty quantification, probability distributions on function spaces are used to describe unknown curves, fields, and other functional quantities. Gaussian measures provide a natural reference for these distributions. In Bayesian inverse problems, for example, a Gaussian prior on an unknown function is updated by observations to obtain a posterior distribution on the same function space \citep{stuart2010inverse}. A transport map from the reference measure to the target distribution transforms reference samples into target samples \citep{marzouk2016sampling}. Related generative methods have also been developed for probability measures on Hilbert spaces \citep{kerrigan2023diffusion}. These problems motivate the estimation of transport maps in infinite-dimensional Gaussian spaces.

In this article, we study the transport map from noisy paired observations. For independent source draws $f_0^{(i)}$ from a Gaussian reference measure, we observe
\[
f_1^{(i)}
=
T\bigl(f_0^{(i)}\bigr)+\xi^{(i)},
\qquad i=1,\ldots,N,
\]
where the observation noise has conditional mean zero. The pairing gives the input--output correspondence, while the map must be estimated from finitely many noisy observations.

In the standard statistical optimal-transport setting, independent samples from the source and target distributions are observed without pairing. If the map is known, estimating a finite-dimensional representation is an approximation problem. With noiseless paired observations, the map is observed directly at the sampled inputs. We consider the case of noisy paired observations, where both finite sample size and observation noise contribute to the estimation error.

Since the observations are paired, $\mathbb E\!\left[f_1^{(i)}\mid f_0^{(i)}\right]
=T\bigl(f_0^{(i)}\bigr)$ under the conditional mean-zero assumption. At the population level, $T$ pushes the source measure $\pi_0$ onto a target measure $\pi_1$. For Cameron--Martin cost, the infinite-dimensional analogue of Brenier's construction gives
\[
T=I+\nabla_{\mathcal H}\phi,
\]
where $\mathcal H$ is the Cameron--Martin space
\citep{brenier1991polar,FeyelUstunel2004,Gonzalez2023}. The gradient representation determines the map classes and their finite-dimensional approximations used in our analysis.

We approximate the potential $\phi$ by cylindrical functions of finitely many Cameron--Martin coordinates and estimate the resulting gradient map by empirical risk minimization. The number of retained coordinates increases with the sample size. The estimation error therefore depends on both cylindrical approximation and statistical estimation. This approximation--estimation tradeoff is a main feature of the infinite-dimensional problem.

A further difficulty comes from the dependence on discarded coordinates. Decay of the high-frequency output coordinates of $T-I$ alone does not guarantee accurate cylindrical approximation, because a retained output coordinate may still depend on omitted input coordinates. We therefore separate output regularity from regularity with respect to discarded inputs. For general Sobolev potentials, weighted derivative regularity and a conditional Gaussian Poincar\'e inequality control this dependence. For potentials of bounded chaos degree, the degree restriction provides another way to control the approximation error. The observation noise creates an additional difficulty: it may belong to the ambient Hilbert space while having infinite Cameron--Martin norm. We use a criterion based on finitely many observed coordinates so that the regression loss remains finite.

Our main objective is to give conditions under which finite-dimensional gradient estimators recover the infinite-dimensional map and to determine the resulting convergence rates. For a diagonal Gaussian class and a nonlinear block-interaction class, we establish matching upper and lower bounds of order $N^{-s/(2s+1)}$ in expected norm, where $s$ measures weighted coordinate regularity. The nonlinear class contains non-Gaussian targets and interactions that cannot be removed by a fixed orthogonal change of coordinates. For bounded-chaos potentials with interaction order $q$, an orthogonal Hermite sieve has dimension of order $d^q$ and gives the rate $N^{-s/(q+2s)}$. The additive case is minimax. We also establish a local nonasymptotic oracle inequality that separates approximation error, estimation error, and the effect of parametrization conditioning. These results describe how coordinate regularity and interaction structure affect the statistical cost of estimating the map.

We study three model classes. The first is a conjugate Gaussian inverse problem for which the transport coefficients are available in closed form. This allows direct verification of the regularity conditions, approximation error, and minimax rate. The second is a nonlinear block-interaction class for which we obtain matching upper and lower bounds. The third is a continuous two-groups model in which each coordinate follows either a Gaussian distribution or a heavier-tailed Laplace distribution with an unknown mixing probability. Such mixtures are used in empirical-Bayes multiple testing and sparse modeling \citep{efron2008microarrays,georgemcculloch1993}. In the paired observation model, the mixing probabilities determine the coordinate transport maps. We obtain a direct prediction-risk bound using a Bernstein condition for the excess loss.

Statistical estimation of optimal transport maps has been studied mainly in finite-dimensional spaces. Under smoothness conditions on the transport potential, \citet{hutterrigollet2021minimax} establish minimax rates for map estimation. Plug-in, sieve, and related estimators have been studied by \citet{manole2024plugin,ding2024statistical,cazelles2026statistical}. \citet{divolnilesweedpooladian2022} obtain rates based on the metric entropy of the class of Brenier potentials. Entropic estimators with finite-sample guarantees are studied by \citet{pooladiannilesweed2021}. Stability results provide another way to control map error through perturbations of the underlying measures \citep{deb2021rates}. For semi-dual estimation, \citet{lidingxueli2025stability} obtain an oracle inequality that separates statistical error, sieve bias, and approximation error.

Transport estimation on infinite-dimensional spaces has received more recent attention. \citet{ponnopratimaizumi2025} establish minimax rates under coordinatewise smoothness conditions when the source and target measures are supported on an infinite-dimensional space. Conditional triangular transport has been developed on separable function spaces, including transport from Gaussian priors to posterior distributions \citep{hosseinihsutaghvaei2025}. Other work studies transport in functional data analysis \citep{zhu2024functional} and transport methods for Gaussian processes and covariance operators \citep{masarotto2018procrustes,panaretoszemel2020invitation,minh2022finite}. In our setting, the data consist of noisy paired evaluations of a single population transport map.

Our problem is also related to regression models involving transport maps. In distribution-on-distribution regression, the observations are pairs of probability measures and transport maps are used to describe conditional Fr\'echet means \citep{ghodratipanaretos2022,ghodratipanaretos2023}. Related models use the tangent geometry of Wasserstein space \citep{chenlinmueller2023}, while supervised transport methods consider labeled pairs of measures \citep{bunnekrausecuturi2022}. Here, the observations are individual functions and their noisy images, and estimation is based on a squared-error criterion in retained Cameron--Martin coordinates. The gradient representation determines the approximation classes, while the parametrization affects their expressiveness and conditioning \citep{baptistamarzoukzahm2024}.

The rest of this article is organized as follows. Section~\ref{sec:setup} introduces the Gaussian reference, the observation model, and the estimator. Sections~\ref{sec:approx-stoch}--\ref{sec:rates-minimax} develop the approximation analysis, the local oracle inequality, and the minimax lower bound. Section~\ref{sec:input} gives regularity conditions for cylindrical approximation and studies structured Hermite sieves. Section~\ref{sec:examples} presents the three model classes. Numerical results are reported in Section~\ref{sec:numerical-illustration}, and Section~\ref{sec:discussion} gives some final remarks. Proofs and additional calculations are collected in the Supplementary Material.


\section{Setup and cylindrical-sieve estimators}
\label{sec:setup}
 
This section defines the Gaussian reference, the paired observation model, and the cylindrical sieve estimator. Section~\ref{subsec:prelim} introduces the Gaussian measure, its Cameron--Martin space, and the associated gradient. Section~\ref{subsec:rep} states the assumptions, the transport representation, and the paired sampling model. Section~\ref{subsec:sieve} defines the finite-dimensional sieve and the empirical risk minimizer. Sections~\ref{sec:approx-stoch} and~\ref{sec:rates-minimax} study its approximation and estimation errors.

\subsection{Gaussian reference and Cameron--Martin coordinates}
\label{subsec:prelim}

\textbf{Notation.} The Gaussian-space objects $\mathcal F$, $\gamma$, $C_\gamma$, $\mathcal H$, $\{e_k\}$, $\{\hat e_k\}$, and $\nabla_{\mathcal H}$ are defined below when first used. We write $\pi_0$ and $\pi_1$ for the source and target measures and measure map error in $L^2(\pi_0;\mathcal H)$, with
\[
\|S\|_{L^2(\pi_0;\mathcal H)}^2:=\int_{\mathcal F}|S(f)|_{\mathcal H}^2\,d\pi_0(f).
\]
The smoothness class $\mathcal W^s$ is defined in Definition~\ref{def:reg-class}. We use $a_N\asymp b_N$ for two sequences bounded above and below by constant multiples of one another, and $O_p(\cdot)$ and $o_p(\cdot)$ for the usual stochastic orders.

Let $(\mathcal{F},\langle\cdot,\cdot\rangle_{\mathcal{F}})$ be a real separable Hilbert space. Let $\mathcal{B}(\mathcal{F})$ be the Borel $\sigma$-algebra generated by the open subsets of the Hilbert space $\mathcal{F}$, and let $\gamma$ be a centered, nondegenerate Gaussian measure defined on $\mathcal{B}(\mathcal{F})$. Its covariance operator $C_\gamma:\mathcal{F}\to\mathcal{F}$,
\[
\langle C_\gamma x,\,y\rangle_{\mathcal F}
=\int_{\mathcal F}\langle x,f\rangle_{\mathcal F}\,
\langle y,f\rangle_{\mathcal F}\,\d\gamma(f),
\qquad x,y\in\mathcal F,
\]
is self-adjoint, positive, and trace class; nondegeneracy means $C_\gamma$ is injective, so $\gamma$ assigns positive mass to every nonempty open set \citep{bogachev1998gaussian}.

The Cameron--Martin space associated with $\gamma$ is
\[
\mathcal{H}:=C_\gamma^{1/2}(\mathcal{F}),
\qquad
\langle h,k\rangle_{\mathcal{H}}
:=\big\langle C_\gamma^{-1/2}h,\,C_\gamma^{-1/2}k\big\rangle_{\mathcal{F}},
\]
a Hilbert space that embeds continuously and densely into $\mathcal{F}$; since $C_\gamma$ is a trace class operator, the embedding $\mathcal{H}\hookrightarrow\mathcal{F}$ is compact. The space $\mathcal{H}$ is the set of admissible shift directions. The translation of every $\gamma$-null set by $h$ is a $\gamma$-null set if and only if $h\in\mathcal{H}$. That is, $\gamma(\cdot-h)\ll\gamma(\cdot)$ if and only if $h\in\mathcal{H}$, by the Cameron--Martin theorem \citep{bogachev1998gaussian}.

For $h\in\mathcal{H}$, let $\hat h\in L^2(\gamma)$ denote the associated
first-chaos, or Paley--Wiener functional. That is, the map
$h\mapsto\hat h$ is a linear isometry from $\mathcal{H}$ into
$L^2(\gamma)$:
\[
    \mathbb{E}_\gamma[\hat h\,\hat k]
    = \langle h,k\rangle_{\mathcal{H}},
    \qquad h,k\in\mathcal{H}.
\]
Thus, under $\gamma$,
\[
    \hat h\sim N\!\left(0,|h|_{\mathcal{H}}^2\right).
\]
In particular, if $\{e_k\}_{k\ge1}$ is an orthonormal basis of
$\mathcal{H}$, then $\{\hat e_k\}_{k\ge1}$ are i.i.d.\ $N(0,1)$
under $\gamma$.

Being self-adjoint, positive, and trace class, $C_\gamma$ has an orthonormal eigenbasis
$\{\varphi_k\}_{k\ge1}$ of $\mathcal F$ with $C_\gamma\varphi_k=\lambda_k\varphi_k$,
$\lambda_1\ge\lambda_2\ge\cdots>0$ and $\sum_k\lambda_k<\infty$. We fix once and for all
the orthonormal basis of $\mathcal H$ it induces,
\[
  e_k:=C_\gamma^{1/2}\varphi_k=\sqrt{\lambda_k}\,\varphi_k,\qquad k\ge1,
\]
whose first-chaos functionals are
$\hat e_k=\langle\,\cdot\,,\varphi_k\rangle_{\mathcal F}/\sqrt{\lambda_k}$, i.i.d.\ $N(0,1)$
under $\gamma$. The index $k$ orders the Cameron--Martin directions by decreasing
eigenvalue $\lambda_k$. This is the basis used throughout, for the sieve features of
Definition~\ref{def:cylindrical}, the regularity class of Definition~\ref{def:reg-class}, and
the spectral-decay analysis of Section~\ref{subsec:spectral}.

For a smooth cylindrical function
\[
  \phi(f)=\psi\!\big(\hat h_1(f),\dots,\hat h_m(f)\big),
  \qquad \psi\in C_b^1(\mathbb R^m),\ h_1,\dots,h_m\in \mathcal{H},
\]
the Cameron--Martin, or Malliavin, gradient $\nabla_{\mathcal H}\phi(f)\in \mathcal{H}$ is the unique element representing the directional derivatives along $\mathcal{H}$,
\[
  \langle\nabla_{\mathcal{H}}\phi(f),h\rangle_{\mathcal{H}}=\lim_{t\to0}\frac{\phi(f+th)-\phi(f)}{t},\qquad h\in \mathcal{H}.
\]
Since $\nabla_{\mathcal H}\hat h=h$ for every $h\in \mathcal{H}$, the chain rule gives the explicit form
\begin{equation}
  \nabla_{\mathcal H}\phi(f)=\sum_{j=1}^m
  \partial_j\psi\!\big(\hat h_1(f),\dots,\hat h_m(f)\big)\,h_j .
  \label{eq:chain}
\end{equation}
The arguments in \eqref{eq:chain} are the first-chaos functionals $\hat h_j$, not the ambient inner products $\langle f,h_j\rangle_{\mathcal F}$, because $\nabla_{\mathcal H}\langle\,\cdot\,,h\rangle_{\mathcal F}=C_\gamma h\neq h$ in general. The symbol $\nabla_{\mathcal H}$ denotes this gradient throughout.

\subsection{Population transport map and paired observation model}
\label{subsec:rep}
 
Let $\mathcal{P}_2(\mathcal{F})$ be the Borel probability measures on $\mathcal F$ with finite second moment, and set
\[
\mathcal{P}_{2,\gamma}(\mathcal{F})
:=\{\mu\in\mathcal{P}_2(\mathcal{F}):\mu\ll\gamma\}.
\]
We transport between two elements $\pi_0,\pi_1\in\mathcal{P}_{2,\gamma}(\mathcal{F})$ with strictly positive densities,
\[
\pi_0=f_0\,\gamma,\qquad \pi_1=f_1\,\gamma.
\]
 
\begin{description}
\item[\textnormal{(A1)} Finite second moments.]\;
$\displaystyle\int_{\mathcal F}\|f\|_{\mathcal F}^2\,\d\pi_j(f)<\infty$ for $j=0,1$; equivalently $\pi_0,\pi_1\in\mathcal P_2(\mathcal F)$.

\item[\textnormal{(A2)} Finite Cameron--Martin transport cost.]\;
\[
  W_{2,\mathcal H}(\pi_0,\pi_1)^2
  := \inf_{S:\,S_\#\pi_0=\pi_1}\int_{\mathcal F} |S(f)-f|_{\mathcal H}^2\,\d\pi_0(f) < \infty .
\]

\item[\textnormal{(A3)} Absolute continuity.]\;
The densities satisfy $f_0,f_1>0$ $\gamma$-a.e.
\end{description}

\begin{definition}
\label{def:Hconvex}
A measurable $\phi:\mathcal F\to\mathbb R\cup\{+\infty\}$ is $\mathcal H$-convex,
or equivalently $1$-convex, if for $\gamma$-a.e.\ $f\in\mathcal F$ the map
\[
  \mathcal H\ni h\longmapsto \phi(f+h)+\tfrac12|h|_{\mathcal H}^2
\]
is convex. Convexity is imposed on Cameron--Martin slices; the population map is
written as the identity plus its Cameron--Martin displacement.
\end{definition}

The population map is fixed by the following structural hypothesis, which we impose
alongside \textnormal{(A1)}--\textnormal{(A3)}.

\begin{ass}
\label{ass:rep}
The Monge problem
\[
  \inf_{S:\,S_\#\pi_0=\pi_1}\int_{\mathcal F}|S(f)-f|_{\mathcal H}^2\,\d\pi_0(f)
\]
admits a minimizer $T$, unique $\pi_0$-almost surely, and there exists an
$\mathcal H$-convex potential $\phi$, as in Definition~\ref{def:Hconvex},
$\mathcal H$-differentiable $\pi_0$-a.e., such that
\[
  T(f)=f+\nabla_{\mathcal H}\phi(f)\quad\text{for }\pi_0\text{-a.e.\ }f\in\mathcal F,
  \qquad \nabla_{\mathcal H}\phi\in L^2(\pi_0;\mathcal H).
\]
\end{ass}

\begin{rem}
\label{rem:rep-scope}
For $\pi_0=\gamma$, the gradient representation follows from Gaussian-space optimal
transport theory \citep{FeyelUstunel2004}. We retain it as a structural assumption for
general sources. All three explicit classes below have Gaussian source and verify their
population maps directly.
\end{rem}

Since $T-I=\nabla_{\mathcal H}\phi$ is $\mathcal H$-valued, the Cameron--Martin cost equals
\[
  \int_{\mathcal F} |T(f)-f|_{\mathcal H}^2\,d\pi_0(f)=\|\nabla_{\mathcal H}\phi\|_{L^2(\pi_0;{\mathcal H})}^2,
\]
so the cost is finite if and only if $\nabla_{\mathcal H}\phi\in L^2(\pi_0;{\mathcal H})$. Uniqueness is meant $\pi_0$-a.s., since potentials differing on a $\pi_0$-null set induce the same map. We call $T$ the population transport map in what follows.

We observe paired data generated by the population map. Each source draw is observed together with a noisy image under $T$, which gives the squared-error criterion used below. The following assumption specifies the sampling and noise model.

\begin{ass}
\label{ass:data}
The data are i.i.d.\ pairs $\{(f_0^{(i)},f_1^{(i)})\}_{i=1}^N$ with
\[
  f_0^{(i)}\stackrel{\textnormal{i.i.d.}}{\sim}\pi_0,\qquad
  f_1^{(i)}=T(f_0^{(i)})+\xi^{(i)},\qquad i=1,\dots,N,
\]
where $T$ is the population map of Assumption~\ref{ass:rep}. The noise $\xi^{(i)}$ is a
 $\mathcal F$-valued Gaussian random element, $\xi^{(i)}\sim N(0,\sigma^2C_\gamma)$
for the same trace-class covariance operator $C_\gamma$ as the reference $\gamma$. The noise intensity is $\sigma^2>0$, and $\xi^{(i)}\in\mathcal F$ almost surely. The noise is independent across $i$ and of the design $\{f_0^{(i)}\}$, and its
Cameron--Martin coordinates are
\begin{equation}
\label{eq:noise-coords}
  \xi^{(i)}_k:=\frac{\big\langle\xi^{(i)},\varphi_k\big\rangle_{\mathcal F}}{\sqrt{\lambda_k}}
  \ \stackrel{\textnormal{i.i.d.}}{\sim}\ N(0,\sigma^2),\qquad k\ge1,
\end{equation}
exactly as for a draw from $\gamma$ itself rescaled by $\sigma$. These are ordinary real
random variables, defined without reference to the Cameron--Martin norm of $\xi^{(i)}$,
which is almost surely infinite; $f_1^{(i)}$ is an element of $\mathcal F$ and not of
$T(f_0^{(i)})+\mathcal H$.
\end{ass}
Here $\pi_1$ denotes the law of the latent noise-free image $T(f_0)$. When $\sigma>0$, the marginal distribution of the observed response $f_1=T(f_0)+\xi$ generally differs from $\pi_1$.
 
This is the Gaussian sequence, or white-noise, model in the Cameron--Martin
coordinates, the canonical setting in which the rate of Section~\ref{sec:minimax} is
sharp. Since $\xi\notin\mathcal H$ almost surely, the quantities
$\|T_{\theta,N}(f_0)-f_1\|_{\mathcal H}^2$ are almost surely $+\infty$, so the criterion
cannot be defined as a difference of squared prediction errors. We define it instead by a
finite expression in the observed coordinates.

Fix a reference $\theta_0\in\Theta_N$, write
$Z_i:=\Phi_N(f_0^{(i)})\in\mathbb R^{d_N}$ for the feature vector of
Definition~\ref{def:cylindrical}, $b_\theta(z):=\nabla g_\theta(z)\in\mathbb R^{d_N}$ for
the sieve displacement in coordinates, and
\begin{equation}
\label{eq:Dik}
  D_{ik}:=\frac{\big\langle f_1^{(i)}-f_0^{(i)},\varphi_k\big\rangle_{\mathcal F}}{\sqrt{\lambda_k}},
  \qquad k=1,\dots,d_N,
\end{equation}
for the observed displacement in the same coordinates, a finite random vector because
$f_0^{(i)},f_1^{(i)}\in\mathcal F$ and $\lambda_k>0$. The estimator minimizes
\begin{equation}
\label{eq:def-Rhat-circ}
  \widehat R_N^\circ(\theta)
  :=\frac1N\sum_{i=1}^N\sum_{k=1}^{d_N}
  \Big\{b_{\theta,k}(Z_i)^2-b_{\theta_0,k}(Z_i)^2
  -2D_{ik}\big[b_{\theta,k}(Z_i)-b_{\theta_0,k}(Z_i)\big]\Big\},
\end{equation}
a real-valued, data-measurable function of $\theta\in\Theta_N$, defined without any
appeal to an infinite quantity and independent of the choice of $\theta_0$ beyond a
$\theta$-independent additive shift. The empirical risk minimizer is
$\hat\theta_N\in\arg\min_{\theta\in\Theta_N}\widehat R_N^\circ(\theta)$, with
$\hat T_N:=T_{\hat\theta_N,N}$. All empirical gradients and Hessians below refer to
\eqref{eq:def-Rhat-circ}.

At the population level, define the excess risk
\begin{equation}
\label{eq:def-RN}
  R_N(\theta):=\|T_{\theta,N}-T\|_{L^2(\pi_0;\mathcal H)}^2 ,
\end{equation}
so that $T$ is the population minimizer over all maps and the sieve oracle is
$\theta_N^\ast\in\arg\min_{\theta\in\Theta_N}R_N(\theta)$.
Section~\ref{sup:foundations} shows that
$\mathbb E\big[\widehat R_N^\circ(\theta)\big]=R_N(\theta)-R_N(\theta_0)$ for every
$\theta,\theta_0\in\Theta_N$, so \eqref{eq:def-Rhat-circ} is an unbiased estimate of the
excess-risk difference and $\theta_N^*$ is its population minimizer. Performance is
measured by the prediction error $\|\hat T_N-T\|_{L^2(\pi_0;\mathcal H)}$.

When only unpaired ensembles $\{f_0^{(i)}\}_{i=1}^N$ and $\{f_1^{(j)}\}_{j=1}^N$ are observed, the squared-error risk above is not computable, and one instead solves a semi-dual, or Kantorovich, problem or uses plug-in and barycentric-projection estimators. These extensions are deferred to Section~\ref{sec:discussion}.


Computable classes are built by approximating the potential $\phi$ with a function of finitely many Cameron--Martin coordinates and taking its $\mathcal H$-gradient as the candidate map. The resulting cylindrical-sieve estimators are finite-dimensional and retain the Cameron--Martin gradient structure $T_{\theta,N}=I+\nabla_{\mathcal H}\phi_{\theta,N}$ of the population map. As Remark~\ref{rem:convexity} makes precise, this structural form is preserved without imposing $\mathcal H$-convexity on $g_\theta$.
 
\subsection{Cylindrical sieve and empirical risk minimization}
\label{subsec:sieve}

The sieve is built in the Cameron--Martin coordinates. A candidate potential depends on the first $d_N$ first-chaos coordinates, so the eigen-ordering of $C_\gamma$ determines which coordinates are retained. The family $\{g_\theta\}$ determines the form and dimension of the finite approximation. We use the method of sieves, with the sieve dimension increasing with the sample size \citep{shenwong1994}.

Let $\{e_k\}_{k\ge1}$ be the orthonormal basis of the Cameron--Martin space
$\mathcal H$ fixed in Section~\ref{subsec:prelim}, and let $d_N\to\infty$ be a sieve
dimension.

\begin{definition}
\label{def:cylindrical}
The feature map of order $d_N$ is the
vector of first-chaos coordinates
\[
  \Phi_N(f):=\big(\hat e_1(f),\dots,\hat e_{d_N}(f)\big)\in\mathbb R^{d_N},
\]
which under $\gamma$ has i.i.d.\ $N(0,1)$ entries. Let
$\{g_\theta:\theta\in\Theta_N\subset\mathbb R^{p_N}\}$ be a parametric class of
continuously differentiable functions $g_\theta:\mathbb R^{d_N}\to\mathbb R$ such that
$g_\theta$ and $\nabla g_\theta$ have at most polynomial growth, which ensures
$\phi_{\theta,N}\in\mathbb D^{1,2}(\gamma)$, and such that
$\nabla g_\theta(\Phi_N)\in L^2(\pi_0)$ for every $\theta\in\Theta_N$. The second
requirement is separate from the first when the source is not the reference. Under a general $\pi_0\ll\gamma$, \textnormal{(A1)} assumes only second moments.
Square integrability against $\pi_0$ is stated separately. It holds when
$\pi_0=\gamma$, as in the three explicit classes, and whenever $\nabla g_\theta$ is bounded. The associated cylindrical potential and
cylindrical-sieve map are
\[
  \phi_{\theta,N}(f):=g_\theta(\Phi_N(f)),\qquad
  T_{\theta,N}(f):=f+\nabla_{\mathcal H}\phi_{\theta,N}(f)
  =f+\sum_{k=1}^{d_N}\partial_k g_\theta(\Phi_N(f))\,e_k,
\]
where the last identity is exact by \eqref{eq:chain} with $h_k=e_k$.
\end{definition}

Two indices govern the construction, the sieve dimension $d_N=\dim\Phi_N$, which is the
number of Cameron--Martin coordinates the map may act on, and the parameter dimension
$p_N=\dim\theta$, which is the complexity of $g_\theta$. The first controls the
approximation error and the second the stochastic error, as made precise in
Sections~\ref{sec:approx-section}--\ref{sec:rates-section}.

By construction $T_{\theta,N}-I=\nabla_{\mathcal H}\phi_{\theta,N}
=\sum_{k=1}^{d_N}\partial_k g_\theta(\Phi_N)\,e_k$ is $\mathcal H$-valued, so each sieve
map is an admissible Cameron--Martin perturbation of the identity, of the same form
$I+\nabla_{\mathcal H}\phi$ as the population map in Assumption~\ref{ass:rep}. Since
$\{e_k\}$ is $\mathcal H$-orthonormal,
\[
  \|T_{\theta,N}-I\|_{L^2(\pi_0;\mathcal H)}^2
  =\mathbb E_{\pi_0}\!\left\|\nabla g_\theta(\Phi_N)\right\|_2^2,
\]
which is finite whenever $\nabla g_\theta(\Phi_N)\in L^2(\pi_0)$.

Varying $\theta$ over $\Theta_N$ produces, for each $N$, the candidate class
\[
\mathcal T_N:=\{\,T_{\theta,N}:\theta\in\Theta_N\,\},
\]
a sieve of Cameron--Martin gradient maps, every element of which is the $\mathcal H$-gradient of a cylindrical potential, of the same structural form $I+\nabla_{\mathcal H}\phi$ as the true map. Different choices of $\{g_\theta\}$ give different map classes. Quadratic potentials give linear diagonal maps, monotone coordinate potentials give rearrangement maps, and deep-network potentials give flexible nonlinear maps. These choices are used in Sections~\ref{subsec:gaussian} and~\ref{subsec:bayesian-example}.
 
\begin{rem}
\label{rem:convexity}
Assumption~\ref{ass:rep} makes the population map $T$ the $\mathcal H$-gradient of an $\mathcal H$-convex potential. The sieve keeps the gradient form $I+\nabla_{\mathcal H}\phi_{\theta,N}$ but does not impose $\mathcal H$-convexity on $g_\theta$. Thus a fitted sieve map is a gradient map but may not itself be an optimal transport map. Its population risk is $R_N(\theta)=\|T_{\theta,N}-T\|_{L^2(\pi_0;\mathcal H)}^2$, so the sieve oracle is the $L^2(\pi_0;\mathcal H)$ projection of $T$ onto the candidate class. Imposing $\mathcal H$-convexity would instead give a constrained approximation and optimization problem.
\end{rem}
 
The best the class can do at the population level is the sieve oracle
\[
\theta_N^*\in\arg\min_{\theta\in\Theta_N}R_N(\theta),
\qquad
T_N^*:=T_{\theta_N^*,N},
\qquad
R_N(\theta)=\|T_{\theta,N}-T\|_{L^2(\pi_0;\mathcal H)}^2,
\]
with $R_N$ the parametric population risk of Section~\ref{subsec:rep}. Its distance to the truth, $\|T_N^*-T\|_{L^2(\pi_0;\mathcal H)}$, is the irreducible sieve approximation error; quantifying it in terms of $d_N$ and the regularity of $\phi$ is the subject of Section~\ref{sec:approx-section}.

Given the paired sample $\{(f_0^{(i)},f_1^{(i)})\}_{i=1}^N$ of
Section~\ref{subsec:rep}, the estimator is the minimizer of the finite contrast
\eqref{eq:def-Rhat-circ} over the sieve,
$\widehat\theta_N\in\arg\min_{\theta\in\Theta_N}\widehat R_N^\circ(\theta)$, with
$\widehat T_N:=T_{\widehat\theta_N,N}$. Written out, coordinate $k$ contributes
$b_{\theta,k}(Z_i)^2-2D_{ik}b_{\theta,k}(Z_i)$ up to a $\theta$-free term, so
\eqref{eq:def-Rhat-circ} is an ordinary least-squares criterion for the sieve
displacement $b_\theta$ against the observed displacement $D$ in the first $d_N$
Cameron--Martin coordinates.
 
For empirical measures with equally many particles, the pushforward constraint $S_\#\hat\pi_0=\hat\pi_1$ reduces to an index permutation, $S(f_0^{(i)})=f_1^{(\sigma(i))}$ for some $\sigma\in S_N$. The squared-error formulation above relaxes this hard assignment to a regression against the paired targets $f_1^{(i)}=T(f_0^{(i)})$.
 
Performance is measured by the prediction error $\|\widehat T_N-T\|_{L^2(\pi_0;\mathcal H)}$, which
the triangle inequality splits through the sieve oracle $T_N^*$ defined above into
\begin{equation}
\label{eq:s3-decomp}
\|\widehat T_N-T\|_{L^2(\pi_0;\mathcal H)}
\le
\|\widehat T_N-T_N^*\|_{L^2(\pi_0;\mathcal H)}
+
\|T_N^*-T\|_{L^2(\pi_0;\mathcal H)}.
\end{equation}
The decomposition itself is deterministic, holding for every realization of the sample.
The approximation term is controlled by the coordinate regularity of $T$, while control
of the stochastic term requires the local regularity and localization conditions
introduced below.

The terms in \eqref{eq:s3-decomp} are governed by the complexity parameters of the sieve. The parameter dimension $p_N$ and the local condition number
\[
\kappa_N:=L_N/\mu_N
\]
control the stochastic term, where $L_N$ is the local Lipschitz constant of $\theta\mapsto T_{\theta,N}$ and $\mu_N$ the local strong-convexity parameter of $R_N$; the sieve dimension $d_N$ controls the approximation term through the rate $d_N^{-s}$. The two terms are analyzed separately in the sequel, the approximation term in Section~\ref{sec:approx-section}, the stochastic term and the resulting oracle inequality in Section~\ref{sec:stoch-section}, and the implied nonasymptotic rates and oracle-balanced sieve dimension in Section~\ref{sec:rates-section}.


\section{Approximation and stochastic analysis}
\label{sec:approx-stoch}

The estimation error has approximation and stochastic components. Section~\ref{sec:approx-section} introduces a coordinate regularity class and gives the approximation rate $d_N^{-s}$ under an expressivity condition. Section~\ref{sec:stoch-section} bounds the stochastic term and gives the oracle inequality.

\subsection{Approximation properties of the cylindrical sieve}
\label{sec:approx-section}
 
The deterministic term $\|T_N^*-T\|_{L^2(\pi_0;\mathcal H)}$ of the decomposition
\eqref{eq:s3-decomp} is controlled by two effects. The first one belongs to the target, namely the decay of the
displacement of $T$ along the Cameron--Martin coordinates. It fixes the cost of
truncating at $d_N$. The second one belongs to the estimator, namely the richness of
$\{g_\theta\}$. It fixes how closely the sieve represents what is kept.
 
\subsubsection{Regularity class for the transport map}
\label{subsec:regularity}

Recall $T=I+\nabla_{\mathcal H}\phi$ with $\nabla_{\mathcal H}\phi\in L^2(\pi_0;\mathcal H)$.
Expand the displacement in the canonical basis $\{e_k\}$ of
Section~\ref{subsec:prelim}. These are the $C_\gamma$-eigen directions, ordered by decreasing eigenvalue $\lambda_1\ge\lambda_2\ge\cdots$:
\[
  \nabla_{\mathcal H}\phi(f)=\sum_{k\ge1}u_k(f)\,e_k,
  \qquad u_k(f):=\langle\nabla_{\mathcal H}\phi(f),e_k\rangle_{\mathcal H}\in L^2(\pi_0).
\]
Write $T^{(d)}:=I+\sum_{k\le d}u_k\,e_k$ for the $d$-coordinate truncation. It acts on the $d$ leading Cameron--Martin directions and as the identity on the rest. Then
\[
  \|T-T^{(d)}\|_{L^2(\pi_0;\mathcal H)}^2=\sum_{k>d}\|u_k\|_{L^2(\pi_0)}^2 .
\]
The truncation acts on the output directions. Each retained coefficient $u_k$ is still a function of the full input $f$, so $T^{(d)}$ is generally not cylindrical and may fail to be a gradient map. The gap between $T^{(d)}$ and the sieve depends on both the richness of $\{g_\theta\}$ and the dependence of $u_1,\dots,u_d$ on coordinates beyond $d$. Assumption~\ref{ass:expressivity} controls this gap. Section~\ref{sec:input} gives verifiable conditions for the Gaussian-source case.

\begin{definition}
\label{def:reg-class}
For $s>0$, the transport map $T=I+\nabla_{\mathcal H}\phi$ belongs to $\mathcal W^s$ if
\[
  \|T\|_{\mathcal W^s}^2:=\sum_{k\ge1}k^{2s}\,\|u_k\|_{L^2(\pi_0)}^2<\infty .
\]
\end{definition}

Membership in $\mathcal W^s$ gives the truncation bound Bounding
$k^{-2s}\le d^{-2s}$ on the tail $k>d$ gives, for every $d\ge1$,
\begin{equation}
\label{eq:truncation}
\|T-T^{(d)}\|_{L^2(\pi_0;\mathcal H)}\le\|T\|_{\mathcal W^s}\,d^{-s},
\qquad T\in\mathcal W^s .
\end{equation}

Because the index $k$ tracks the eigen-ordering of $C_\gamma$, membership in $\mathcal W^s$ asks
the coefficient sequence $(\|u_k\|_{L^2(\pi_0)})_k$ to be square-summable against the
weights $k^{2s}$, i.e.\ the action of $T$ to decay along the small-eigenvalue,
high-frequency, Cameron--Martin directions. The exponent $s$ is governed jointly by the
smoothness of $\phi$ and the spectral decay of $C_\gamma$. When $\lambda_k\asymp k^{-\beta}$,
the index weight $k^{2s}$ is equivalent to the spectral weight $\lambda_k^{-2s/\beta}$, so
$\mathcal W^s$ is a $C_\gamma$-weighted Sobolev scale. This is made explicit in
Section~\ref{subsec:spectral} and in the Gaussian benchmark of
Section~\ref{subsec:gaussian}, where $s$ is read off in closed form from the tail
spectrum.

\subsubsection{Sieve approximation error}
\label{subsec:approx-error}
 
The sieve must approximate the potential associated with $T^{(d_N)}$. We state this requirement as an expressivity condition.
 
\begin{ass}
\label{ass:expressivity}
There is a constant $C'>0$ such that, for every $N$,
\[
\inf_{\theta\in\Theta_N}\big\|T_{\theta,N}-T^{(d_N)}\big\|_{L^2(\pi_0;\mathcal H)}\le C'\,d_N^{-s}.
\]
\end{ass}
 
\begin{theorem}
\label{thm:approx}
If $T\in\mathcal W^s$, per Definition~\ref{def:reg-class}, and
Assumption~\ref{ass:expressivity} holds, then the sieve oracle satisfies
\[
\|T_N^*-T\|_{L^2(\pi_0;\mathcal H)}\le C_a\,d_N^{-s},
\qquad
C_a:=C'+\|T\|_{\mathcal W^s}.
\]
\end{theorem}
 
\begin{proof}
By the definition of the sieve oracle and the triangle inequality,
\[
\|T_N^*-T\|\le \inf_{\theta\in\Theta_N}\|T_{\theta,N}-T^{(d_N)}\|
 +\|T^{(d_N)}-T\|\le (C'+\|T\|_{\mathcal W^s})d_N^{-s}.
\]
All norms in the first inequality are in $L^2(\pi_0;\mathcal H)$.
\end{proof}

The expressivity assumption is additional to membership in $\mathcal W^s$;
Section~\ref{app:expressivity} gives a counterexample to uniform approximation based on
output regularity alone. It holds with zero approximation error for the retained
coordinates in the Gaussian and mixture classes, and for complete retained blocks in
Section~\ref{subsec:block}.


\subsection{Stochastic analysis of the estimator}
\label{sec:stoch-section}
 
We bound the stochastic term $\|\widehat T_N-T_N^*\|_{L^2(\pi_0;\mathcal H)}$ in \eqref{eq:s3-decomp} using local curvature, gradient concentration, and the Lipschitz behavior of the parametrization near $\theta_N^*$. Throughout, $\mathcal N_N$ denotes a neighborhood of the sieve oracle $\theta_N^*$.
 
\subsubsection{Local strong convexity and quadratic growth}
\label{subsec:lsc}
 
\begin{rem}
\label{rem:parameter-space}
For the local oracle theorem, $\Theta_N\subset\mathbb R^{p_N}$ is compact with nonempty
interior, the criteria are continuous, and the population oracle is interior.
Compactness and continuity give existence of minimizers. Uniform concentration and
localization are separate assumptions. The grid theorem below has its own parameter
set and does not use these interiority conventions.
\end{rem}
 
\begin{ass}
\label{ass:invertible}
The oracle $\theta_N^*$ is interior to $\Theta_N$, the population risk $R_N$ is twice
continuously differentiable in a neighborhood of $\theta_N^*$, and the Hessian
$H_{R_N}(\theta_N^*)$ is invertible; see, e.g., \citet[Theorem~5.23]{Vaart1998}.
\end{ass}
 
At an interior minimizer the second-order necessary condition gives
$H_{R_N}(\theta_N^*)\succeq0$, which invertibility upgrades to
$H_{R_N}(\theta_N^*)\succ0$. The smallest eigenvalue of a symmetric matrix is a
continuous function of its entries, so there are a closed Euclidean ball
$\mathcal N_N=\{\theta:\|\theta-\theta_N^*\|\le r_N\}\subset\mathrm{int}(\Theta_N)$
and a constant $\mu_N>0$ with
$v^\top H_{R_N}(\theta)v\ge\mu_N\|v\|^2$ for every $\theta\in\mathcal N_N$ and every
$v\in\mathbb R^{p_N}$. Since $\nabla R_N(\theta_N^*)=0$ at the interior optimum, a
second-order Taylor expansion then yields the quadratic growth
\begin{equation}
\label{eq:quad-growth}
R_N(\theta)-R_N(\theta_N^*)\ge\tfrac{\mu_N}{2}\,\|\theta-\theta_N^*\|^2,
\qquad \theta\in\mathcal N_N,
\end{equation}
so $\mu_N$ is the local strong-convexity parameter of $R_N$. We also record the
Lipschitz behavior of the parametrization, which converts parameter error into map
error.
 
\begin{ass}
\label{ass:lipschitz-map}
There exists $L_N>0$ such that
\[
\|T_{\theta,N}-T_{\theta',N}\|_{L^2(\pi_0;\mathcal H)}\le L_N\,\|\theta-\theta'\|,
\qquad \forall\,\theta,\theta'\in\mathcal N_N.
\]
\end{ass}
 
A uniform operator-norm bound on the derivative of the map parametrization implies this condition.
 
\subsubsection{Gradient concentration and localization}
\label{subsec:emp-proc}

\begin{ass}
\label{ass:grad-conc}
There are a constant $C_g<\infty$, not depending on $N$, $\delta$, $d_N$ or $p_N$, and a
sequence $\zeta_N\to0$, such that for every $N$ and every $\delta\in(0,1)$,
\[
  \mathbb P\left(
  \sup_{\theta\in\mathcal N_N}\big\|\nabla\widehat R_N^\circ(\theta)-\nabla R_N(\theta)\big\|
  >C_g\,(1+\sigma)L_N\sqrt{\frac{p_N}{N\delta}}
  \right)\ \le\ \delta+\zeta_N,
\]
with $L_N$ as in Assumption~\ref{ass:lipschitz-map}.
\end{ass}

The constant $C_g$ is uniform in $N$, $\delta$, $d_N$ and $p_N$, as the
statements below require; it is allowed to depend on the sieve family and on the constants
appearing in whichever primitive conditions are used to verify this assumption, and
Proposition~\ref{prop:verify} records that dependence explicitly. The factor
$1+\sigma$ reflects the separate design and noise contributions to the gradient
fluctuation, and keeps $C_g$ free of the noise level as $\sigma\downarrow0$; the Gaussian
and block calculations isolate a variance term proportional to $\sigma^2$. The
sequence $\zeta_N$ absorbs any high-probability design event used to verify the derivative
envelopes, and is zero when those envelopes hold $\pi_0$-a.s.

\begin{ass}
\label{ass:localization}
There is a sequence $\eta_N\to0$ such that, for every $N$,
\[
  \mathbb P\big(\widehat\theta_N\notin\mathcal N_N\big)\ \le\ \eta_N.
\]
\end{ass}
 
A sufficient condition for localization is a positive population margin
\[
m_N:=\inf_{\theta\in\Theta_N\setminus\mathcal N_N}
[R_N(\theta)-R_N(\theta_N^*)]>0
\]
and $\sup_{\Theta_N}|\widehat R_N^\circ-\mathbb E\widehat R_N^\circ|=o_p(m_N)$.
The ERM basic inequality then excludes the complement of $\mathcal N_N$ with
probability tending to one. This condition accounts for the shrinking neighborhood
and growing parameter dimension. The primitive derivative conditions in
Section~\ref{app:verify} address local curvature and concentration; localization remains
separate. The explicit classes below admit direct risk calculations.

\subsubsection{Condition number and an oracle inequality}
\label{subsec:oracle}
 
The assumptions above suffice to bound the stochastic term of \eqref{eq:s3-decomp}, and
with it the whole prediction error, in terms of the sieve's complexity and the local
geometry of $R_N$ at the oracle. A single quantity summarizes that geometry.
 
\begin{definition}
\label{def:kappa}
\[
\kappa_N:=\frac{L_N}{\mu_N},
\]
with $L_N$ the local Lipschitz constant of $\theta\mapsto T_{\theta,N}$, as in
Assumption~\ref{ass:lipschitz-map}, and $\mu_N$ the local strong-convexity
parameter of $R_N$, as in \eqref{eq:quad-growth}.
\end{definition}
Theorem~\ref{thm:error_bounds} below is an abstract local oracle result.
Assumptions~\ref{ass:invertible}--\ref{ass:grad-conc} control the risk and the empirical
fluctuation on $\mathcal N_N$, while Assumption~\ref{ass:localization} ensures that the
empirical minimizer enters that neighborhood with high probability. The theorem does not
require $\widehat\theta_N$ to be an interior stationary point of the empirical criterion,
since its proof uses only the ERM basic inequality
$\widehat R_N^\circ(\widehat\theta_N)\le\widehat R_N^\circ(\theta_N^*)$.
\begin{theorem}
\label{thm:error_bounds}
Under Assumptions~\ref{ass:invertible}--\ref{ass:localization}, with
$C:=2C_g$ for the constant $C_g$ of Assumption~\ref{ass:grad-conc}, for every
$\delta\in(0,1)$, with probability at least
$1-\delta-\eta_N-\zeta_N$ (for the $\zeta_N$ of Assumption~\ref{ass:grad-conc}; $\zeta_N=0$
whenever (B1) and (B3) are verified in their a.s.\ form),
\[
\|\widehat T_N-T\|_{L^2(\pi_0;\mathcal H)}
\le
C\,(1+\sigma)\kappa_NL_N\sqrt{\frac{p_N}{N\delta}}
+\|T_N^*-T\|_{L^2(\pi_0;\mathcal H)}.
\]
The proof is given in Section~\ref{sup:proofs-general}.
\end{theorem}

The bound separates the stochastic estimation error from the deterministic
approximation error. On the smoothness class of Section~\ref{sec:approx-section}, where
Theorem~\ref{thm:approx} gives $\|T_N^*-T\|_{L^2(\pi_0;\mathcal H)}\le C_ad_N^{-s}$, the
two terms are, schematically, an approximation error $d_N^{-s}$ and an estimation error
$(1+\sigma)\kappa_NL_N\sqrt{p_N/N}$ carrying the conditioning factor. The first is
deterministic, unaffected by the sampling. In the second, the effective
dimension $p_N$ is scaled by $\kappa_NL_N=L_N^2/\mu_N$, where one factor of $L_N$ is
carried by the gradient concentration of Assumption~\ref{ass:grad-conc}, the second
converts parameter error into map error via Assumption~\ref{ass:lipschitz-map}, and the
local condition number $\kappa_N=L_N/\mu_N$ measures the loss in between.
Localization enters only through the failure probability $\eta_N$, since
Assumptions~\ref{ass:invertible}--\ref{ass:grad-conc} are imposed on $\mathcal N_N$
alone and apply only when $\widehat\theta_N$ lands there.

The estimation bound depends on the parametrization through $\kappa_N=L_N/\mu_N$. Two parametrizations of the same candidate class have the same approximation error but may have different values of $\kappa_N$ and hence different bounds. Exact empirical minimization over a fixed candidate class is unchanged by a bijective reparametrization.
 
Increasing $d_N$ decreases the approximation term but generally increases $p_N$, $L_N$,
and $\kappa_N$, and with them the estimation term. The balance between the two fixes
both the attainable rate and the sieve dimension attaining it.


\section{Rates and minimax lower bound}
\label{sec:rates-minimax}

Combining the approximation bound of Section~\ref{sec:approx-section} with the oracle inequality of Section~\ref{sec:stoch-section} gives the rate of the cylindrical-sieve estimator. Section~\ref{sec:rates-section} gives the nonasymptotic rate and the balanced sieve dimension. Section~\ref{sec:minimax} gives a minimax lower bound of order $N^{-s/(2s+1)}$ over $\mathcal W^s(B)$. On the diagonal subclass $\mathcal W^s_{\mathrm{diag}}(B)$ with $\pi_0=\gamma$, the sieve attains the same rate in the well-conditioned case.

\subsection{Nonasymptotic rates and the oracle sieve dimension}
\label{sec:rates-section}
 
We first state the nonasymptotic bound. Its balance depends on how $L_N$, $\mu_N$, and $p_N$ grow with the sieve dimension.
 
\subsubsection{Nonasymptotic rates, consistency, and the parametric regime}
\label{subsec:nonasymptotic}
 
\begin{corollary}
\label{cor:nonasymptotic}
Suppose the assumptions of Theorem~\ref{thm:error_bounds} hold, $T\in\mathcal W^s$, and
Assumption~\ref{ass:expressivity} holds. Then, for the same constant $C$ and every
$\delta\in(0,1)$, with probability at least $1-\delta-\eta_N-\zeta_N$,
\[
\|\widehat T_N-T\|_{L^2(\pi_0;\mathcal H)}
\ \le\
C\,(1+\sigma)\kappa_NL_N\sqrt{\frac{p_N}{N\delta}}
\ +\
C_a\,d_N^{-s}
\]
for the constant $C_a$ of Theorem~\ref{thm:approx}. The bound follows by inserting the
approximation bound of Theorem~\ref{thm:approx} into Theorem~\ref{thm:error_bounds}.
\end{corollary}
 
Fixing $\delta$ at any constant in
Corollary~\ref{cor:nonasymptotic} and using $\eta_N+\zeta_N\to0$ gives the asymptotic form
\begin{equation}
\label{eq:op-rate}
\|\widehat T_N-T\|_{L^2(\pi_0;\mathcal H)}
=O_p\!\Big((1+\sigma)\kappa_NL_N\sqrt{\tfrac{p_N}{N}}\Big)+O\big(d_N^{-s}\big).
\end{equation}
In particular, the estimator is consistent whenever
$\kappa_NL_N\sqrt{p_N/N}\to0$ and $d_N^{-s}\to0$. If in addition
$\kappa_NL_N=O(1)$, $p_N=O(1)$, and $d_N^{-s}=o(N^{-1/2})$, then
$\|\widehat T_N-T\|_{L^2(\pi_0;\mathcal H)}=O_p(N^{-1/2})$: the usual parametric
root-$N$ rate.
 
\subsubsection{Role of spectral decay and sieve complexity}
\label{subsec:spectral}
 
In concrete settings the local constants admit polynomial orders in the sieve
dimension. Suppose there are exponents $a_L,a_\mu\ge0$ and $q\ge0$ with
\[
L_N\asymp d_N^{a_L},
\qquad
\mu_N\asymp d_N^{-a_\mu},
\qquad
p_N\asymp d_N^{q},
\]
so that $\kappa_N\asymp d_N^{a_L+a_\mu}$, $\kappa_NL_N\asymp d_N^{2a_L+a_\mu}$, and the stochastic
term of Corollary~\ref{cor:nonasymptotic} is of order
$(1+\sigma)d_N^{2a_L+a_\mu+q/2}N^{-1/2}$. Balancing it against the bias $d_N^{-s}$ selects
\[
d_N^\star\asymp \big(N/(1+\sigma)^2\big)^{\,1/\big(2(s+2a_L+a_\mu)+q\big)},
\quad
\|\widehat T_N-T\|_{L^2(\pi_0;\mathcal H)}
=O_p\!\Big(\big((1+\sigma)^2/N\big)^{s/\big(2(s+2a_L+a_\mu)+q\big)}\Big).
\]
Assumption~\ref{ass:grad-conc} gives the factor $1+\sigma$ in this general bound. For fixed $\sigma>0$, it does not affect the exponent in $N$. The Gaussian and block examples below have explicit variance terms proportional to $\sigma^2$. The choice $d_N^\star$ balances the two terms for given smoothness and conditioning exponents.

The exponent $s$ is determined by the target coefficient sequence, while $a_L$, $a_\mu$, and $q$ depend on the sieve and its parametrization. Since the coordinates $\hat e_k$ are standardized, the eigenvalues $\lambda_k$ cancel from the estimation term. The factor $L_N$ enters both gradient concentration and the parameter-to-map conversion, which gives the exponent $2a_L+a_\mu$ in the stochastic term. Section~\ref{subsec:gaussian} gives these quantities explicitly for the conjugate Gaussian example.


\subsection{Minimax lower bound}
\label{sec:minimax}

We derive a lower bound of the same order as the upper rate in
Section~\ref{sec:rates-section}. The Fano argument is stated for $\sigma>0$ under the
noisy model of Assumption~\ref{ass:data},
\[
f_1^{(i)}=T(f_0^{(i)})+\xi^{(i)},
\qquad
\xi^{(i)}\ \text{centered Gaussian, independent of }f_0^{(i)},
\quad
\mathbb E\big[(\xi^{(i)}_k)^2\big]=\sigma^2 .
\]
When $\sigma=0$, the conditional observation laws are degenerate and the
Kullback--Leibler calculation used in the proof is not available. Write
$\mathcal W^s(B):=\{T=I+\nabla_{\mathcal H}\phi:\|T\|_{\mathcal W^s}\le B\}$, and let
$\mathbb E_T$ denote expectation under this model with map $T$ and $N$ paired draws.

The lower bound holds over $\mathcal W^s(B)$ for general $\pi_0$. The matching upper bound is proved on the aligned diagonal subclass defined below with $\pi_0=\gamma$, where the sieve has a uniform $d^{-s}$ approximation rate. Section~\ref{app:expressivity} shows that Assumption~\ref{ass:expressivity} for a single target does not imply a uniform approximation bound over all of $\mathcal W^s(B)$. A matching uniform upper bound over the full class is not established here.

\begin{definition}
\label{def:diag-class}
For $s>0$ and $B>0$, the diagonal subclass
\[
\mathcal W_{\mathrm{diag}}^s(B):=\Big\{T=I+\nabla_{\mathcal H}\phi:
\phi(f)=\tfrac12\sum_{k\ge1}a_k\big(\hat e_k(f)^2-1\big),\ \ a_k\ge\underline a\ \forall k,\ \ 
\sum_{k\ge1}k^{2s}a_k^2\le B^2\Big\}
\]
consists of the transport maps whose displacement is diagonal in the Cameron--Martin
eigenbasis, $u_k=a_k\hat e_k$. The centering $-1$ makes $\phi$ well defined in
$L^2(\gamma)$ whenever $\sum_ka_k^2<\infty$. The summands are independent, mean zero, and have summable variance, exactly as in Remark~\ref{rem:gaussian-rates}. Here $\underline a\in(-1,0]$ is a floor on the displacement coefficients, fixed once and
for all; it enters constants but no rate. The sign condition on $a_k$ makes
$\phi$ $\mathcal H$-convex. Writing
$h=\sum_kh_ke_k\in\mathcal H$ and using linearity of $\hat e_k$,
$\phi(f+h)+\tfrac12|h|_{\mathcal H}^2=\phi(f)+\sum_ka_k\hat e_k(f)h_k+\tfrac12\sum_k(a_k+1)h_k^2$,
a convex quadratic form in $h$ iff $a_k+1\ge0$ for every $k$. The floor is bounded away
from $-1$ because the endpoint $a_k=-1$ is inadmissible in the standing setup. It collapses
the $k$th coordinate, $T(f)$ having $(1+a_k)\hat e_k(f)=0$ there, so $T_\#\gamma$ is
degenerate along $e_k$ and the absolute continuity $\pi_1\ll\gamma$ required by (A3) of
Section~\ref{subsec:rep} fails. Under $\pi_0=\gamma$ the
coordinates satisfy $\hat e_k\sim N(0,1)$, so $\|u_k\|_{L^2(\pi_0)}=|a_k|$ and hence
$\mathcal W_{\mathrm{diag}}^s(B)\subset\mathcal W^s(B)$; for a general source the two
norms differ by the factors $\|\hat e_k\|_{L^2(\pi_0)}$ and the containment need not
hold with the same $B$.
\end{definition}

\subsubsection{Construction of the hypothesis family}
\label{subsec:hypothesis}
 
Fix a resolution $d\ge1$ and perturb a band of $d$ high coordinates. Let
$\{v_k\}_{k=d+1}^{2d}$ be coordinate fields with $\|v_k\|_{L^2(\pi_0)}=1$, each acting along the Cameron--Martin direction $e_k$. Normalized coordinate functionals provide one example. For
$\omega\in\{0,1\}^{d}$ set
\[
T_\omega:=I+\tau\sum_{k=d+1}^{2d}\omega_{k}\,v_k\,e_k,
\qquad \tau>0 .
\]
The coordinate fields of $T_\omega-I$ are $u_k^\omega=\tau\,\omega_k\,v_k$, so
$\|u_k^\omega\|_{L^2(\pi_0)}=\tau\,\omega_k$.

Concretely, take $v_k:=\hat e_k/\|\hat e_k\|_{L^2(\pi_0)}$, the normalized first-chaos
functional of $e_k$ under $\pi_0$; by (A1) and $\pi_0\ll\gamma$ with $f_0>0$ $\gamma$-a.e.,
$0<\|\hat e_k\|_{L^2(\pi_0)}<\infty$, so $v_k$ is well defined and $\|v_k\|_{L^2(\pi_0)}=1$.
Setting $c_k:=1/\|\hat e_k\|_{L^2(\pi_0)}$ and
\[
\phi_\omega(f):=\frac{\tau}{2}\sum_{k=d+1}^{2d}\omega_k\,c_k\,\hat e_k(f)^2,
\]
the chain rule \eqref{eq:chain} gives
$\nabla_{\mathcal H}\phi_\omega=\tau\sum_{k=d+1}^{2d}\omega_k c_k\hat e_k\,e_k
=\tau\sum_{k=d+1}^{2d}\omega_k v_k\,e_k$, so $T_\omega=I+\nabla_{\mathcal H}\phi_\omega$ is of Brenier form. The potential $\phi_\omega$ is a quadratic form in the $\hat e_k$'s with
nonnegative coefficients $\tau\omega_kc_k\ge0$, hence convex along $\mathcal H$; the map $h\mapsto\phi_\omega(f+h)+|h|_{\mathcal H}^2/2$ is therefore convex, so
$\phi_\omega$ is $\mathcal H$-convex in the sense of Definition~\ref{def:Hconvex}. Since $f_0>0$ $\gamma$-a.e., the standing assumption $\pi_0=f_0\gamma$ implies
$\pi_0\sim\gamma$. The map $T_\omega$ differs from the identity only through finitely many
strictly positive coordinate scalings $1+\tau\omega_kc_k>0$, so $T_\omega$ is a linear
isomorphism fixing all but finitely many coordinates and $T_{\omega\#}\gamma\sim\gamma$.
Therefore $T_{\omega\#}\pi_0\sim T_{\omega\#}\gamma\sim\gamma$, so the target measure
$T_{\omega\#}\pi_0$ has a strictly positive density with respect to $\gamma$, as required
by (A3). Each $T_\omega$ also has finite Cameron--Martin transport cost, since its
displacement is supported on $d$ coordinates with $\|u_k^\omega\|_{L^2(\pi_0)}=\tau\omega_k$,
so (A1) and (A2) hold as well and $T_\omega$ is admissible in the standing model. That
$T_\omega$ is the optimal map for the pair $(\pi_0,T_{\omega\#}\pi_0)$ follows
coordinatewise, since for any coupling the $k$th coordinate cost is at least
$W_2^2$ between the two $k$th marginals, and the monotone coordinatewise scaling
$T_\omega$ attains every one of these lower bounds simultaneously. Therefore, $T_\omega\in\mathcal W^s(B)$ whenever the norm bound of
part (i) below does not exceed $B$. The three properties the Fano argument needs are
collected next and proved in Section~\ref{app:proof-hypothesis}.

\begin{lemma}
\label{lem:hypothesis}
There is a subset $\Omega\subseteq\{0,1\}^d$ with $\log|\Omega|\ge c\,d$ such that
\begin{enumerate}
\item[(i)] $\|T_\omega\|_{\mathcal W^s}\le C_1\,\tau\,d^{\,s+1/2}$ for every $\omega\in\Omega$;
\item[(ii)] $\|T_\omega-T_{\omega'}\|_{L^2(\pi_0;\mathcal H)}^2\ge c_2\,\tau^2\,d$ for distinct $\omega,\omega'\in\Omega$;
\item[(iii)] $\mathrm{KL}\big(P_{T_\omega}^{\otimes N}\,\big\|\,P_{T_{\omega'}}^{\otimes N}\big)
\le \dfrac{N\,\tau^2 d}{2\sigma^2}$.
\end{enumerate}
\end{lemma}

The factor $\sigma^{-2}$ in (iii) is where the noiseless case fails, the bound diverging
as $\sigma\downarrow0$, so that the Fano step of Section~\ref{subsec:minimax-bound} has no
usable information bound in that limit.

\subsubsection{Statement of the lower bound}
\label{subsec:minimax-bound}
 
\begin{theorem}
\label{thm:minimax}
There are constants $c=c(s,B,\sigma)>0$ and $N_0=N_0(s,B,\sigma)$, neither depending on
$N$, such that, for all $N\ge N_0$,
\[
\inf_{\widehat T}\ \sup_{T\in\mathcal W^s(B)}\
\mathbb E_T\big\|\widehat T-T\big\|_{L^2(\pi_0;\mathcal H)}
\ \ge\ c\,\Big(\frac{\sigma^2}{N}\Big)^{s/(2s+1)},
\]
the infimum taken over all estimators measurable in the paired sample. The parameters $s,B,\sigma>0$ are fixed in this asymptotic statement. When $\pi_0=\gamma$, the same lower bound holds
with $\mathcal W^s(B)$ replaced by the diagonal subclass
$\mathcal W^s_{\mathrm{diag}}(B)$ of Definition~\ref{def:diag-class}.
\end{theorem}

The proof in Section~\ref{app:proof-minimax} uses a Fano argument over a hypercube of perturbed diagonal maps. Proposition~\ref{prop:gaussian-verify} gives the matching expected-norm upper bound on $\mathcal W^s_{\mathrm{diag}}(B)$ with $\pi_0=\gamma$, where the diagonal quadratic sieve satisfies Assumption~\ref{ass:expressivity} uniformly. Hence the cylindrical-sieve estimator is minimax rate-optimal on this class. Section~\ref{subsec:block} gives the same rate $N^{-s/(2s+1)}$ on a nonlinear block class with a non-Gaussian pushforward. The two-groups class of Section~\ref{subsec:bayesian-example} has the same exponent up to a logarithmic factor; a matching lower bound for that class is not available.


\section{Input dependence and cylindrical approximation}
\label{sec:input}

Definition~\ref{def:reg-class} controls decay in the output index. The family in Section~\ref{app:expressivity} shows that this condition alone does not give uniform $d^{-s}$ cylindrical approximation when a retained output depends strongly on a discarded input coordinate. We therefore add an input regularity index and derive a uniform approximation bound. Throughout this section, $\pi_0=\gamma$, so the Gaussian Poincar\'e inequality is available.

\subsection{Two-index regularity class}
\label{subsec:two-index}

The classes in this section are defined as gradient-map classes. For a set $\mathcal G$ of maps $T=I+\nabla_{\mathcal H}\phi$, define
\begin{equation}
\label{eq:transport-subclass}
  \mathcal G^{\mathrm{OT}}:=\big\{T\in\mathcal G:\ \phi\ \text{is}\ \mathcal H\text{-convex
  in the sense of Definition~\ref{def:Hconvex}}\big\}
\end{equation}
for its transport subclass, on which Assumption~\ref{ass:rep} holds and $T$ is the optimal
map between $\gamma$ and $T_\#\gamma$. The upper bounds for $\mathcal G$ also hold on $\mathcal G^{\mathrm{OT}}$. Corollary~\ref{cor:additive-minimax} proves the lower bound on the transport subclass.

Let $\phi\in\mathbb D^{2,2}(\gamma)$, so that the second Cameron--Martin derivatives
$D_ju_k=D_jD_k\phi$ exist in $L^2(\gamma)$, where $u_k=D_k\phi$ as in
Section~\ref{subsec:regularity}.

\begin{definition}
\label{def:two-index}
For $s,t>0$ and $B_1,B_2>0$, let $\mathcal W^{s,t}(B_1,B_2)$ be the set of gradient maps
$T=I+\nabla_{\mathcal H}\phi$ with $\phi\in\mathbb D^{2,2}(\gamma)$ such that
\begin{equation}
\label{eq:two-index}
  \|T\|_{\mathrm{out},s}^2:=\sum_{k\ge1}k^{2s}\,\|u_k\|_{L^2(\gamma)}^2\ \le\ B_1^2,
  \qquad
  \|T\|_{\mathrm{in},t}^2:=\sum_{j\ge1}j^{2t}\sum_{k\ge1}\|D_ju_k\|_{L^2(\gamma)}^2\ \le\ B_2^2 .
\end{equation}
\end{definition}

The first quantity is the norm of Definition~\ref{def:reg-class} and measures how fast the
displacement decays across output directions. The second measures how strongly the
displacement, in all its output coordinates at once, responds to a change in the $j$th
input coordinate, with the same polynomial weight. Both are expectations of explicit
derivatives of $\phi$ and can be evaluated for a given model, as we do in
Remark~\ref{rem:input-counterexample} and Section~\ref{subsec:block}.

Write $\mathcal F_d:=\sigma(\hat e_1,\dots,\hat e_d)$ for the $\sigma$-field of the first
$d$ Cameron--Martin coordinates and
\begin{equation}
\label{eq:def-gd}
  g_d:=\mathbb E_\gamma\big[\phi\mid\mathcal F_d\big],
  \qquad
  T_d:=I+\nabla_{\mathcal H}g_d .
\end{equation}
Since $g_d$ is a function of $\hat e_1,\dots,\hat e_d$ alone, $T_d$ is a cylindrical
gradient map in the sense of Definition~\ref{def:cylindrical}, unlike the output
truncation $T^{(d)}$ of Section~\ref{subsec:regularity}.

\begin{theorem}
\label{thm:two-index-approx}
Let $T\in\mathcal W^{s,t}(B_1,B_2)$. Then, for every $d\ge1$,
\begin{equation}
\label{eq:two-index-approx}
  \big\|T-T_d\big\|_{L^2(\gamma;\mathcal H)}^2
  \ \le\ B_2^2\,d^{-2t}\ +\ B_1^2\,d^{-2s} .
\end{equation}
Consequently $\inf_{\text{cylindrical }g}\|T-(I+\nabla_{\mathcal H}g)\|_{L^2(\gamma;\mathcal H)}
\le (B_1^2d^{-2s}+B_2^2d^{-2t})^{1/2}$, uniformly over $\mathcal W^{s,t}(B_1,B_2)$.
\end{theorem}

The first term in \eqref{eq:two-index-approx} controls dependence on discarded input coordinates, while the second controls discarded output directions. The first term follows from the conditional Gaussian Poincar\'e inequality. The class and truncation are defined through the eigenbasis $\{e_k\}$. Both indices are needed over $\mathbb D^{2,2}(\gamma)$: Section~\ref{app:expressivity} gives a counterexample without input regularity, while output regularity controls the tail $\sum_{k>d}\|u_k\|^2$.

\begin{rem}
\label{rem:input-counterexample}
The family of Section~\ref{app:expressivity} is excluded by the second condition. For that family, There $\phi_m(z)=an^{-2}\sin(nz_1)\sin(z_m)$ with
$n=m^{s/2}$, so $u_1=an^{-1}\cos(nz_1)\sin(z_m)$ and
$D_mu_1=an^{-1}\cos(nz_1)\cos(z_m)$, whence
$\|D_mu_1\|_{L^2(\gamma)}^2=\tfrac{a^2}{n^2}\,\mathbb E[\cos^2(nZ)]\,\mathbb E[\cos^2 Z]
\ \ge\ c\,a^2n^{-2}$ for an absolute $c>0$ and $n\ge1$. Retaining only the term $j=m$,
\[
  \|T\|_{\mathrm{in},t}^2\ \ge\ m^{2t}\,\|D_mu_1\|_{L^2(\gamma)}^2\ \ge\ c\,a^2\,m^{2t-s},
\]
which diverges as $m\to\infty$ whenever $t>s/2$. The family therefore stays in a fixed
$\mathcal W^s$ ball while leaving every $\mathcal W^{s,t}$ ball with $t>s/2$, which identifies the role of the input regularity condition in Theorem~\ref{thm:two-index-approx}.
\end{rem}

\subsection{From a cylindrical map to a finite-parameter sieve}
\label{subsec:complexity}

Theorem~\ref{thm:two-index-approx} gives a cylindrical map with no finite parametrization. We next introduce a finite Hermite parametrization for a structured subclass and record its dimension.

Let $\{h_n\}_{n\ge0}$ be the Hermite polynomials orthonormal in $L^2(N(0,1))$, so
$h_n'=\sqrt n\,h_{n-1}$, and for a multi-index $\alpha\in\mathbb N^{\mathbb N}$ of finite
support put $H_\alpha:=\prod_j h_{\alpha_j}(\hat e_j)$. Every $\phi\in L^2(\gamma)$ has a
chaos expansion $\phi=\sum_\alpha c_\alpha H_\alpha$. Two elementary identities drive
everything below. Since $\partial_kH_\alpha=\sqrt{\alpha_k}\,H_{\alpha-e_k}$,
\begin{equation}
\label{eq:hermite-orth}
  \big\langle\nabla_{\mathcal H}H_\alpha,\nabla_{\mathcal H}H_\beta\big\rangle_{L^2(\gamma;\mathcal H)}
  =|\alpha|\,\delta_{\alpha\beta},
  \qquad |\alpha|:=\sum_j\alpha_j .
\end{equation}
The gradient fields $\{\nabla_{\mathcal H}H_\alpha\}$ are therefore orthogonal in
$L^2(\gamma;\mathcal H)$, with squared norms $|\alpha|$. In these coordinates
\begin{equation}
\label{eq:norms-in-chaos}
  \|T-I\|_{L^2(\gamma;\mathcal H)}^2=\sum_\alpha c_\alpha^2|\alpha|,
  \qquad
  \|T\|_{\mathrm{out},s}^2=\sum_\alpha c_\alpha^2\,w_s(\alpha),
  \qquad
  \|T\|_{\mathrm{in},t}^2=\sum_\alpha c_\alpha^2\,(|\alpha|-1)\,w_t(\alpha),
\end{equation}
with $w_r(\alpha):=\sum_jj^{2r}\alpha_j$; these are verified in
Section~\ref{app:two-index}. Since $w_r(\alpha)\ge|\alpha|$, the output condition alone
gives $\|T-I\|^2\le B_1^2$.

\begin{definition}
\label{def:structured}
For integers $1\le q\le m$ let
$\Lambda_{q,m}:=\{\alpha:1\le|\alpha|\le m,\ |\operatorname{supp}\alpha|\le q\}$, and let
\[
  \mathcal C_{q,m}^{s}(B):=\Big\{T=I+\nabla_{\mathcal H}\phi:\ \|T\|_{\mathrm{out},s}\le B,\ \
  \phi=\textstyle\sum_{\alpha\in\Lambda_{q,m}}c_\alpha H_\alpha\Big\}.
\]
\end{definition}

The potential is a polynomial of degree at most $m$ in the Cameron--Martin coordinates,
and a sum of terms each involving at most $q$ coordinates, an analysis of variance
truncated at interaction order $q$. Only the output condition is imposed. Since a multi-index with $|\alpha|\le m$ has support size at most $m$, we take $q\le m$; otherwise the dimension count below is $d^{\,q\wedge m}$.

The diagonal class of Definition~\ref{def:diag-class} is contained in
$\mathcal C_{1,2}^{s}(B)$, since a diagonal potential
$\tfrac12\sum_ka_k(\hat e_k^2-1)$ has $c_{2e_k}=a_k/\sqrt2$ and
$\|T\|_{\mathrm{out},s}^2=\sum_kk^{2s}a_k^2$. For bounded chaos degree, no separate input index is needed. If $\alpha\notin\mathcal I_d$ below, some $j>d$ has $\alpha_j\ge1$, so
$w_s(\alpha)\ge(d+1)^{2s}$, while $|\alpha|\le m$ bounds the term. Thus the degree bound limits dependence on discarded input coordinates. By \eqref{eq:norms-in-chaos} and $|\alpha|\le m$,
\[
  \|T\|_{\mathrm{in},s}^2=\sum_\alpha c_\alpha^2(|\alpha|-1)w_s(\alpha)
  \le(m-1)\sum_\alpha c_\alpha^2w_s(\alpha)\le(m-1)B^2 ,
\]
so $\mathcal C_{q,m}^{s}(B)\subseteq\mathcal W^{s,s}\big(B,\sqrt{m-1}\,B\big)$, with any positive second radius when $m=1$. Thus the degree bound gives the required input-derivative control at the matched index.

The sieve is the corresponding finite index set and the linear family it spans,
\begin{equation}
\label{eq:hermite-sieve}
  \mathcal I_d:=\{\alpha\in\Lambda_{q,m}:\operatorname{supp}\alpha\subseteq[d]\},
  \qquad
  g_\theta:=\sum_{\alpha\in\mathcal I_d}\theta_\alpha H_\alpha,
  \qquad
  p_d:=|\mathcal I_d| .
\end{equation}
Counting supports of each size gives
$p_d=\sum_{r=1}^{q\wedge d}\binom dr\binom mr\le q\,(dm)^q$, and $p_d\sim\binom
mq d^q/q!$ as $d\to\infty$ at fixed $q\le m$: the sieve dimension grows polynomially in
$d$ with exponent equal to the interaction order.

Because the family is linear in $\theta$ and the features are orthogonal by
\eqref{eq:hermite-orth}, the population projection is explicit. We use the coordinatewise moment estimator With
$D_i=(D_{ik})_{k\le d}$ the observed displacement coordinates \eqref{eq:Dik} and
$\nabla H_\alpha(Z_i)$ the feature evaluated at the $i$th design point, set
\begin{equation}
\label{eq:proj-estimator}
  \widehat c_\alpha:=\frac1{N|\alpha|}\sum_{i=1}^N
  \big\langle D_i,\nabla H_\alpha(Z_i)\big\rangle,
  \qquad \alpha\in\mathcal I_d,
  \qquad
  \widehat T_d:=I+\sum_{\alpha\in\mathcal I_d}\widehat c_\alpha\nabla_{\mathcal H}H_\alpha .
\end{equation}
Each $\widehat c_\alpha$ is an average of i.i.d.\ real random variables and is unbiased for
$c_\alpha$ by orthogonality of the Hermite features.

\begin{theorem}
\label{thm:structured-risk}
Fix $1\le q\le m$ and $s>0$, and let $\widehat T_d$ be the estimator
\eqref{eq:proj-estimator}. Then, for every $N\ge1$ and $d\ge1$,
\begin{equation}
\label{eq:structured-risk}
  \sup_{T\in\mathcal C_{q,m}^{s}(B)}
  \mathbb E_T\big\|\widehat T_d-T\big\|_{L^2(\gamma;\mathcal H)}^2
  \ \le\ \big(9^{m-1}B^2+\sigma^2\big)\frac{p_d}{N}
  \ +\ m\,B^2(d+1)^{-2s} .
\end{equation}
Choosing $d\asymp\big(N/(B^2+\sigma^2)\big)^{1/(q+2s)}$ gives
\begin{equation}
\label{eq:structured-rate}
  \sup_{T\in\mathcal C_{q,m}^{s}(B)}
  \mathbb E_T\big\|\widehat T_d-T\big\|_{L^2(\gamma;\mathcal H)}
  \ \lesssim\ \Big(\frac{B^2+\sigma^2}{N}\Big)^{s/(q+2s)},
\end{equation}
with an implicit constant depending only on $q$, $m$ and $B$.
\end{theorem}

\begin{corollary}
\label{cor:additive-minimax}
Let $q=1$ and $m\ge2$, and let $\mathcal C_{1,m}^{s}(B)^{\mathrm{OT}}$ be the transport
subclass \eqref{eq:transport-subclass}. Then
\[
  \inf_{\widehat T}\ \sup_{T\in\mathcal C_{1,m}^{s}(B)^{\mathrm{OT}}}\
  \mathbb E_T\big\|\widehat T-T\big\|_{L^2(\gamma;\mathcal H)}\ \asymp\ N^{-s/(2s+1)},
\]
and the same holds with the supremum over the full gradient class
$\mathcal C_{1,m}^{s}(B)$. The estimator \eqref{eq:proj-estimator} is therefore minimax
rate-optimal on both.
\end{corollary}

The upper bound is \eqref{eq:structured-rate} at $q=1$, valid on the gradient class and
hence on the subclass. The lower bound is Theorem~\ref{thm:minimax}: at $\pi_0=\gamma$ its
hypothesis family has potentials $\tfrac\tau2\sum_k\omega_k\hat e_k^2$, whose chaos
expansion is supported on the indices $2e_k$ and whose second derivatives are the constants
$\tau\omega_k\in[0,\tau]$, so $1+\tau\omega_k>0$ and the family lies in
$\mathcal C_{1,2}^{s}(B)^{\mathrm{OT}}\subseteq\mathcal C_{1,m}^{s}(B)^{\mathrm{OT}}$. The bound therefore applies to the transport subclass. For comparison with the diagonal class of Definition~\ref{def:diag-class}, $\mathcal C_{q,m}^{s}(B)$ as defined is
a class of gradient maps and the standing model of Section~\ref{subsec:rep} asks for
transport maps. Within $\mathcal C_{1,m}^{s}(B)$ the additional requirement is
$1+\psi_k''\ge0$ for each univariate potential $\psi_k$, which is
Definition~\ref{def:Hconvex} in the separable case. Since $\psi_k''$ is a polynomial of
degree $m-2$ bounded below by $-1$ on all of $\mathbb R$, it must be constant or of even
degree with positive leading coefficient. So $m=3$ adds no transport map beyond $m=2$: a
nonzero cubic term makes $1+\psi_k''$ a non-constant affine function, which changes sign.
From $m=4$ the transport subclass is strictly larger than the diagonal class, for instance
$\psi_k(z)=a_kz^4$ with $a_k>0$, for which $1+12a_kz^2>0$ everywhere and the induced map
$z\mapsto z+4a_kz^3$ is nonlinear.

The estimator attains $s/(q+2s)$ at every $q$; for $q\ge2$ its optimality is open. What
the exponent does show is how the sieve dimension
enters. The parameter count $p_d$ grows like $d^q$, and the balance replaces the $1$ of the additive case by
$q$.

\begin{rem}
\label{rem:structured-scope}
The scope of Theorem~\ref{thm:structured-risk} is as follows. The degree $m$ is fixed and the constant $9^{m-1}$ is exponential in it, so the result is
informative for low-degree, low-order interactions; Section~\ref{subsec:three-index}
examines that constant and lets $m$ grow. As stated the class does not contain the
trigonometric block potentials of Section~\ref{subsec:block}, which have unbounded chaos
degree. And for $q\ge2$ the optimality of the exponent is open.
\end{rem}

\subsection{Three-index interpolation with growing chaos degree}
\label{subsec:three-index}

Theorem~\ref{thm:two-index-approx} holds at unbounded chaos degree but produces no finite
parametrization; Theorem~\ref{thm:structured-risk} produces one but fixes the degree.
This subsection lets the degree grow with $N$ and exhibits how the three regularity
indices, together with the interaction order, jointly determine the rate.

For the estimator \eqref{eq:proj-estimator}, the exponential factor also appears in the variance. Take
$q=1$ and $\phi=c\,h_m(\hat e_1)$, so that $u_1=c\sqrt m\,h_{m-1}(\hat e_1)$ and
$\langle u,\nabla H_{me_1}\rangle=c\,m\,h_{m-1}(\hat e_1)^2$. Its variance is
$c^2m^2\big(\mathbb E[h_{m-1}^4]-1\big)$, while $\|T-I\|^2|\alpha|=c^2m^2$, so the ratio is
$\mathbb E[h_{m-1}^4]-1$. The fourth moments of normalized Hermite polynomials satisfy
$\mathbb E[h_n^4]=\sum_{r=0}^n\binom nr^4(r!)^2(2n-2r)!/(n!)^2$, and retaining the $r=0$
term alone gives
\[
  \mathbb E[h_n^4]\ \ge\ \binom{2n}n\ \ge\ \frac{4^n}{2\sqrt n},
\]
so the ratio grows exponentially in the degree. Thus the variance itself contains the exponential factor. A least-squares estimator would require control of the smallest eigenvalue of the empirical Gram matrix for high-degree Hermite features and of the truncated part of the expansion under random design.

The exponential variance factor limits the usable degree to $m=O(\log N)$. A polynomial chaos-degree tail,
$\sum_{|\alpha|>m}c_\alpha^2|\alpha|\lesssim m^{-2\beta}$, would then contribute only
$(\log N)^{-2\beta}$. We therefore assume geometric decay in the chaos degree.

\begin{definition}
\label{def:three-index}
For $\varpi>0$ set
$\|T\|_{\mathrm{ch},\varpi}^2:=\sum_\alpha c_\alpha^2|\alpha|\,e^{2\varpi|\alpha|}$, and
for $q\ge1$, $s,t>0$ let
\[
\begin{aligned}
  \mathcal K_q^{s,t,\varpi}(B_1,B_2,B_3):=\Big\{T=I+\nabla_{\mathcal H}\phi:\ \
  &|\operatorname{supp}\alpha|\le q\text{ whenever }c_\alpha\ne0,\\
  &\|T\|_{\mathrm{out},s}\le B_1,\ \|T\|_{\mathrm{in},t}\le B_2,\
  \|T\|_{\mathrm{ch},\varpi}\le B_3\Big\}.
\end{aligned}
\]
\end{definition}

No degree restriction is imposed. The index $\varpi$ asks the chaos coefficients to decay
geometrically, which is a Gaussian analyticity condition and is satisfied for every
$\varpi>0$ by the trigonometric potentials of Section~\ref{subsec:block}: the coefficients of
$\cos$ in the Hermite basis are $e^{-1/2}/\sqrt{n!}$ up to sign, so
$\sum_nn\,e^{2\varpi n}/n!<\infty$ for every $\varpi$. The class
$\mathcal K_q^{s,t,\varpi}$ therefore contains potentials that
Definition~\ref{def:structured} excludes.

\begin{theorem}
\label{thm:three-index}
Let $q\ge1$, $s,t>0$ and $\varpi>\tfrac12\log3$, and set
\[
  A_\varpi:=\frac{e^{-2\varpi}}{1-3e^{-2\varpi}},
  \qquad
  \rho_1:=\frac{2\varpi}{\log3}>1,
  \qquad
  \rho_0:=\frac{2s}{q+2s},
  \qquad
  \chi:=\frac{\rho_0}{\rho_1+\rho_0}\in(0,1).
\]
Take $m_N:=\lceil \chi\log N/\log3\rceil\vee q$ and
$d_N:=\big\lceil\big(N^{1-\chi}m_N^{-(q-1)}\big)^{1/(q+2s)}\big\rceil\vee1$, and let
$\widehat T$ be the estimator
\eqref{eq:proj-estimator} on the index set $\mathcal I_{d_N,m_N}$ of \eqref{eq:hermite-sieve}.
Then there is $C=C(q,s,\varpi,B_1,B_3,\sigma)$, proportional to
$A_\varpi B_3^2+\sigma^2$ in its dependence on the noise and the analytic radius, such
that for all $N\ge3$,
\begin{equation}
\label{eq:three-index-rate}
  \sup_{T\in\mathcal K_q^{s,t,\varpi}(B_1,B_2,B_3)}
  \mathbb E_T\big\|\widehat T-T\big\|_{L^2(\gamma;\mathcal H)}^2
  \ \le\ C\,(\log N)^{q}\;N^{-\frac{\rho_1\rho_0}{\rho_1+\rho_0}} .
\end{equation}
\end{theorem}

The exponent combines the analytic index and the finite-degree exponent harmonically:
$\big(\rho_1\rho_0/(\rho_1+\rho_0)\big)^{-1}=\rho_1^{-1}+\rho_0^{-1}$. As
$\varpi\to\infty$ it increases to $\rho_0=2s/(q+2s)$, so \eqref{eq:three-index-rate}
recovers the exponent of Theorem~\ref{thm:structured-risk} in the limit of very fast chaos
decay, at the cost of a polylogarithmic factor that fixing the degree removes. At the
threshold $\varpi\downarrow\tfrac12\log3$ we have $\rho_1\downarrow1$ and the exponent
falls to $\rho_0/(1+\rho_0)$; below the threshold the fourth-moment bound of
Lemma~\ref{lem:analytic-moment} diverges and the argument gives nothing.

\begin{rem}
\label{rem:three-index-roles}
The four structural quantities act in different places.
The interaction order $q$ and the output index $s$ fix the sieve dimension and the
coordinate-truncation error, and enter the exponent through $\rho_0=2s/(q+2s)$. The
analytic index $\varpi$ fixes the degree-truncation error and enters through
$\rho_1=2\varpi/\log3$, the $\log3$ coming from the hypercontractive constant that limits
how fast the degree may grow. The input index $t$ affects only logarithmic factors here. The
proof bounds the coordinate-truncation error by the smaller of $m_NB_1^2d^{-2s}$, which
uses the output index alone, and $B_1^2d^{-2s}+2B_2^2d^{-2t}$, which uses
Theorem~\ref{thm:two-index-approx}'s route and removes the factor $m_N$ when $t\ge s$.

For a finite-parameter sieve it is therefore the output index that governs the polynomial
exponent, while the input index enters Theorem~\ref{thm:two-index-approx}, where no degree
bound is available and no finite parametrization is claimed.
\end{rem}

\section{Three explicit classes}
\label{sec:examples}

We study three explicit classes. The diagonal Gaussian class is linear and coordinatewise. The block class allows nonlinear interactions within fixed-size blocks and has matching upper and lower bounds. The two-groups class has a non-Gaussian target and is estimated on a grid.

\subsection{Conjugate Gaussian inverse problem}
\label{subsec:gaussian}
Gaussian optimal transport on separable Hilbert spaces has also been studied from an operator-theoretic perspective; see, for example, \citet{yun2025gaussian}. We consider here the commuting, strictly positive case, for which the transport map is diagonal in the common eigenbasis.

Take $\pi_0=\gamma=N(0,C_0)$ and a centered Gaussian target $N(0,C_1)$ with commuting,
strictly positive covariance operators. In their common eigenbasis the transport map has
coordinate multiplier $a_k=\sqrt{\lambda_k^{(1)}/\lambda_k^{(0)}}$. We assume $\sum_k(a_k-1)^2<\infty$, which gives finite Cameron--Martin cost and equivalence of the
two Gaussian measures. For a conjugate inverse problem the multiplier is
\[
a_k=\left(1+\lambda_k^{(0)}\kappa_k^2/\varepsilon^2\right)^{-1/2},
\]
where $\kappa_k$ are the forward-operator eigenvalues and $\varepsilon$ is the noise
level of that inverse problem. These known multipliers provide a calibration target for
the paired-data estimator; detailed conjugate calculations appear in
Section~\ref{sup:gaussian-details}.

The displacement regression is $D_{ik}=(a_k-1)Z_{ik}+\xi_{ik}$. Coordinatewise least
squares, projected onto a compact interval containing the displacement coefficients,
gives an explicit sieve estimator. On the box $\Theta_N^{(B)}$ of
\eqref{eq:widened-box}, Proposition~\ref{prop:gaussian-verify} gives, for $N>2$,
\[
\sup_{T\in\mathcal W^s_{\mathrm{diag}}(B)}
\mathbb E_T\|\widehat T_N-T\|_{L^2(\gamma;\mathcal H)}^2
\le \frac{\sigma^2d_N}{N-2}+B^2d_N^{-2s}.
\]
Together with Theorem~\ref{thm:minimax}, this identifies the expected-norm minimax rate
$N^{-s/(2s+1)}$ on the diagonal class. More generally, the exact truncation bias is
$\sum_{k>d_N}(a_k-1)^2$. If $|a_k-1|\asymp k^{-\rho}$ with $\rho>1/2$, its square root
has order $d_N^{-(\rho-1/2)}$; this tail statement also applies at the boundary of the
weighted regularity range.

\subsection{Nonlinear block-interaction class}
\label{subsec:block}

The diagonal class of Definition~\ref{def:diag-class} is linear and coordinatewise. We next allow nonlinear interactions within fixed-size blocks while keeping the map linear in the block coefficients. This gives an explicit estimator and matching upper and lower bounds.

Throughout this subsection $\pi_0=\gamma$. Write $B_j:=\{2j-1,2j\}$ for the $j$th coordinate
block, $z_{B_j}:=(z_{2j-1},z_{2j})\in\mathbb R^2$, and fix the interaction profile
\begin{equation}
\label{eq:block-psi}
  \psi(x,y):=\tfrac13[\cos x+\cos(x+y)] ,
  \qquad
  \nabla\psi(x,y)=-\tfrac13\big(\sin x+\sin(x+y),\ \sin(x+y)\big).
\end{equation}
The Hessian is
\[
\nabla^2\psi(x,y)=-\tfrac13\left[
\cos x\begin{pmatrix}1&0\\0&0\end{pmatrix}
+\cos(x+y)\begin{pmatrix}1&1\\1&1\end{pmatrix}\right].
\]
Thus $\|\nabla^2\psi\|_{\rm op}\le1$, $\|\nabla^2\psi\|_F^2\le2$, and
$\|\nabla\psi\|^2\le5/9\le1$. Set
\begin{equation}
\label{eq:block-sigma2}
\varsigma^2:=\mathbb E\|\nabla\psi(Z)\|^2
=\frac{(1-e^{-2})/2+(1-e^{-4})+e^{-1/2}-e^{-5/2}}9>0,
\quad Z\sim N(0,I_2).
\end{equation}
Here $\mathbb E[\sin X\sin(X+Y)]=(e^{-1/2}-e^{-5/2})/2$ for independent standard
normal $X,Y$. Write $\mu_\psi=(e^{-1/2}+e^{-1})/3$.

\begin{definition}
\label{def:block-class}
For $s>0$, $B>0$ and $a_0\in(0,\tfrac12]$, let
\[
  \mathcal A_s(B,a_0):=\Big\{a=(a_j)_{j\ge1}:\ \sum_{j\ge1}j^{2s}a_j^2\le B^2,\ \
  |a_j|\le a_0\ \forall j\Big\},
\]
and for $a\in\mathcal A_s(B,a_0)$ define the centered potential and the map
\[
  \phi_a(f):=\sum_{j\ge1}a_j\Big[\psi\big(\hat e_{2j-1}(f),\hat e_{2j}(f)\big)-\mu_\psi\Big],
  \qquad
  T_a:=I+\nabla_{\mathcal H}\phi_a .
\]
\end{definition}

On block $j$ the map is $z\mapsto z+a_j\nabla\psi(z)$, the gradient of a potential
with Hessian between $I_2/2$ and $3I_2/2$. A fixed orthogonal change of coordinates cannot
make this family coordinatewise. The Hessians at $(0,\pi/2)$ and $(\pi/2,-\pi/2)$ have
commutator
\[
\frac19\begin{pmatrix}0&1\\-1&0\end{pmatrix}\ne0.
\]
If a fixed rotation separated the potential into one-variable functions, all its
Hessians would be diagonal in that basis and would commute.

\begin{proposition}
\label{prop:block-population}
Let $a\in\mathcal A_s(B,a_0)$ with $a_0\le\tfrac12$. Then $\phi_a\in\mathbb D^{1,2}(\gamma)$,
the coefficient fields on block $j$ are $u_{B_j}=a_j\nabla\psi(\hat e_{B_j})$, and
\begin{equation}
\label{eq:block-coef}
\|u_{2j-1}\|_{L^2(\gamma)}^2+\|u_{2j}\|_{L^2(\gamma)}^2=\varsigma^2a_j^2,
\qquad \|T_a\|_{\mathcal W^s}^2\le2^{2s}\varsigma^2B^2.
\end{equation}
Thus $T_a\in\mathcal W^s$. The pushforward $\pi_1:=T_{a\#}\gamma$ satisfies
$\pi_1\sim\gamma$, the pair $(\gamma,\pi_1)$ satisfies
\textnormal{(A1)}--\textnormal{(A3)}, and $T_a$ is the optimal map for the
Cameron--Martin cost, so Assumption~\ref{ass:rep} holds with potential $\phi_a$.
\end{proposition}

The sieve truncates at $m$ blocks, so $d_N=2m$ and $p_N=m$. Because $\phi_a$ is linear in
$a$, the regression is a linear model with a nonlinear vector feature. Writing
$D_{i,B_j}:=(D_{i,2j-1},D_{i,2j})$ for the observed displacement \eqref{eq:Dik} on block
$j$ and $\Psi_{ij}:=\nabla\psi(Z_{i,B_j})$,
\begin{equation}
\label{eq:block-regression}
  D_{i,B_j}=a_j\,\Psi_{ij}+\xi_{i,B_j},
  \qquad i\le N,\ j\le m .
\end{equation}
The contrast \eqref{eq:def-Rhat-circ} is separable across blocks and quadratic in $a$, so
its minimizer over $[-a_0,a_0]^m$ is available in closed form,
\begin{equation}
\label{eq:block-estimator}
  \widehat a_j=\Pi_{[-a_0,a_0]}\bigg(
  \frac{\sum_{i\le N}\langle\Psi_{ij},D_{i,B_j}\rangle}{\sum_{i\le N}\|\Psi_{ij}\|^2}\bigg),
  \qquad
  \widehat T_m:=T_{\widehat a,m} ,
\end{equation}
with $\Pi$ the projection onto the interval.

\begin{theorem}
\label{thm:block-upper}
Fix $\sigma>0$ and $0<a_0\le\tfrac12$, and let $\widehat T_m$ be the estimator \eqref{eq:block-estimator}.
There is an absolute constant $c>0$ such that, for every $N\ge1$ and $m\ge1$,
\[
  \sup_{a\in\mathcal A_s(B,a_0)}
  \mathbb E_a\big\|\widehat T_m-T_a\big\|_{L^2(\gamma;\mathcal H)}^2
  \ \le\ \frac{2\sigma^2m}{N}+4a_0^2\varsigma^2m\,e^{-cN}+\varsigma^2B^2m^{-2s} .
\]
For all sufficiently large $N$, choosing $m\asymp(B^2N/\sigma^2)^{1/(2s+1)}$ gives
\[
  \sup_{a\in\mathcal A_s(B,a_0)}
  \mathbb E_a\big\|\widehat T_m-T_a\big\|_{L^2(\gamma;\mathcal H)}
  \ \lesssim\ B^{1/(2s+1)}\Big(\frac{\sigma^2}{N}\Big)^{s/(2s+1)} .
\]
\end{theorem}

\begin{theorem}
\label{thm:block-lower}
There are constants $c'=c'(s,B,\sigma,a_0)>0$ and $N_0$ such that, for all $N\ge N_0$,
\[
  \inf_{\widehat T}\ \sup_{a\in\mathcal A_s(B,a_0)}\
  \mathbb E_a\big\|\widehat T-T_a\big\|_{L^2(\gamma;\mathcal H)}
  \ \ge\ c'\Big(\frac{\sigma^2}{N}\Big)^{s/(2s+1)},
\]
the infimum over all estimators measurable in the paired sample.
\end{theorem}

Both theorems are proved in Section~\ref{app:block-proofs}. They give matching rates
on the same class in expected norm. For any nonzero coefficient sequence, the target
has a non-Gaussian block; each nonzero block also has dependent coordinates, as verified
in the population proof.

\begin{rem}
\label{rem:block-scope}
The class has fixed-size independent blocks and one scalar parameter per block.
Truncation preserves all input dependence within each retained block, so the retained
map is represented exactly. Dependence on discarded inputs requires additional
approximation control, as discussed in Section~\ref{sec:discussion}.
\end{rem}

\subsection{Continuous two-groups model}
\label{subsec:bayesian-example}

We consider a non-Gaussian two-groups model in which coordinate $k$ is drawn from the Gaussian reference with probability $1-\delta_k$ and from a heavier-tailed alternative with probability $\delta_k$. Such null/alternative mixtures are standard in empirical-Bayes multiple testing \citep{efron2008microarrays} and sparse modeling \citep{georgemcculloch1993}. We assume
\[
\delta_k=c_kk^{-2\rho},\qquad \rho>\tfrac12,\qquad c_k\in[c_-,c_+]\subset(0,\tfrac13],
\]
where $\rho,c_-,c_+$ are fixed and the sequence $(c_k)$ is unknown. The transport map is therefore unknown. The bounds below use only the fixed class limits $c_-,c_+$ and the decay exponent $\rho$. The sieve ultimately estimates $\delta_k$, the coordinatewise prior probability of a non-null signal.

Let $\nu$ be the standard Laplace measure, $d\nu(y)=\tfrac12e^{-|y|}dy$, with CDF
$F_\nu$, a distribution with heavier tails than the $N(0,1)$ coordinates of $\gamma$, in
the same spirit as classical two-groups models with a diffuse, heavy-tailed slab
\citep{georgemcculloch1993}. For $\theta\in[0,1]$ define the mixture CDF and its
monotone rearrangement against $\Phi$,
\[
F_\theta:=(1-\theta)\Phi+\theta F_\nu,
\qquad
M_\theta:=F_\theta^{-1}\circ\Phi,
\]
so that, if $X\sim N(0,1)$, then $M_\theta(X)\sim(1-\theta)N(0,1)+\theta\,\nu$ by the
probability integral transform. Both $\Phi$ and $F_\nu$ are symmetric about $0$, so
$F_\theta$ is symmetric and $M_\theta$ is odd, with $M_\theta(0)=0$; at $\theta=0$,
$F_0=\Phi$ and $M_0=\mathrm{id}$, and at $\theta=1$, $M_1=F_\nu^{-1}\circ\Phi$ is exactly
the Gaussian-to-Laplace rearrangement used to construct a heavy-tailed departure from
Gaussianity.

Define the potential by
\[
\phi(f):=\sum_{k\ge1}\bar\Xi_{\delta_k}\big(\hat e_k(f)\big),
\qquad
\Xi_\theta(x):=\int_0^x\big(M_\theta(t)-t\big)\,dt,
\qquad
\bar\Xi_\theta:=\Xi_\theta-\mathbb E_\Phi[\Xi_\theta],
\]
the centering constant $\mathbb E_\Phi[\Xi_\theta]:=\mathbb E_{X\sim N(0,1)}[\Xi_\theta(X)]$
finite because Lemma~\ref{lem:mtheta-growth} in Section~\ref{app:two-groups-lemmas} bounds $|M_\theta(x)-x|$ by $C_1(1+x^2+\log(1/\theta))$, so that $\Xi_\theta$ has at most cubic growth, which is integrable against a Gaussian tail. The same bound gives $M_\theta-\mathrm{id}$
at most quadratic growth for each fixed $\theta$, so the cylindrical potential built
from it meets the polynomial-growth requirement of Definition~\ref{def:cylindrical}.

Proposition~\ref{prop:two-groups-population} verifies that $\phi\in\mathbb D^{1,2}(\gamma)$ and that the pushforward $\pi_1:=T_\#\gamma=\bigotimes_k\pi_{1,k}$, with $\pi_{1,k}:=(1-\delta_k)N(0,1)+\delta_k\nu$, satisfies $\pi_1\sim\gamma$. Both follow from $\sum_k\delta_k<\infty$ when $\rho>\tfrac12$. The convex coordinate potential is $x^2/2+\Xi_\theta(x)$, whose derivative $M_\theta$ is nondecreasing. The chain rule \eqref{eq:chain} then gives
\[
T(f)=f+\nabla_{\mathcal H}\phi(f)=\sum_{k\ge1}M_{\delta_k}\big(\hat e_k(f)\big)\,e_k,
\qquad
u_k(f)=M_{\delta_k}\big(\hat e_k(f)\big)-\hat e_k(f).
\]
Since $M_{\delta_k}$ is the monotone, hence $W_2$-optimal, coupling of $N(0,1)$ to
$\pi_{1,k}$, the mixture bound of Lemma~\ref{lem:mixture-w2} applies with $P=N(0,1)$,
$Q=\nu$, and $\theta=\delta_k$, and gives the exact, nonasymptotic bound
\begin{equation}
\label{eq:two-groups-bias}
\|u_k\|_{L^2(\pi_0)}^2=W_2\big(N(0,1),(1-\delta_k)N(0,1)+\delta_k\nu\big)^2
\ \le\ \delta_k\,W_0^2,
\qquad
W_0^2:=W_2\big(N(0,1),\nu\big)^2<\infty,
\end{equation}

By \eqref{eq:two-groups-bias},
$\|T\|_{\mathcal W^s}^2\le c_+W_0^2\sum_kk^{2s-2\rho}$, so
$T\in\mathcal W^s$ for every $s<\rho-\tfrac12$. At the boundary
$s=\rho-\tfrac12$, \eqref{eq:two-groups-bias} gives only an upper bound; see
Remark~\ref{rem:supremal}. For the rate calculation we use the tail bound
\begin{equation}
\label{eq:two-groups-tail}
\|T_d-T\|^2_{L^2(\pi_0;\mathcal H)}
\le c_+W_0^2\sum_{k>d}k^{-2\rho}
\lesssim d^{-2(\rho-1/2)}.
\end{equation}
Thus $\rho$ determines the approximation rate in this class.

For the retained coordinates use
\[
g_\theta(z)=\sum_{k\le d_N}\Xi_{\theta_k}(z_k),\qquad
T_{\theta,N}(f)=f+\sum_{k\le d_N}(M_{\theta_k}(\hat e_k(f))-\hat e_k(f))e_k.
\]
The retained truth is represented at $\theta_k=\delta_k$. Set
$\ell(\eta)=e^\eta/(1+e^\eta)$, $\eta_k=\operatorname{logit}(\theta_k)$ and
$\beta_k=\eta_k+2\rho\log k$. Then
\begin{equation}
\label{eq:beta-star-bound}
\beta_k^*=\log c_k-\log(1-c_kk^{-2\rho})
\in[\log c_-,\log(c_+/(1-c_+))].
\end{equation}
Choose the deterministic class-dependent box
\begin{equation}
\label{eq:def-BN}
B_-:=\log c_- -1,\quad B_+:=\log\frac{c_+}{1-c_+}+1,\quad
\mathcal B_N=[B_-,B_+]^{d_N},\quad
\Theta_N=\{\eta:\beta(\eta)\in\mathcal B_N\}.
\end{equation}
Every true centered coordinate has margin at least $1$ in this box. Writing
$r_0=B_+-B_-$, we have $|\eta_k-\eta_k^*|\le r_0$ for every candidate and coordinate.
The box uses the class bounds, not the individual unknown $c_k$.

By the chain rule, the sieve's Jacobian entry in the new coordinate is the score
multiplier
\[
q_\eta(x):=\ell'(\eta)\,\partial_\theta M_\theta(x)\big|_{\theta=\ell(\eta)}
=\theta(1-\theta)\,\partial_\theta M_\theta(x),
\qquad
\partial_\eta\big[M_{\ell(\eta)}(x)\big]=q_\eta(x).
\]
Write $\widetilde R_N(\eta):=R_N(\ell(\eta))$ and
$\widehat{\widetilde R}_N^\circ(\eta):=\widehat R_N^\circ(\ell(\eta))$, with $\ell$
applied coordinatewise, for the population risk and the finite empirical contrast read
in the log-odds coordinate.

\begin{proposition}
\label{prop:log-odds-props}
Let $L_0$ be the constant of Lemma~\ref{lem:theta-sensitivity} in
Section~\ref{app:two-groups-lemmas}. The score multiplier is uniformly bounded,
\[
  \sup_{\eta\in\mathbb R}\ \sup_{x\in\mathbb R}\ |q_\eta(x)|\ \le\ L_0 .
\]
Consequently $\eta\mapsto M_{\ell(\eta)}(x)$ is $L_0$-Lipschitz uniformly in $x$, and
for every $\eta,\eta'$,
\begin{equation}
\label{eq:logodds-sup-lip}
  \sup_{x\in\mathbb R}\big|M_{\ell(\eta)}(x)-M_{\ell(\eta')}(x)\big|
  \ \le\ L_0\,|\eta-\eta'| .
\end{equation}
\end{proposition}

\begin{proof}
$|q_\eta(x)|=\theta(1-\theta)|\partial_\theta M_\theta(x)|\le(1-\theta)L_0\le L_0$
by Lemma~\ref{lem:theta-sensitivity}, and \eqref{eq:logodds-sup-lip} follows by
integrating $\partial_\eta M_{\ell(\eta)}=q_\eta$.
\end{proof}

The bound \eqref{eq:logodds-sup-lip} is uniform in $x$ and its constant is independent of $d_N$ and the coordinate index. Together with the fixed coordinatewise radius of $\Theta_N$, it gives the envelope
\begin{equation}
\label{eq:g-envelope}
  \sup_{\eta\in\Theta_N}\ \max_{k\le d_N}\ \sup_{x\in\mathbb R}
  \big|M_{\ell(\eta_k)}(x)-M_{\delta_k}(x)\big|\ \le\ L_0r_0\ =:\ K .
\end{equation}
Although $M_\theta$ is unbounded, the difference between a candidate map and the truth is uniformly bounded over the parameter space and coordinates. This gives a direct prediction-risk analysis.

\subsubsection{Direct bound on the prediction risk}
\label{subsec:logodds-direct}

We estimate the map by minimizing the contrast \eqref{eq:def-Rhat-circ} over a
discretization of $\Theta_N$. Fix a mesh width $h>0$, write
$a_k=B_--2\rho\log k$ and $b_k=B_+-2\rho\log k$, and define
\begin{equation}
\label{eq:def-grid}
\mathcal G_k=\{a_k+jh: j=0,\ldots,\lfloor(b_k-a_k)/h\rfloor\}\cup\{b_k\},
\qquad \mathcal G_N=\prod_{k\le d_N}\mathcal G_k,
\end{equation}
so that $|\mathcal G_k|\le r_0/h+2$ for every $k$, and set
\begin{equation}
\label{eq:def-grid-erm}
  \widehat\eta_N\in\arg\min_{\eta\in\mathcal G_N}\widehat{\widetilde R}_N^\circ(\eta),
  \qquad
  \widehat T_N:=T_{\ell(\widehat\eta_N),N}.
\end{equation}
Since $\widehat{\widetilde R}_N^\circ$ is separable across coordinates, \eqref{eq:def-grid-erm} is computed coordinate by coordinate by evaluating $|\mathcal G_k|$ scalar criteria. The grid is fixed before observing the data.

The resulting risk bound is stated below. Write
\begin{equation}
\label{eq:def-Drho}
  \mathcal D_\rho(c_-,c_+):=\big\{(\delta_k)_{k\ge1}:\ \delta_k=c_kk^{-2\rho},\
  c_k\in[c_-,c_+]\big\}
\end{equation}
for the model class, with $\rho>\tfrac12$ and $[c_-,c_+]\subset(0,\tfrac13]$ fixed, and
$T_\delta$ for the transport map with inclusion probabilities $\delta$.

\begin{theorem}
\label{thm:logodds-fast}
Fix $\sigma>0$, $\rho>1/2$ and $0<c_-<c_+\le1/3$. For $0<h\le1$, let
$\widehat T_N$ be the grid estimator \eqref{eq:def-grid-erm} on the box
\eqref{eq:def-BN}. There is $C$, depending only on $\sigma,r_0,L_0$, such that
for $N\ge2$ and every integer $d_N\ge1$,
\begin{equation}
\label{eq:logodds-fast-risk}
\sup_{\delta\in\mathcal D_\rho(c_-,c_+)}
\mathbb E_\delta\|\widehat T_N-T_\delta\|_{L^2(\gamma;\mathcal H)}^2
\le C d_N\left[h^2+\frac{1+\log(2+r_0/h)}{N}\right]
 +\frac{c_+W_0^2}{2\rho-1}d_N^{-(2\rho-1)}.
\end{equation}
In particular, for fixed $A\ge1$ and $C_0>0$, any mesh satisfying
\[N^{-A}\le h\le\min\{1,C_0\sqrt{\log N/N}\}\]
gives an estimation term
$C' d_N\log N/N$, where $C'$ may also depend on $A,C_0$.
Choosing $d_N\asymp(N/\log N)^{1/(2\rho)}$ gives
\begin{equation}
\label{eq:logodds-fast-rate}
\sup_{\delta\in\mathcal D_\rho(c_-,c_+)}
\mathbb E_\delta\|\widehat T_N-T_\delta\|_{L^2(\gamma;\mathcal H)}
\lesssim(N/\log N)^{-(\rho-1/2)/(2\rho)}.
\end{equation}
\end{theorem}

For coordinate $k$, the excess loss is
\begin{equation}
\label{eq:excess-loss}
L_\eta(Z,\xi)=g_\eta(Z)^2-2\xi g_\eta(Z),\qquad
 g_\eta=M_{\ell(\eta)}-M_{\delta_k}.
\end{equation}
The envelope \eqref{eq:g-envelope} yields
\begin{equation}
\label{eq:var-mean}
\mathbb E L_\eta=\|g_\eta\|_2^2,\qquad
\operatorname{Var}(L_\eta)\le(K^2+4\sigma^2)\|g_\eta\|_2^2.
\end{equation}
Section~\ref{app:proof-logodds-fast} proves the corresponding exponential bound and applies it to each coordinate grid. Writing $s_{\rm tail}=\rho-1/2$, the resulting exponent is $s_{\rm tail}/(2s_{\rm tail}+1)$ and follows from the tail bound \eqref{eq:two-groups-tail}.

\begin{rem}
\label{rem:grid-vs-continuous}
The theorem concerns a global minimizer over the specified deterministic grid.
It also covers a grid-valued estimator $\widetilde\eta\in\mathcal G_N$ whose normalized
empirical contrast is within a deterministic $\varepsilon_N\ge0$ of the grid minimum,
with an additional $C\varepsilon_N$ in the risk bound; see the proof.
\end{rem}

\begin{rem}
\label{rem:no-lower-bound-nongaussian}
The bound is uniform over $\mathcal D_\rho(c_-,c_+)$. The diagonal lower bound of
Theorem~\ref{thm:minimax} concerns a different class, so the exponent here is an attained
upper bound.
\end{rem}

\section{Numerical study}
\label{sec:numerical-illustration}
\label{subsec:num-rate}

We evaluate the grid estimator at $\rho=0.75,1,1.5$, with $\sigma=0.3$ and
$c_k=1/3$. The class bounds are $c_-=1/4$ and $c_+=1/3$.
We take $d_N$ to be the nearest integer to $(N/\log N)^{1/(2\rho)}$, with minimum $2$,
and use $h_N=\sqrt{\log N/N}$ in centered log-odds coordinates. Both endpoints of each
coordinate interval are included, and every grid value is compared.
The experiment comprises $535$ replicates over $17$ settings, with $N$ ranging from $250$
to $4000$ for $\rho=0.75$, and from $500$ to $16000$ for $\rho=1,1.5$.

Population risks are evaluated through \eqref{eq:exact-split}. Quantile evaluation uses
symmetry and log-survival probabilities. Retained-coordinate risks use two quadrature
orders, with adaptive integration when their discrepancy exceeds the specified tolerance.
The coordinate tail is bounded using geometric integer blocks and an analytic infinite
remainder; the displayed bias is the midpoint of the numerically evaluated block bounds.
The algorithm and error accounting are described in Section~\ref{sup:numerical-integration},
and Table~\ref{tab:grid-details} lists the computed values setting by setting.
Figure~\ref{fig:grid-rates} shows the result.

\begin{figure}[H]
\centering
\includegraphics[width=\textwidth]{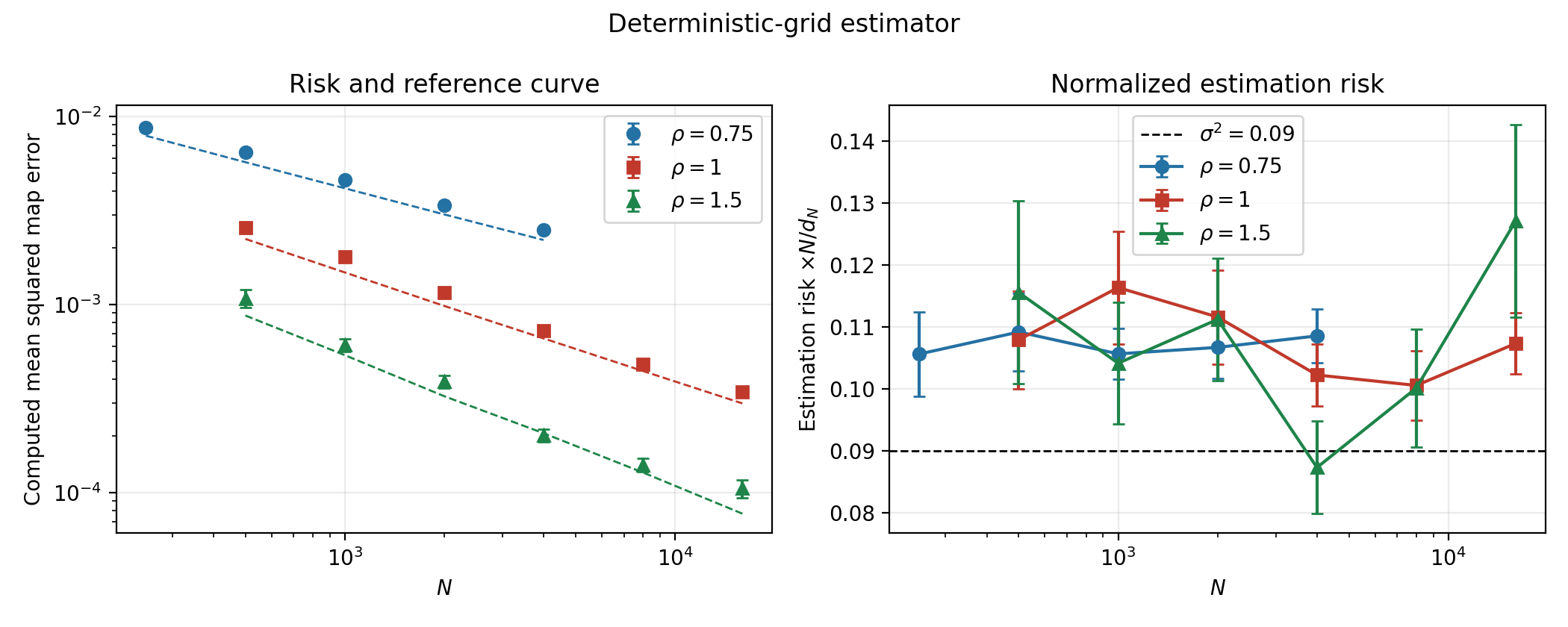}
\caption{Deterministic-grid estimator. Left: computed mean squared map error and
the reference $\sigma^2d_N/N$ plus the computed bias (dashed). Right: normalized
estimation risk and the reference level $\sigma^2=0.09$. Error bars show one Monte
Carlo standard error; numerical integration and tail errors are recorded separately.}
\label{fig:grid-rates}
\end{figure}

Observed and reference log-risk curves are fitted using identical weighted least squares
weights, the inverse delta-method variance of the observed log risk. A bootstrap
resamples replicates within each $N$, holding the original weights and computed biases
fixed. Table~\ref{tab:grid-slopes} reports the finite-range slopes and the 95\% bootstrap
interval for their difference, based on $10{,}000$ resamples.

\begin{table}[H]
\centering
\begin{tabular}{cccc}
\hline
$\rho$ & Observed slope & Reference slope & Difference: 95\% interval\\
\hline
$0.75$ & $-0.4539$ & $-0.4595$ & $[-0.0233,0.0350]$\\
$1.0$ & $-0.5908$ & $-0.5795$ & $[-0.0422,0.0205]$\\
$1.5$ & $-0.6983$ & $-0.6941$ & $[-0.0697,0.0608]$\\
\hline
\end{tabular}
\caption{Finite-range slope comparison for the grid estimator.}
\label{tab:grid-slopes}
\end{table}

The fitted slopes follow the reference slopes across $\rho$, and all three intervals for their difference contain zero. The reference curve uses $\sigma^2d_N/N$ plus the computed bias, while the theoretical upper bound has an additional logarithmic factor. Endpoint frequencies and the full pointwise results are reported in the Supplementary Material.

The largest numerically evaluated tail-interval half-width is 0.52\% of
the computed total risk, and the largest change between the two tail quadrature orders
is 0.007\% of that risk. These diagnostics are separate from the Monte Carlo
error bars.


\section{Discussion}
\label{sec:discussion}

Paired observations turn transport-map estimation into a regression problem with Cameron--Martin loss. The transport structure links the regression function to the source and target distributions, while the infinite-dimensional setting introduces an additional approximation problem. In particular, decay of the output coordinates of $T-I$ does not by itself control the dependence of retained outputs on discarded input coordinates. Section~\ref{sec:input} gives two ways to handle this issue. For general Sobolev potentials, weighted derivative regularity together with a conditional Gaussian Poincar\'e inequality controls the discarded-input dependence. For bounded-chaos potentials, the degree restriction provides the required control.

The explicit model classes show how these conditions lead to concrete rates. The diagonal Gaussian class gives a case in which the transport coefficients and approximation error can be calculated directly. The nonlinear block-interaction class shows that the rate $N^{-s/(2s+1)}$ also holds beyond diagonal linear maps, with a matching lower bound. For bounded-chaos potentials with interaction order $q$, the Hermite sieve gives the rate $N^{-s/(q+2s)}$, which is minimax in the additive case. The Gaussian--Laplace mixture gives a non-Gaussian example with a direct prediction-risk bound.

Several questions remain open. Matching lower bounds are not available for $\mathcal C_{q,m}^{s}$ when $q\ge2$ or for the growing-degree classes $\mathcal K_q^{s,t,\varpi}$. It is also unclear whether least-squares estimation can reduce the degree dependence that appears in the variance of the projection estimator. For the mixture class, sharper two-sided bounds on $W_2^2(\mu_\theta,\mu_{\theta'})$ would help determine whether the logarithmic factor in the current upper bound is necessary.

The paired-observation framework may also be extended in several directions. One is adaptive selection of the sieve dimension, chaos degree, and interaction order. Another is stochastic transport, which allows transport mechanisms beyond deterministic Monge maps \citep{nietert2025estimation}. These extensions would broaden the use of the framework for transport problems on function spaces, including those arising in Bayesian inverse problems and related infinite-dimensional models.

\medskip
\bibliographystyle{plainnat}
\bibliography{reference}

\clearpage
\begin{center}
{\Large\bfseries Supplementary Material}\\[0.5ex]
{\large for ``Sieve Estimation of Optimal Transport Maps from Paired Data in Gaussian Spaces''}
\end{center}
\medskip

\renewcommand{\thesection}{S\arabic{section}}
\setcounter{section}{0}
\renewcommand{\theHsection}{sup.\arabic{section}}
\renewcommand{\theequation}{S\arabic{equation}}
\setcounter{equation}{0}
\renewcommand{\theHequation}{sup.\arabic{equation}}

This supplement gives the verification arguments, proofs, additional calculations, and numerical details for the main manuscript. Results, equations, and sections from the main manuscript are referenced without the prefix ``S''; labels with ``S'' refer to this supplement.

\section{Paired-noise contrast}
\label{sup:foundations}

This section shows that the finite contrast \eqref{eq:def-Rhat-circ} is an unbiased estimate of the excess-risk difference.

Write $u_k(f)=\langle T(f)-f,e_k\rangle_{\mathcal H}$ as in Section~\ref{subsec:regularity} and $b_{\theta,k}(z)=\partial_kg_\theta(z)$, so that the $k$th Cameron--Martin coordinate of the sieve displacement at $f_0^{(i)}$ is $b_{\theta,k}(Z_i)$. Since $e_k=\sqrt{\lambda_k}\varphi_k$, the observed displacement coordinate \eqref{eq:Dik} decomposes, under Assumption~\ref{ass:data}, as
\begin{equation}
\label{eq:D-decomp}
  D_{ik}=\big\langle T(f_0^{(i)})-f_0^{(i)},e_k\big\rangle_{\mathcal H}+\xi^{(i)}_k
  =u_k(f_0^{(i)})+\xi^{(i)}_k ,
\end{equation}
with $\xi^{(i)}_k$ the i.i.d.\ $N(0,\sigma^2)$ noise coordinates of
\eqref{eq:noise-coords}, independent of the design. Both terms are ordinary real random
variables. Taking expectations in \eqref{eq:def-Rhat-circ} term by term, the cross terms
in $\xi^{(i)}_k$ vanish and, for every $\theta,\theta_0\in\Theta_N$,
\begin{align*}
\mathbb E\big[\widehat R_N^\circ(\theta)\big]
&=\sum_{k=1}^{d_N}\mathbb E_{\pi_0}\Big[b_{\theta,k}(Z)^2-b_{\theta_0,k}(Z)^2
-2u_k\big(b_{\theta,k}(Z)-b_{\theta_0,k}(Z)\big)\Big]\\
&=\sum_{k=1}^{d_N}\mathbb E_{\pi_0}\big[(b_{\theta,k}(Z)-u_k)^2\big]
-\sum_{k=1}^{d_N}\mathbb E_{\pi_0}\big[(b_{\theta_0,k}(Z)-u_k)^2\big]
=R_N(\theta)-R_N(\theta_0),
\end{align*}
the last equality because the tail coordinates $k>d_N$ contribute
$\sum_{k>d_N}\|u_k\|_{L^2(\pi_0)}^2$ to both $R_N(\theta)$ and $R_N(\theta_0)$ and cancel.
So $\theta_N^*$ minimizes $\mathbb E[\widehat R_N^\circ]$ over $\Theta_N$, and every
empirical gradient and Hessian below is a derivative of the finite function
\eqref{eq:def-Rhat-circ}.

\section{Verification of the general sieve conditions}
\label{sup:verification}

This section gives conditions on the cylindrical sieve that imply Assumptions~\ref{ass:invertible}--\ref{ass:localization}, together with the corresponding proofs.

\subsection{Limit on uniform cylindrical approximation}
\label{app:expressivity}

The following transport family lies in a fixed $\mathcal W^s$ ball but has no uniform $d^{-s}$ cylindrical approximation.

Take $\pi_0=\gamma$, let $m=d+1$, put $n:=m^{s/2}\ge1$ and $0<a<1/4$, and set
$\phi_m(z)=an^{-2}\sin(nz_1)\sin(z_m)$. The Hessian of $\phi_m$ has operator norm at most
$2a<\tfrac12$, so $I+\nabla^2\phi_m\succ0$ and $T_m:=I+\nabla_{\mathcal H}\phi_m$ is a
monotone gradient map altering two coordinates. Its nonzero displacement fields are
$u_1=an^{-1}\cos(nz_1)\sin(z_m)$ and $u_m=an^{-2}\sin(nz_1)\cos(z_m)$, whence
$\|T_m\|_{\mathcal W^s}^2\le a^2n^{-2}+a^2m^{2s}n^{-4}\le2a^2$, so the whole family lies in
one fixed $\mathcal W^s$ ball. Yet any predictor using only the first $d=m-1$ coordinates
incurs, for $u_1$, at least its conditional variance,
\[
\inf_{v}\mathbb E\big|u_1-v(Z_1,\dots,Z_d)\big|^2
=\frac{a^2}{n^2}\,\mathbb E\big[\cos^2(nZ_1)\big]\,\mathbb E\big[\sin^2(Z_m)\big]
\ \ge\ c\,a^2m^{-s},
\]
so the root error is at least $cm^{-s/2}$, not $O(m^{-s})$. Membership in
$\mathcal W^s$ therefore does not yield a uniform $d^{-s}$ cylindrical approximation
across the class.

\subsection{Primitive verification conditions}
\label{app:verify}

Assumptions~\ref{ass:invertible}--\ref{ass:localization} are stated at the level of the
abstract parametric risk $R_N$, with $\mu_N$, $L_N$, and the rate $O_p(\sqrt{p_N/N})$
left unconnected to the concrete sieve family $\{g_\theta\}$ of
Definition~\ref{def:cylindrical}. We give primitive conditions directly on
$(g_\theta,\pi_0)$ that verify Assumptions~\ref{ass:invertible}--\ref{ass:grad-conc}.
Assumptions~\ref{ass:invertible} and~\ref{ass:lipschitz-map} follow from the envelope
conditions (B1)--(B3) in their $\pi_0$-a.s.\ form.
Assumption~\ref{ass:grad-conc} also requires the uniform empirical-process condition (B4)
below: (B1)--(B3) bound the empirical and population averages of
$\partial_\theta S(\theta,f_0)$ separately, but leave uncontrolled the rate at which their
difference concentrates, which (B4) supplies.

Recall $T_{\theta,N}(f)=f+\sum_{k=1}^{d_N}\partial_kg_\theta(\Phi_N(f))e_k$, and write
$J_\theta(z):=\partial_\theta\nabla_zg_\theta(z)\in\mathbb R^{d_N\times p_N}$ for the
$\theta$-Jacobian of the gradient map and $r_k(\theta,f):=\partial_kg_\theta(\Phi_N(f))-u_k(f)$
for the coordinatewise residual against the truth, with $u_k$ as in
Section~\ref{subsec:regularity}. Since $u_k$ does not depend on $\theta$,
$\partial_\theta r_k=\partial_\theta\partial_kg_\theta$, and, splitting off the tail
coordinates $k>d_N$, on which every sieve map agrees with the identity,
\[
R_N(\theta)=\sum_{k=1}^{d_N}\mathbb E_{\pi_0}\big[r_k(\theta,f_0)^2\big]
+\|T-T^{(d_N)}\|_{L^2(\pi_0;\mathcal H)}^2 .
\]
Write $S(\theta,f):=2\sum_kr_k(\theta,f)J_\theta(\Phi_N(f))_{k,\cdot}$, so that
$\nabla R_N(\theta)=\mathbb E_{\pi_0}[S(\theta,f_0)]$ and $\nabla\widehat R_N^\circ(\theta)
=\widehat S_N(\theta):=\tfrac1N\sum_iS(\theta,f_0^{(i)})$ away from the noise term
isolated in the proof of Proposition~\ref{prop:verify}(c). Differentiating,
\[
\partial_\theta S(\theta,f)=2J_\theta(\Phi_N(f))^\top J_\theta(\Phi_N(f))
+2\sum_{k=1}^{d_N}r_k(\theta,f)\,\partial_\theta^2\partial_kg_\theta(\Phi_N(f)),
\]
an empirical-process quantity that the pointwise envelopes (B1)--(B3) do not control
uniformly on their own.

\begin{ass}
\label{ass:sieve-primitive}
There is a Euclidean ball $\mathcal N_N=\{\theta:\|\theta-\theta_N^*\|\le r_N\}\subset\Theta_N$
and a set $\mathcal E_N\subset\mathcal F$ with $N\,\pi_0(\mathcal F\setminus\mathcal E_N)\to0$
such that:
\begin{enumerate}
\item[(B1)] for every $f\in\mathcal E_N$, $\theta\mapsto\nabla_zg_\theta(\Phi_N(f))$ is
twice continuously differentiable on $\mathcal N_N$, with
$\sup_{\theta\in\mathcal N_N}\|\partial_\theta^2\partial_kg_\theta(\Phi_N(f))\|_{\rm op}\le M_N'$
for every $k=1,\ldots,d_N$;
\item[(B2)] the Gram matrix $G_N(\theta):=\mathbb E_{\pi_0}[J_\theta(\Phi_N(f_0))^\top J_\theta(\Phi_N(f_0))]$
satisfies $v^\top G_N(\theta)v\ge\underline\mu_N\|v\|^2$ for all $v\in\mathbb R^{p_N}$,
$\theta\in\mathcal N_N$, and
\[
\sup_{\theta\in\mathcal N_N}\ \sup_{\|v\|=1}\ \Big|\mathbb E_{\pi_0}\Big[\sum_{k=1}^{d_N}
r_k(\theta,f_0)\,v^\top\partial_\theta^2\partial_kg_\theta(\Phi_N(f_0))\,v\Big]\Big|
\ \le\ \tfrac12\underline\mu_N
\]
the left side is $0$ whenever $g_\theta$ is affine in $\theta$, as
for the quadratic-potential sieve of Section~\ref{subsec:gaussian}, where
$\partial_\theta^2\partial_kg_\theta\equiv0$;
\item[(B3)] for every $f\in\mathcal E_N$, $\sup_{\theta\in\mathcal N_N}\|J_\theta(\Phi_N(f))\|_{\rm op}\le M_N$;
\item[(B4)] there is $b_N(\cdot)$ such that, for every $\delta\in(0,1)$, with
probability $\ge1-\delta$,
\[
\sup_{\theta\in\mathcal N_N}\Big\|\big(\mathbb P_N-\mathbb P\big)\big[\partial_\theta
S(\theta,f_0)\big]\Big\|_{\rm op}\ \le\ b_N(\delta),
\]
a uniform empirical-process bound on the stochastic Hessian fluctuation. It is separate from the pointwise envelopes (B1)--(B3), which do not by themselves imply an $N^{-1/2}$ concentration rate. Each worked example below either verifies the required concentration directly or replaces this condition by model-specific structure.
\end{enumerate}
\end{ass}

When $J_\theta$ and its curvature are uniformly bounded in $f$, we may take $\mathcal E_N=\mathcal F$, as in the log-odds parametrization of Section~\ref{subsec:bayesian-example}. For unbounded Jacobians, such as the diagonal quadratic potential of Section~\ref{subsec:gaussian}, we use a subset $\mathcal E_N\subset\mathcal F$. A union bound gives
\begin{equation}
\label{eq:union-bound-event}
\Pr\big[f_0^{(i)}\in\mathcal E_N\ \text{for every }i\le N\big]
\ \ge\ 1-N\,\pi_0(\mathcal F\setminus\mathcal E_N)\ \to\ 1.
\end{equation}
The samplewise verification is carried out on the event~\eqref{eq:union-bound-event}, while population moment bounds use the $\pi_0$-a.s. version of (B3).

\begin{proposition}
\label{prop:verify}
Suppose Assumption~\ref{ass:sieve-primitive} holds with (B1) and (B3) in their
$\pi_0$-a.s.\ form, and let $R_N(\theta_N^*)\le C_a^2d_N^{-2s}$ as in
Theorem~\ref{thm:approx}. Then:
\begin{enumerate}
\item[(a)] Assumption~\ref{ass:invertible} holds, and the quadratic growth
\eqref{eq:quad-growth} holds with $\mu_N=\underline\mu_N$.
\item[(b)] Assumption~\ref{ass:lipschitz-map} holds with $L_N=M_N$.
\item[(c)] If, in addition,
\begin{align*}
r_NM_N'd_N&=O\!\left(M_N\sqrt{p_N/N}\right),\\
r_Nb_N(\delta)&=O\!\left(M_N\sqrt{p_N/(N\delta)}\right),
\end{align*}
then Assumption~\ref{ass:grad-conc} holds with $\zeta_N=0$. The empirical-process condition
(B4) supplies the second display; the envelopes (B1)--(B3) alone do not imply the
required uniform $N^{-1/2}$ concentration.
\end{enumerate}
\end{proposition}

Section~\ref{app:proof-verify} proves parts (a)--(c). The $\pi_0$-a.s. versions of (B1) and (B3) are used for the population Hessian and moment bounds. Assumption~\ref{ass:localization} is verified separately for the Gaussian example. The two-groups example uses the direct prediction-risk bound and does not invoke Assumption~\ref{ass:localization}.

\subsection{Proof of Proposition~\ref{prop:verify}} \label{app:proof-verify}

\begin{proof}
In this proof, (B3) is used in its a.s.\ form, $\mathcal E_N=\mathcal F$ up to a
$\pi_0$-null set. Hence the population integrals in parts (b) and (c) are defined on all
of $\mathcal F$ and $\zeta_N=0$. Part (c) controls four random terms with probability at
least $1-\delta/4$ each; a union bound gives probability at least $1-\delta$.

For (a), differentiating $R_N(\theta)=\sum_k\mathbb E_{\pi_0}[r_k(\theta,f_0)^2]+\mathrm{const}$
twice in $\theta$ gives $H_{R_N}(\theta)=2G_N(\theta)+2C_N(\theta)$, where
$C_N(\theta)_{jl}:=\mathbb E_{\pi_0}\big[\sum_kr_k(\theta,f_0)\,\partial_{\theta_j}\partial_{\theta_l}\partial_kg_\theta(\Phi_N(f_0))\big]$.
For unit $v$, (B2) gives $v^\top H_{R_N}(\theta)v\ge2\underline\mu_N-2\cdot\tfrac12\underline\mu_N=\underline\mu_N$,
so $H_{R_N}(\theta)\succeq\underline\mu_N I$ throughout $\mathcal N_N$, which is exactly the
local strong convexity underlying \eqref{eq:quad-growth} with $\mu_N=\underline\mu_N$, and in particular
$H_{R_N}(\theta_N^*)$ is invertible, giving Assumption~\ref{ass:invertible}. Interiority of
$\theta_N^*$ in $\Theta_N$ is as noted in Remark~\ref{rem:parameter-space}.

For (b), by the fundamental theorem of calculus along the segment from $\theta'$ to
$\theta$ and Jensen's inequality, $\partial_kg_\theta(z)-\partial_kg_{\theta'}(z)
=\int_0^1J_{\theta_t}(z)_{k,\cdot}(\theta-\theta')\,dt$, so, for $\pi_0$-a.e.\ $f$ (writing $z=\Phi_N(f)$),
\[
\sum_{k=1}^{d_N}\big(\partial_kg_\theta(z)-\partial_kg_{\theta'}(z)\big)^2
\le\int_0^1\big\|J_{\theta_t}(z)(\theta-\theta')\big\|_2^2\,dt
\le M_N^2\|\theta-\theta'\|^2
\]
by (B3). Taking $\pi_0$-expectation and square roots gives
$\|T_{\theta,N}-T_{\theta',N}\|_{L^2(\pi_0;\mathcal H)}\le M_N\|\theta-\theta'\|$, i.e.\
Assumption~\ref{ass:lipschitz-map} with $L_N=M_N$.

For (c), write $\nabla\widehat R_N^\circ(\theta)-\nabla R_N(\theta)=A_N(\theta)-B_N(\theta)$,
where
\[
B_N(\theta):=\frac2N\sum_{i=1}^NJ_\theta(\Phi_N(f_0^{(i)}))^\top\xi_{1:d_N}^{(i)},
\qquad
\xi_{1:d_N}^{(i)}:=\big(\xi_1^{(i)},\ldots,\xi_{d_N}^{(i)}\big),
\]
collects the noise coordinates \eqref{eq:noise-coords}, and $A_N(\theta):=\widehat S_N(\theta)-\mathbb E_{\pi_0}[S(\theta,f_0)]$
is the centered empirical average of $S(\theta,f):=2\sum_kr_k(\theta,f)J_\theta(\Phi_N(f))_{k,\cdot}$.

Fix $\delta\in(0,1)$.

For $B_N$ at the oracle: conditionally on the design, $B_N(\theta_N^*)\sim N(0,\Sigma_N)$
with $\mathrm{tr}\,\Sigma_N=\tfrac{4\sigma^2}{N^2}\sum_i\|J_{\theta_N^*}(\Phi_N(f_0^{(i)}))\|_F^2
\le\tfrac{4\sigma^2M_N^2p_N}N$ by (B3), using $\|J\|_F^2\le p_N\|J\|_{\rm op}^2$, so
$\mathbb E\|B_N(\theta_N^*)\|^2\le4\sigma^2M_N^2p_N/N$. By Markov's inequality applied to
the squared norm, with probability $\ge1-\delta/4$,
\begin{equation}
\label{eq:bn-oracle-delta}
\|B_N(\theta_N^*)\|\ \le\ 2\sigma M_N\sqrt{\tfrac{4p_N}{N\delta}}.
\end{equation}
Because $B_N(\theta_N^*)$ is conditionally Gaussian, Gaussian concentration would
sharpen the $\delta$-dependence here to $\sqrt{\log(1/\delta)}$. We keep the cruder
bound~\eqref{eq:bn-oracle-delta}, since the next step is not obviously sub-Gaussian and a
single uniformly achievable $\delta$-rate is preferable for combining the two.

For $B_N$ across $\mathcal N_N$: by (B1), $\partial_\theta B_N(\theta)_{jl}
=\tfrac2N\sum_i\sum_k\xi_k^{(i)}\partial_{\theta_j}\partial_{\theta_l}\partial_kg_\theta(\Phi_N(f_0^{(i)}))$,
so $\|\partial_\theta B_N(\theta)\|_{\rm op}\le\tfrac{2M_N'}N\sum_i\|\xi_{1:d_N}^{(i)}\|_1$
uniformly in $\theta$. Since $\|\xi_{1:d_N}^{(i)}\|_1$ is a sum of $d_N$ i.i.d.\
$|N(0,\sigma^2)|$ terms, $\mathbb E\|\xi_{1:d_N}^{(i)}\|_1=d_N\sigma\sqrt{2/\pi}$ and
$\mathrm{Var}(\|\xi_{1:d_N}^{(i)}\|_1)=d_N\sigma^2(1-2/\pi)$, so
$\mathbb E\big[(\tfrac1N\sum_i\|\xi_{1:d_N}^{(i)}\|_1)^2\big]
=(d_N\sigma\sqrt{2/\pi})^2+\tfrac{d_N\sigma^2(1-2/\pi)}N\le C_1^2\sigma^2d_N^2$ for an
absolute constant $C_1$ since $N\ge1$. By Markov's inequality, with probability
$\ge1-\delta/4$,
\[
\tfrac1N\sum_i\|\xi_{1:d_N}^{(i)}\|_1\ \le\ C_1\sigma d_N\sqrt{4/\delta}.
\]
By the mean value theorem, on this event,
\begin{equation}
\label{eq:bn-uniform-delta}
\sup_{\theta\in\mathcal N_N}\|B_N(\theta)-B_N(\theta_N^*)\|\ \le\ r_N\cdot2M_N'C_1\sigma d_N\sqrt{4/\delta}.
\end{equation}
Under $r_NM_N'd_N=O(M_N\sqrt{p_N/N})$, the right side of~\eqref{eq:bn-uniform-delta} is
$O(\sigma M_N\sqrt{p_N/(N\delta)})$, matching the order of~\eqref{eq:bn-oracle-delta}, so
on the intersection of the two events
\[
\sup_{\theta\in\mathcal N_N}\|B_N(\theta)\|\ \le\ C_2\,\sigma M_N\sqrt{\tfrac{p_N}{N\delta}}
\]
for an absolute constant $C_2$.

For $A_N$ at the oracle: writing $S(\theta,f)=2J_\theta(\Phi_N(f))^\top r(\theta,f)$ with
$r(\theta,f)=(r_1,\dots,r_{d_N})^\top$, (B3) gives
$\|S(\theta_N^*,f)\|\le2M_N\|r(\theta_N^*,f)\|$ pointwise, so
\[
\mathbb E\|S(\theta_N^*,f_0)\|^2\le4M_N^2\sum_{k=1}^{d_N}\mathbb E_{\pi_0}\big[r_k(\theta_N^*,f_0)^2\big]
\le4M_N^2R_N(\theta_N^*)\le4C_a^2M_N^2d_N^{-2s},
\]
using Theorem~\ref{thm:approx}. Since $A_N(\theta_N^*)$ is a centered average of $N$
i.i.d.\ copies of $S(\theta_N^*,f_0)$,
$\mathbb E\|A_N(\theta_N^*)\|^2\le4C_a^2M_N^2d_N^{-2s}/N$, and by Markov's inequality,
with probability $\ge1-\delta/4$,
\[
\|A_N(\theta_N^*)\|\ \le\ 2C_aM_Nd_N^{-s}\sqrt{\tfrac{4}{N\delta}}
\ \le\ 2C_aM_N\sqrt{\tfrac{4p_N}{N\delta}},
\]
the last step because $d_N^{-s}\le1\le\sqrt{p_N}$ for $s>0$ and $p_N\ge1$. Taking the operator norm of $J^\top$ directly removes an extra factor $d_N^{1/2}$
from a coordinatewise Cauchy--Schwarz bound.

For $A_N$ on $\mathcal N_N$, condition (B4) is also used.
Differentiating, $\partial_\theta S(\theta,f)=2J_\theta(\Phi_N(f))^\top J_\theta(\Phi_N(f))
+2\sum_kr_k(\theta,f)\partial_\theta^2\partial_kg_\theta(\Phi_N(f))$: (B1) and (B3) give a
pointwise envelope on this quantity for each fixed $f\in\mathcal E_N$, and hence
bound $\widehat S_N'(\theta):=\tfrac1N\sum_i\partial_\theta S(\theta,f_0^{(i)})$ and
$\mathbb E_{\pi_0}[\partial_\theta S(\theta,f_0)]$ each on their own. Conditions (B1) and (B3) give pointwise envelopes but do not give the $N^{-1/2}$ concentration rate for $\partial_\theta A_N(\theta)=(\mathbb P_N-\mathbb P)[\partial_\theta S(\theta,f_0)]$. Condition (B4) supplies this bound. With probability $\ge1-\delta/4$,
\[
\sup_{\theta\in\mathcal N_N}\|\partial_\theta A_N(\theta)\|_{\rm op}\ \le\ b_N(\delta/4),
\]
so, by the mean value theorem, on this event,
\[
\sup_{\theta\in\mathcal N_N}\|A_N(\theta)-A_N(\theta_N^*)\|\ \le\ r_Nb_N(\delta/4).
\]
Combining with the oracle bound,
\[
\sup_{\theta\in\mathcal N_N}\|A_N(\theta)\|\ \le\
2C_aM_N\sqrt{\tfrac{4p_N}{N\delta}}\ +\ r_Nb_N(\delta/4),
\]
whose first term already has the target order and whose second does so under
$r_Nb_N(\delta)=O\big(M_N\sqrt{p_N/(N\delta)}\big)$, the second condition of the
proposition. On the intersection of all four events,
\[
\sup_{\theta\in\mathcal N_N}\big\|\nabla\widehat R_N^\circ(\theta)-\nabla R_N(\theta)\big\|
\ \le\ C_g\,(1+\sigma)M_N\sqrt{\tfrac{p_N}{N\delta}}
\]
with probability at least $1-\delta$, where $C_g$ depends only on $C_a$, on the constant
$C_1$ above, and on the implicit constants in the two $O(\cdot)$ conditions of the
statement. The constant $C_g$ is independent of $N$, $\delta$, $d_N$, and $p_N$. The $B_N$ terms are proportional to $\sigma$, while the $A_N$ terms do not contain $\sigma$, giving the factor $1+\sigma$. Part (b) gives $M_N=L_N$, so Assumption~\ref{ass:grad-conc} holds with $\zeta_N=0$.

\end{proof}

\section{Proofs for the input-dependence results}
\label{app:two-index}

\subsection{Proof of Theorem~\ref{thm:two-index-approx}}

\begin{proof}
Write $u_k=D_k\phi$ and $\mathcal F_d=\sigma(\hat e_1,\dots,\hat e_d)$. We first record
the commutation
\begin{equation}
\label{eq:commute}
  D_kg_d=\mathbb E\big[u_k\mid\mathcal F_d\big],\qquad k\le d,
  \qquad\text{and}\qquad D_kg_d=0,\quad k>d .
\end{equation}
We prove \eqref{eq:commute} in the chaos basis. Write
$\phi=\sum_\alpha c_\alpha H_\alpha$, convergent in $L^2(\gamma)$. Conditional expectation
given $\mathcal F_d$ is the orthogonal projection of $L^2(\gamma)$ onto the closed span of
$\{H_\alpha:\operatorname{supp}\alpha\subseteq[d]\}$, so
$g_d=\sum_{\operatorname{supp}\alpha\subseteq[d]}c_\alpha H_\alpha$. Since
$\partial_kH_\alpha=\sqrt{\alpha_k}H_{\alpha-e_k}$ and
$\phi\in\mathbb D^{1,2}(\gamma)$ gives $\sum_\alpha c_\alpha^2\alpha_k<\infty$, both
\[
  D_kg_d=\sum_{\operatorname{supp}\alpha\subseteq[d]}c_\alpha\sqrt{\alpha_k}H_{\alpha-e_k},
  \qquad
  \mathbb E[u_k\mid\mathcal F_d]
  =\sum_{\operatorname{supp}(\alpha-e_k)\subseteq[d]}c_\alpha\sqrt{\alpha_k}H_{\alpha-e_k}
\]
converge in $L^2(\gamma)$. For $k\le d$ and $\alpha_k\ge1$ the two index sets coincide,
because deleting one unit from a coordinate $k\le d$ does not change whether $\alpha$ has
mass beyond $d$; this gives the first part of \eqref{eq:commute}. The second part holds
because $g_d$ does not depend on $\hat e_k$ for $k>d$. Since $\{e_k\}$ is
orthonormal in $\mathcal H$, the displacement of $T-T_d$ splits exactly:
\begin{equation}
\label{eq:two-index-split}
  \|T-T_d\|_{L^2(\gamma;\mathcal H)}^2
  =\sum_{k\le d}\big\|u_k-\mathbb E[u_k\mid\mathcal F_d]\big\|_{L^2(\gamma)}^2
  +\sum_{k>d}\|u_k\|_{L^2(\gamma)}^2 .
\end{equation}
The second sum is at most $d^{-2s}\sum_{k>d}k^{2s}\|u_k\|^2\le B_1^2d^{-2s}$.

For the first sum, fix $k\le d$ and condition on $\mathcal F_d$. At fixed
$z_{[d]}:=(\hat e_1,\dots,\hat e_d)$ the map
$z_{>d}\mapsto u_k(z_{[d]},z_{>d})$ is a function of a standard Gaussian vector, and its
Cameron--Martin gradient in the free coordinates is $(D_ju_k)_{j>d}$. The Gaussian
Poincar\'e inequality, with its sharp constant $1$, gives
\[
  \operatorname{Var}\big(u_k\mid\mathcal F_d\big)
  \ \le\ \mathbb E\Big[\sum_{j>d}(D_ju_k)^2\ \Big|\ \mathcal F_d\Big]
  \qquad\gamma\text{-a.s.}
\]
Taking expectations and using
$\|u_k-\mathbb E[u_k\mid\mathcal F_d]\|_{L^2(\gamma)}^2=\mathbb E[\operatorname{Var}(u_k\mid\mathcal F_d)]$,
then summing over $k\le d$ and enlarging the range of $k$,
\[
  \sum_{k\le d}\big\|u_k-\mathbb E[u_k\mid\mathcal F_d]\big\|_{L^2(\gamma)}^2
  \ \le\ \sum_{j>d}\sum_{k\ge1}\|D_ju_k\|_{L^2(\gamma)}^2
  \ \le\ d^{-2t}\sum_{j>d}j^{2t}\sum_{k\ge1}\|D_ju_k\|_{L^2(\gamma)}^2
  \ \le\ B_2^2d^{-2t}.
\]
Substituting both bounds into \eqref{eq:two-index-split} gives
\eqref{eq:two-index-approx}. Finally $g_d$ is $\mathcal F_d$-measurable, hence a function
of $\hat e_1,\dots,\hat e_d$, so $T_d$ is admissible in the infimum over cylindrical
gradient maps.
\end{proof}

\subsection{Two norms in the chaos basis}

\begin{lemma}
\label{lem:chaos-norms}
Let $\phi=\sum_\alpha c_\alpha H_\alpha\in\mathbb D^{2,2}(\gamma)$ and
$w_r(\alpha)=\sum_jj^{2r}\alpha_j$. Then \eqref{eq:hermite-orth} holds, and
\[
  \|T-I\|_{L^2(\gamma;\mathcal H)}^2=\sum_\alpha c_\alpha^2|\alpha|,
  \qquad
  \|T\|_{\mathrm{out},s}^2=\sum_\alpha c_\alpha^2w_s(\alpha),
  \qquad
  \|T\|_{\mathrm{in},t}^2=\sum_\alpha c_\alpha^2(|\alpha|-1)w_t(\alpha).
\]
\end{lemma}

\begin{proof}
From $h_n'=\sqrt n\,h_{n-1}$ we get $\partial_kH_\alpha=\sqrt{\alpha_k}H_{\alpha-e_k}$,
with the convention that the term vanishes when $\alpha_k=0$. Hence
$\langle\nabla_{\mathcal H}H_\alpha,\nabla_{\mathcal H}H_\beta\rangle_{L^2(\gamma;\mathcal H)}
=\sum_k\sqrt{\alpha_k\beta_k}\,\mathbb E[H_{\alpha-e_k}H_{\beta-e_k}]$, and since
$\{H_\alpha\}$ is orthonormal the summand vanishes unless $\alpha-e_k=\beta-e_k$, that is
unless $\alpha=\beta$; at $\alpha=\beta$ the sum is $\sum_k\alpha_k=|\alpha|$. This is
\eqref{eq:hermite-orth}, and the first display follows.

Next, $u_k=D_k\phi=\sum_\alpha c_\alpha\sqrt{\alpha_k}H_{\alpha-e_k}$. The multi-indices
$\alpha-e_k$ appearing here are distinct for distinct $\alpha$ at fixed $k$, so
$\|u_k\|_{L^2(\gamma)}^2=\sum_\alpha c_\alpha^2\alpha_k$ and
$\sum_kk^{2s}\|u_k\|^2=\sum_\alpha c_\alpha^2\sum_kk^{2s}\alpha_k=\sum_\alpha c_\alpha^2w_s(\alpha)$.

For the second derivatives, $D_jD_k\phi=\sum_\alpha c_\alpha\sqrt{\alpha_k}\sqrt{(\alpha-e_k)_j}\,H_{\alpha-e_k-e_j}$,
so $\|D_ju_k\|^2=\sum_\alpha c_\alpha^2\alpha_k\alpha_j$ when $j\ne k$ and
$\|D_ku_k\|^2=\sum_\alpha c_\alpha^2\alpha_k(\alpha_k-1)$. Summing over $k$ at fixed $j$,
\[
  \sum_k\|D_ju_k\|^2
  =\sum_\alpha c_\alpha^2\Big[\alpha_j\sum_{k\ne j}\alpha_k+\alpha_j(\alpha_j-1)\Big]
  =\sum_\alpha c_\alpha^2\,\alpha_j\,(|\alpha|-1),
\]
and multiplying by $j^{2t}$ and summing over $j$ gives the third display.
\end{proof}

\subsection{Proof of Theorem~\ref{thm:structured-risk}}

\begin{proof}
Fix $T\in\mathcal C_{q,m}^{s}(B)$ with coefficients $(c_\alpha)_{\alpha\in\Lambda_{q,m}}$.

By \eqref{eq:hermite-orth} the fields
$\{\nabla_{\mathcal H}H_\alpha\}$ are orthogonal, so
\begin{equation}
\label{eq:structured-split}
  \big\|\widehat T_d-T\big\|_{L^2(\gamma;\mathcal H)}^2
  =\sum_{\alpha\in\mathcal I_d}(\widehat c_\alpha-c_\alpha)^2|\alpha|
  \ +\ \sum_{\alpha\in\Lambda_{q,m}\setminus\mathcal I_d}c_\alpha^2|\alpha| ,
\end{equation}
an identity, not a triangle-inequality bound.

Let $\alpha\in\Lambda_{q,m}\setminus\mathcal I_d$. Then
$\alpha_j\ge1$ for some $j>d$, so $w_s(\alpha)\ge j^{2s}\ge(d+1)^{2s}$, while
$|\alpha|\le m$ by definition of $\Lambda_{q,m}$. Hence
$c_\alpha^2|\alpha|\le m\,c_\alpha^2\le m\,(d+1)^{-2s}c_\alpha^2w_s(\alpha)$, and summing,
by Lemma~\ref{lem:chaos-norms},
\[
  \sum_{\alpha\in\Lambda_{q,m}\setminus\mathcal I_d}c_\alpha^2|\alpha|
  \ \le\ m\,(d+1)^{-2s}\sum_\alpha c_\alpha^2w_s(\alpha)
  \ \le\ m\,B^2(d+1)^{-2s}.
\]
Only the output condition is used. The degree bound makes this possible, since it caps
the factor $|\alpha|$ that each discarded coefficient carries, the role played by the
input index of Theorem~\ref{thm:two-index-approx} when the degree is unbounded.

Fix $\alpha\in\mathcal I_d$ and write
$X_i:=\langle D_i,\nabla H_\alpha(Z_i)\rangle$. Since $\operatorname{supp}\alpha\subseteq[d]$,
the inner product involves only retained coordinates. By Assumption~\ref{ass:data} and
\eqref{eq:D-decomp}, $D_{ik}=u_k(f_0^{(i)})+\xi^{(i)}_k$ with $\xi^{(i)}$ centered and
independent of the design, so
\[
  \mathbb E X_i=\big\langle\nabla_{\mathcal H}\phi,\nabla_{\mathcal H}H_\alpha\big\rangle_{L^2(\gamma;\mathcal H)}
  =\sum_\beta c_\beta\big\langle\nabla_{\mathcal H}H_\beta,\nabla_{\mathcal H}H_\alpha\big\rangle
  =c_\alpha|\alpha| ,
\]
using \eqref{eq:hermite-orth}. Hence $\mathbb E\widehat c_\alpha=c_\alpha$, for every
$\alpha\in\mathcal I_d$ and irrespective of the discarded coefficients.

With $u$ the full displacement field,
$X_i=\langle u(f_0^{(i)}),\nabla H_\alpha(Z_i)\rangle+\langle\xi^{(i)},\nabla H_\alpha(Z_i)\rangle$,
and the two terms are uncorrelated because $\xi^{(i)}$ is centered and independent of the
design. The second has variance
$\sigma^2\,\mathbb E\|\nabla H_\alpha\|^2=\sigma^2|\alpha|$. For the first, Cauchy--Schwarz
gives $\mathbb E\langle u,\nabla H_\alpha\rangle^2\le\mathbb E[\|u\|^2\|\nabla H_\alpha\|^2]
\le(\mathbb E\|u\|^4)^{1/2}(\mathbb E\|\nabla H_\alpha\|^4)^{1/2}$. Both factors are bounded using Minkowski's inequality and the hypercontractive bound
$\|F\|_4\le3^{n/2}\|F\|_2$ for $F$ of chaos degree at most $n$
\citep[Section~1.4]{Nualart2006}. Since $u_k$ has chaos degree at most $m-1$,
\[
  \mathbb E\|u\|^4=\big\|\textstyle\sum_ku_k^2\big\|_2^2
  \le\Big(\sum_k\|u_k\|_4^2\Big)^2
  \le\Big(3^{m-1}\sum_k\|u_k\|_2^2\Big)^2
  =9^{m-1}\big(\mathbb E\|u\|^2\big)^2 ,
\]
and likewise, using $\|H_\beta\|_4^2\le3^{|\beta|}$ and
$\|\nabla H_\alpha\|^2=\sum_k\alpha_kH_{\alpha-e_k}^2$,
\[
  \mathbb E\|\nabla H_\alpha\|^4\le\Big(\sum_k\alpha_k\,3^{|\alpha|-1}\Big)^2
  =9^{|\alpha|-1}|\alpha|^2\le9^{m-1}|\alpha|^2 .
\]
Since $\mathbb E\|u\|^2=\|T-I\|_{L^2(\gamma;\mathcal H)}^2\le B^2$ by
Lemma~\ref{lem:chaos-norms} and $w_s(\alpha)\ge|\alpha|$, we get
$\mathbb E\langle u,\nabla H_\alpha\rangle^2\le9^{m-1}B^2|\alpha|$ and therefore
\[
  \operatorname{Var}(\widehat c_\alpha)=\frac{\operatorname{Var}(X_1)}{N|\alpha|^2}
  \ \le\ \frac{9^{m-1}B^2+\sigma^2}{N|\alpha|}.
\]

Taking expectations in \eqref{eq:structured-split} and using
the approximation, unbiasedness, and variance bounds above,
\[
  \mathbb E\big\|\widehat T_d-T\big\|^2
  \le\sum_{\alpha\in\mathcal I_d}|\alpha|\,\frac{9^{m-1}B^2+\sigma^2}{N|\alpha|}
  +m\,B^2(d+1)^{-2s}
  =\big(9^{m-1}B^2+\sigma^2\big)\frac{p_d}N+m\,B^2(d+1)^{-2s},
\]
which is \eqref{eq:structured-risk}. Nothing above used the value of $T$ beyond membership in
the class, so the bound is uniform over it.

Since $p_d\le q(dm)^q$, the right side of
\eqref{eq:structured-risk} is at most $C_1(q,m)(B^2+\sigma^2)d^q/N+mB^2d^{-2s}$. Equating
the two terms gives $d\asymp\big(N/(B^2+\sigma^2)\big)^{1/(q+2s)}$, at which both are of
order $\big((B^2+\sigma^2)/N\big)^{2s/(q+2s)}$. Jensen's inequality applied to the square
root gives \eqref{eq:structured-rate}.
\end{proof}

\subsection{Proof of Theorem~\ref{thm:three-index}}

\begin{lemma}
\label{lem:analytic-moment}
Let $\varpi>\tfrac12\log3$ and $\|T\|_{\mathrm{ch},\varpi}\le B_3$, with no restriction on
the chaos degree of $\phi$. Then, with $u=\nabla_{\mathcal H}\phi$ and
$A_\varpi=e^{-2\varpi}/(1-3e^{-2\varpi})$,
\[
  \mathbb E\|u\|_{\mathcal H}^4\ \le\ \big(A_\varpi B_3^2\big)^2 .
\]
\end{lemma}

\begin{proof}
Decompose $u_k=\sum_{n\ge0}u_k^{(n)}$ by chaos degree, where
$u_k^{(n)}:=\sum_{|\alpha|=n+1}c_\alpha\sqrt{\alpha_k}H_{\alpha-e_k}$ lies in the $n$th
chaos, so that $\sum_k\|u_k^{(n)}\|_2^2=\sum_{|\alpha|=n+1}c_\alpha^2|\alpha|$. Minkowski's
inequality in $L^2$ and then in $L^4$, followed by the hypercontractive bound
$\|F\|_4\le3^{n/2}\|F\|_2$ on the $n$th chaos, give
\[
  \big(\mathbb E\|u\|^4\big)^{1/2}=\Big\|\sum_ku_k^2\Big\|_2\le\sum_k\|u_k\|_4^2
  \le\sum_k\Big(\sum_{n\ge0}3^{n/2}\big\|u_k^{(n)}\big\|_2\Big)^2 .
\]
For each $k$ apply Cauchy--Schwarz with the weights $e^{\mp\varpi(n+1)}$:
\[
  \Big(\sum_n3^{n/2}\big\|u_k^{(n)}\big\|_2\Big)^2
  \le\Big(\sum_n3^ne^{-2\varpi(n+1)}\Big)
  \Big(\sum_ne^{2\varpi(n+1)}\big\|u_k^{(n)}\big\|_2^2\Big)
  =A_\varpi\sum_ne^{2\varpi(n+1)}\big\|u_k^{(n)}\big\|_2^2 ,
\]
where the first factor equals $A_\varpi$ since $3e^{-2\varpi}<1$. Summing over
$k$ and using \eqref{eq:norms-in-chaos},
\[
  \sum_k\|u_k\|_4^2\le A_\varpi\sum_{n\ge0}e^{2\varpi(n+1)}\sum_{|\alpha|=n+1}c_\alpha^2|\alpha|
  =A_\varpi\sum_\alpha c_\alpha^2|\alpha|e^{2\varpi|\alpha|}\le A_\varpi B_3^2 . \qedhere
\]
\end{proof}

\begin{proof}[Proof of Theorem~\ref{thm:three-index}]
Write $m=m_N$, $d=d_N$, $\mathcal I=\mathcal I_{d,m}$. The risk split and the unbiasedness argument in the proof of
Theorem~\ref{thm:structured-risk} do not use any degree bound on $\phi$ and apply verbatim:
the risk splits exactly as in \eqref{eq:structured-split}, and each $\widehat c_\alpha$,
$\alpha\in\mathcal I$, is unbiased for $c_\alpha$ by orthogonality of the features, whatever the
discarded part of the expansion.

For the variance, only the estimator's feature set is truncated at degree $m$; the true
displacement $u=\nabla_{\mathcal H}\phi$ is not, and its chaos degree is unbounded on
$\mathcal K_q^{s,t,\varpi}$. Lemma~\ref{lem:analytic-moment} controls
$\mathbb E\|u\|^4$ uniformly in the degree. Combining it by Cauchy--Schwarz with
$\mathbb E\|\nabla H_\alpha\|^4\le9^{|\alpha|-1}|\alpha|^2$, which holds for the retained
features by the variance computation in the proof of Theorem~\ref{thm:structured-risk},
\[
  \mathbb E\langle u,\nabla H_\alpha\rangle^2
  \le\big(\mathbb E\|u\|^4\big)^{1/2}\big(\mathbb E\|\nabla H_\alpha\|^4\big)^{1/2}
  \le A_\varpi B_3^2\,3^{|\alpha|-1}|\alpha|
  \le V_m|\alpha|,
  \qquad V_m:=A_\varpi B_3^2\,3^{m-1},
\]
so that, exactly as there,
$\operatorname{Var}(\widehat c_\alpha)\le(V_m+\sigma^2)/(N|\alpha|)$. The exponential factor is $3^{m-1}$ because the analytic index removes the degree
dependence from the factor contributed by $u$. Hence
\begin{equation}
\label{eq:three-index-terms}
  \mathbb E\big\|\widehat T-T\big\|^2\ \le\ {\rm(I)}+{\rm(II)}+{\rm(III)} ,
\end{equation}
where $p_{d,m}=|\mathcal I|\le q(dm)^q$ as in \eqref{eq:hermite-sieve} and
\[
  {\rm(I)}:=\big(V_m+\sigma^2\big)\frac{p_{d,m}}N,
  \qquad
  {\rm(II)}:=\sum_{\substack{\alpha\notin\mathcal I\\ |\alpha|\le m}}c_\alpha^2|\alpha|,
  \qquad
  {\rm(III)}:=\sum_{|\alpha|>m}c_\alpha^2|\alpha| .
\]

For $|\alpha|>m$ we have $e^{2\varpi|\alpha|}\ge e^{2\varpi m}$, so
$\rm(III)\le e^{-2\varpi m}\|T\|_{\mathrm{ch},\varpi}^2\le B_3^2e^{-2\varpi m}$. By the
choice of $m$, $e^{-2\varpi m}\le e^{-2\varpi \chi\log N/\log3}=N^{-\rho_1\chi}$.

Every $\alpha$ contributing to $\rm(II)$ has $\alpha_j\ge1$ for some $j>d$ and $|\alpha|\le m$.
Two bounds are available. Using the output index alone, exactly as in the approximation bound of the proof
of Theorem~\ref{thm:structured-risk}, $\rm(II)\le mB_1^2(d+1)^{-2s}$. Using
Theorem~\ref{thm:two-index-approx}'s route, the terms with $|\alpha|=1$ contribute at most
$B_1^2(d+1)^{-2s}$ and those with $|\alpha|\ge2$ at most $2B_2^2(d+1)^{-2t}$, by
$|\alpha|\le2(|\alpha|-1)$ and $w_t(\alpha)\ge(d+1)^{2t}$. Hence
\[
  \rm(II)\ \le\ \min\Big\{mB_1^2(d+1)^{-2s},\ B_1^2(d+1)^{-2s}+2B_2^2(d+1)^{-2t}\Big\}
  \ \le\ mB_1^2d^{-2s} ,
\]
and the second expression improves the first by the factor $m$ when $t\ge s$. We use the
first bound below, which suffices for the stated rate.

By construction
$m\le\big(\chi\log N/\log3\big)+1+q$, so $3^{m}\le3^{1+q}N^{\chi}$ and
$V_m\le3^{1+q}A_\varpi B_3^2N^\chi$; the factor $3^{1+q}$ is a constant, since $q$ is
fixed, and is absorbed below. Hence $\rm(I)\le C_0N^{\chi-1}(dm)^q$ with
$C_0=C_0(q,\varpi,B_3,\sigma)$. With
$d^{\,q+2s}=N^{1-\chi}m^{-(q-1)}$,
\[
  {\rm(I)}\le C_0\,m^{q}\,N^{\chi-1}\big(N^{1-\chi}m^{-(q-1)}\big)^{\frac q{q+2s}}
  =C_0\,m^{\,q-\frac{q(q-1)}{q+2s}}\,N^{-(1-\chi)\rho_0},
\]
using $\chi-1+q(1-\chi)/(q+2s)=-(1-\chi)\big(1-\tfrac q{q+2s}\big)=-(1-\chi)\rho_0$, and
\[
  {\rm(II)}\le B_1^2\,m\,\big(N^{1-\chi}m^{-(q-1)}\big)^{-\frac{2s}{q+2s}}
  =B_1^2\,m^{\,1+\frac{2s(q-1)}{q+2s}}\,N^{-(1-\chi)\rho_0}.
\]
Both polylogarithmic exponents are at most $q$, since $q-q(q-1)/(q+2s)\le q$ and
$1+2s(q-1)/(q+2s)\le1+(q-1)=q$. Rounding $d_N$ up changes each of $\rm(I)$ and $\rm(II)$
by at most a constant factor depending on $q$ and $s$, since
$d_N\le2\big(N^{1-\chi}m^{-(q-1)}\big)^{1/(q+2s)}$ once the latter is at least $1$.

Finally the choice $\chi=\rho_0/(\rho_1+\rho_0)$ equates the two $N$-exponents:
$\rho_1\chi=\rho_1\rho_0/(\rho_1+\rho_0)$ and
$(1-\chi)\rho_0=\rho_0\rho_1/(\rho_1+\rho_0)$ are the same number. Adding the three terms and
using $m\le C\log N$ gives \eqref{eq:three-index-rate} with the stated
$(\log N)^q$. Nothing above used the value of $T$ beyond membership in
$\mathcal K_q^{s,t,\varpi}$, so the bound is uniform over the class.
\end{proof}

\section{Additional analysis for the two-groups model}
\label{sup:two-groups}

Section~\ref{subsec:bayesian-example} of the main manuscript presents the two-groups
model in its log-odds parametrization and states the resulting rate. This section supplies
the population-level verifications making $T$ a well-defined transport map, together with
the one-dimensional technical lemmas the analysis rests on.

\subsection{Population-level verifications for the two-groups model}
\label{app:two-groups-population}

This subsection verifies that $\phi\in\mathbb D^{1,2}(\gamma)$ with
Cameron--Martin gradient $\nabla_{\mathcal H}\phi$, and that $\pi_1$ and $\gamma$
are mutually absolutely continuous.

\begin{proposition}
\label{prop:two-groups-population}
Let $\delta_k=c_kk^{-2\rho}$ with $\rho>\tfrac12$ and $c_k\in[c_-,c_+]\subset(0,\tfrac13]$,
and let $\phi=\sum_{k\ge1}\bar\Xi_{\delta_k}(\hat e_k)$ be the potential of
Section~\ref{subsec:bayesian-example}. Then
$\phi\in\mathbb D^{1,2}(\gamma)$ with
$\nabla_{\mathcal H}\phi=\sum_{k\ge1}u_ke_k$ and
$u_k=M_{\delta_k}(\hat e_k)-\hat e_k$, and the pushforward
$\pi_1:=T_\#\gamma=\bigotimes_{k\ge1}\pi_{1,k}$, for
$\pi_{1,k}:=(1-\delta_k)N(0,1)+\delta_k\nu$, satisfies $\pi_1\sim\gamma$, the two measures being
mutually absolutely continuous, so $f_1>0$ $\gamma$-a.e. Since
$\pi_0=\gamma$ here, $f_0\equiv1$, so (A3) holds for the pair $(\pi_0,\pi_1)$. The map
$T$ has finite Cameron--Martin cost and is the optimal map for that cost, so
\textnormal{(A1)}, \textnormal{(A2)} and Assumption~\ref{ass:rep} hold for the pair
$(\gamma,\pi_1)$ with potential $\phi$.
\end{proposition}

\begin{proof}
Gaussian Poincar\'e and \eqref{eq:two-groups-bias} give
$\operatorname{Var}_\Phi(\Xi_{\delta_k})\le\|u_k\|_2^2\le\delta_kW_0^2$.
Since $\sum_k\delta_k<\infty$, the centered partial potentials and their gradients
converge in $L^2(\gamma)$ and $L^2(\gamma;\mathcal H)$, respectively.
Each finite partial potential belongs to $\mathbb D^{1,2}$ by the growth bound in
Lemma~\ref{lem:mtheta-growth}. Closedness of the gradient gives the claimed potential
and coefficient representation.

The coordinate pushforwards are independent with laws $\pi_{1,k}$. Both the reference
product and this product give full mass to
$\{z:\sum_k\lambda_kz_k^2<\infty\}$, because
$\mathbb E_{\pi_{1,k}}z_k^2=1+\delta_k\le2$ and $\sum_k\lambda_k<\infty$.
Each $\pi_{1,k}$ is equivalent to $\Phi$, with a positive Lebesgue density, and
\[
H^2(\pi_{1,k},\Phi)\le\mathrm{TV}(\pi_{1,k},\Phi)
=\delta_k\mathrm{TV}(\nu,\Phi)\le\delta_k,
\]
where $H^2(P,Q)=\tfrac12\int(\sqrt{dP}-\sqrt{dQ})^2$.
Kakutani's criterion gives equivalence of the products and hence $\pi_1\sim\gamma$.

For any coupling $\Pi$ of $(\gamma,\pi_1)$, Tonelli's theorem gives
\[
\int|y-x|_{\mathcal H}^2\,d\Pi
\ge\sum_kW_2^2(\Phi,\pi_{1,k}).
\]
The coordinatewise monotone map attains every term of this lower bound, so it is
optimal with cost at most $W_0^2\sum_k\delta_k<\infty$.
The preceding second-moment bound proves (A1), and the equivalence proves (A3).
Uniqueness follows from uniqueness of each one-dimensional optimal coupling.
The Gaussian-source representation theorem \citep{FeyelUstunel2004} supplies a
$1$-convex Sobolev potential for this optimal map. It has the same gradient as $\phi$,
so Gaussian Poincar\'e implies that the two potentials differ by a constant.
Consequently $\phi$ admits the required $1$-convex version.
\end{proof}

\subsection{Technical lemmas for the two-groups model}
\label{app:two-groups-lemmas}

\begin{lemma}
\label{lem:mixture-w2}
For probability measures $P,Q$ on $\mathbb R$ with finite second moment and any
$\theta\in[0,1]$,
\[
W_2\big((1-\theta)P+\theta Q,\ P\big)^2\ \le\ \theta\,W_2(P,Q)^2 .
\]
\end{lemma}

\begin{proof}
Let $(X,Y)$ be an optimal coupling of $(P,Q)$ and let
$B\sim\mathrm{Bernoulli}(\theta)$ be independent of it. Set $Z:=X$ if $B=0$ and $Z:=Y$
otherwise, so that $Z\sim(1-\theta)P+\theta Q$ and $(X,Z)$ is an admissible
coupling of $P$ with the mixture. Conditioning on $B$,
$\mathbb E[(X-Z)^2]=\theta\,\mathbb E[(X-Y)^2]=\theta\,W_2(P,Q)^2$, and the
claim follows since $W_2^2$ is the infimum over couplings. 
\end{proof}

\begin{lemma}
\label{lem:mtheta-growth}
For every $\theta\in[0,1]$ and $x\in\mathbb R$,
\[
|M_\theta(x)-x|\ \le\ C_1\big(1+x^2+\log(1/\theta)\big),
\]
for an absolute constant $C_1$, with the convention $\log(1/0):=+\infty$ interpreted as
no constraint when $\theta=0$. At $\theta=0$, the identity $M_0=\mathrm{id}$ makes the bound trivial.
\end{lemma}

The bound is proved in Section~\ref{app:proof-mtheta-growth}. It converts the
pointwise envelopes below into bounds valid on a high-probability design event.

\begin{lemma}
\label{lem:theta-sensitivity}
There is an absolute constant $L_0<\infty$ such that
\[
\sup_{x\in\mathbb R}\big|\partial_\theta M_\theta(x)\big|\ \le\ \frac{L_0}\theta,
\qquad\forall\,\theta\in(0,1].
\]
\end{lemma}

\begin{proof}
By oddness of $\partial_\theta M_\theta$ (differentiate $M_\theta(-x)=-M_\theta(x)$ in
$\theta$) it suffices to take $x\ge0$, so $t:=M_\theta(x)\ge0$. Implicit
differentiation of $F_\theta(M_\theta(x))=\Phi(x)$ in $\theta$, using
$\partial_\theta F_\theta(t)=F_\nu(t)-\Phi(t)$, gives
\[
\partial_\theta M_\theta(x)=\frac{\Phi(t)-F_\nu(t)}{f_\theta(t)}
=\frac{\tfrac12e^{-t}-\bar\Phi(t)}{f_\theta(t)},
\qquad
f_\theta(t)=(1-\theta)\varphi(t)+\tfrac\theta2e^{-t}\ \ge\ \tfrac\theta2e^{-t}.
\]
The map $t\mapsto\big(\tfrac12e^{-t}-\bar\Phi(t)\big)e^t=\tfrac12-\bar\Phi(t)e^t$ is
continuous on $[0,\infty)$, equals $0$ at $t=0$, and tends to $\tfrac12$ as $t\to\infty$
since $\bar\Phi(t)e^t\to0$ ($\bar\Phi(t)=O(e^{-t^2/2})$ decays faster than $e^{-t}$
grows). It is therefore bounded, so $L_1:=\sup_{t\ge0}\big|\tfrac12e^{-t}
-\bar\Phi(t)\big|e^t<\infty$. Hence
\[
\big|\partial_\theta M_\theta(x)\big|
=\frac{\big|\tfrac12e^{-t}-\bar\Phi(t)\big|}{f_\theta(t)}
\le\frac{L_1e^{-t}}{\tfrac\theta2e^{-t}}=\frac{2L_1}\theta,
\]
and taking $L_0:=2L_1$ gives the claim.
\end{proof}

\begin{lemma}
\label{lem:mtheta-second}
Write $A(t):=\tfrac12e^{-t}-\bar\Phi(t)$ for $t\ge0$, extended by oddness, so that
$\partial_\theta M_\theta(x)=A(t)/f_\theta(t)$ with $t=M_\theta(x)$ as in the proof of
Lemma~\ref{lem:theta-sensitivity}. Then
\[
  \mathcal V(\theta):=\mathbb E_{X\sim N(0,1)}\big[(\partial_\theta M_\theta(X))^2\big]
  =\int_{\mathbb R}\frac{A(t)^2}{f_\theta(t)}\,dt ,
\]
$\mathcal V$ is continuous and strictly positive on $(0,1]$ with
$\mathcal V_{\min}:=\inf_{\theta\in(0,1]}\mathcal V(\theta)>0$, and
$\mathcal V(\theta)\to\infty$ as $\theta\downarrow0$.
\end{lemma}

\begin{proof}
The change of variables $t=M_\theta(x)$ pushes $X\sim N(0,1)$ forward with density
$f_\theta$, so $\mathbb E[(A(t)/f_\theta(t))^2]=\int A(t)^2f_\theta(t)^{-1}\,dt$, giving the
identity. The integrand is continuous in $\theta$ for each $t$ and dominated locally
uniformly since $f_\theta(t)\ge\tfrac\theta2e^{-|t|}>0$, and is not a.e.\ zero since
$A\not\equiv0$, so $\mathcal V$ is continuous and strictly positive on $(0,1]$. As
$\theta\downarrow0$, $f_\theta(t)\to\varphi(t)$ pointwise, so Fatou's lemma gives
$\liminf_{\theta\downarrow0}\mathcal V(\theta)\ge\int A(t)^2/\varphi(t)\,dt=+\infty$,
because $A(t)^2\varphi(t)^{-1}\sim\tfrac{\sqrt{2\pi}}4e^{t^2/2-2t}\to\infty$. Hence
$\mathcal V\ge\mathcal V(1)>0$ on $(0,\theta_0]$ for some $\theta_0\in(0,1)$, while
$\inf_{[\theta_0,1]}\mathcal V>0$ because a continuous strictly positive function attains a
positive minimum on a compact set.
\end{proof}

\subsection{Proof of Lemma~\ref{lem:mtheta-growth}} \label{app:proof-mtheta-growth}

\begin{proof}
By oddness it suffices to take $x\ge0$, hence $t:=M_\theta(x)\ge0$, since
$\bar F_\theta(t)=\bar\Phi(x)$ and $\bar F_\theta$ is strictly decreasing. Both $\Phi$ and $F_\nu$ have everywhere-positive densities. It therefore suffices to exhibit $t_0\ge0$ with
$\bar F_\theta(t_0)\le\bar\Phi(x)$, which then forces $t\le t_0$. (Bounding $\bar F_\theta$
below by a multiple of $\bar\Phi(x)$ would instead bound $t$ from below; the inequality
needed here runs the other way.)

There is an absolute constant $c_0\in(0,\tfrac12]$ with $\bar\Phi(y)\ge c_0e^{-y^2}$ for every $y\ge0$. The ratio $\bar\Phi(y)e^{y^2}$ is continuous and positive on $[0,\infty)$,
equals $\tfrac12$ at $y=0$, and $\to\infty$ as $y\to\infty$ (since $\bar\Phi(y)\sim
\varphi(y)/y$), hence is bounded away from $0$ throughout. Fix
$K:=1+\sqrt{2\log(2/c_0)}+\log(2/c_0)$ and set
\[
t_0:=x^2+\log(1/\theta)+K.
\]

Gaussian piece. Using $\sqrt{p+q}\le\sqrt p+\sqrt q$ for $p,q\ge0$,
$\sqrt{2x^2+2\log(2/c_0)}\le\sqrt2\,x+\sqrt{2\log(2/c_0)}\le x^2+1+\sqrt{2\log(2/c_0)}\le
x^2+K$, the middle step because $x^2-\sqrt2\,x+1=(x-\tfrac1{\sqrt2})^2+\tfrac12>0$ for
every $x$. Hence $t_0\ge x^2+K\ge\sqrt{2x^2+2\log(2/c_0)}$, so
$t_0^2\ge2x^2+2\log(2/c_0)$, and with the standard bound $\bar\Phi(t_0)\le\tfrac12e^{-t_0^2/2}$,
\[
(1-\theta)\bar\Phi(t_0)\ \le\ \bar\Phi(t_0)\ \le\ \tfrac12e^{-t_0^2/2}
\ \le\ \tfrac12\cdot\tfrac{c_0}2e^{-x^2}\ \le\ \tfrac14\bar\Phi(x).
\]

Laplace piece. Since $K\ge\log(2/c_0)$, $t_0\ge x^2+\log(1/\theta)+\log(2/c_0)$,
so
\[
\theta\bar F_\nu(t_0)=\tfrac\theta2e^{-t_0}
\le\tfrac\theta2\cdot e^{-x^2}\cdot\theta\cdot\tfrac{c_0}2
=\tfrac{\theta^2c_0}4e^{-x^2}
\le\tfrac{c_0}4e^{-x^2}
\le\tfrac14\bar\Phi(x),
\]
using $\theta\le1$ in the last inequality but one.

Summing the two pieces, $\bar F_\theta(t_0)=(1-\theta)\bar\Phi(t_0)+\theta\bar F_\nu(t_0)
\le\tfrac14\bar\Phi(x)+\tfrac14\bar\Phi(x)=\tfrac12\bar\Phi(x)\le\bar\Phi(x)$, so
$t=M_\theta(x)\le t_0=x^2+\log(1/\theta)+K$. Since $x,t\ge0$, one has
$|M_\theta(x)-x|=|t-x|\le\max(t,x)\le t_0$, using $t_0\ge x^2+K\ge x$, so the same bound
controls the displacement itself. This is the claim with $C_1:=\max(1,K)$.
\end{proof}

\section{Proof of the oracle inequality}
\label{sup:proofs-general}
\label{appendix_a1}

\begin{proof}
For every $\delta\in(0,1)$, we show that with probability at least $1-\delta-\eta_N-\zeta_N$,
\[
\|\widehat T_N - T\|_{L^2(\pi_0;\mathcal H)}
\ \le\
C_g\,(1+\sigma)\kappa_NL_N\sqrt{\frac{p_N}{N\delta}} + \|T_N^*-T\|_{L^2(\pi_0;\mathcal H)},
\]
which is the stated bound, with $C_g$ and $\zeta_N$ from Assumption~\ref{ass:grad-conc}. The approximation term $\|T_N^*-T\|_{L^2(\pi_0;\mathcal H)}$ is kept in this form throughout the proof.

By the triangle inequality,
\begin{equation}\label{eq:triangle-decomp}
\|\widehat T_N - T\|_{L^2(\pi_0;\mathcal H)}
\le
\|T_{\hat\theta_N,N}-T_{\theta_N^*,N}\|_{L^2(\pi_0;\mathcal H)}
+
\|T_{\theta_N^*,N}-T\|_{L^2(\pi_0;\mathcal H)}.
\end{equation}
For the first term in \eqref{eq:triangle-decomp}, Assumption~\ref{ass:lipschitz-map} yields
\begin{equation}\label{eq:lipschitz-term}
\|T_{\hat\theta_N,N}-T_{\theta_N^*,N}\|_{L^2(\pi_0;\mathcal H)}
\le
L_N\|\hat\theta_N-\theta_N^*\|.
\end{equation}
The second term in \eqref{eq:triangle-decomp} is $\|T_N^*-T\|_{L^2(\pi_0;\mathcal H)}$ by definition. The identity $\mathbb E[\widehat R_N^\circ(\theta)]=R_N(\theta)-R_N(\theta_0)$ from Section~\ref{sup:foundations} connects the empirical criterion to the population risk. Combining \eqref{eq:triangle-decomp} and \eqref{eq:lipschitz-term} gives
\begin{equation}\label{eq:prebound}
\|\widehat T_N-T\|_{L^2(\pi_0;\mathcal H)}
\le
L_N\|\hat\theta_N-\theta_N^*\| + \|T_N^*-T\|_{L^2(\pi_0;\mathcal H)}.
\end{equation}
We next bound $\|\hat\theta_N-\theta_N^*\|$ using the basic inequality for the empirical minimizer.

Let $\mathcal{E}_N:=\{\hat\theta_N\in\mathcal{N}_N\}$ denote the event that the estimator $\hat\theta_N$ lies in the neighborhood $\mathcal{N}_N$. By Assumption~\ref{ass:localization}, $\mathbb{P}(\mathcal{E}_N^c)\le\eta_N$. Since $\hat\theta_N\in\arg\min_{\theta\in\Theta_N}\widehat R_N^\circ(\theta)$,
\begin{equation}\label{eq:basic-ineq}
\widehat R_N^\circ(\hat\theta_N)\ \le\ \widehat R_N^\circ(\theta_N^*),
\end{equation}
by definition of $\hat\theta_N$ as a minimizer over $\Theta_N$. Writing $\Delta_N(\theta):=(\widehat R_N^\circ-R_N)(\theta)-(\widehat R_N^\circ-R_N)(\theta_N^*)$,
so that $\Delta_N(\theta_N^*)=0$, \eqref{eq:basic-ineq} rearranges to
\[
R_N(\hat\theta_N)-R_N(\theta_N^*)
\ \le\
-\Delta_N(\hat\theta_N)
\ \le\
|\Delta_N(\hat\theta_N)|.
\]
On $\mathcal E_N$, both $\theta_N^*$ and $\hat\theta_N$ lie in the convex set
$\mathcal N_N$ (a Euclidean ball), so the segment between them does too, and the mean
value theorem applied to $\Delta_N$ along this segment gives
\begin{equation}\label{eq:mvt-delta}
|\Delta_N(\hat\theta_N)|
\ \le\
\sup_{\theta\in\mathcal N_N}\big\|\nabla\widehat R_N^\circ(\theta)-\nabla R_N(\theta)\big\|\cdot\|\hat\theta_N-\theta_N^*\|.
\end{equation}
Combining, on $\mathcal E_N$,
\begin{equation}\label{eq:growth-upper}
R_N(\hat\theta_N)-R_N(\theta_N^*)
\ \le\
\sup_{\theta\in\mathcal N_N}\big\|\nabla\widehat R_N^\circ(\theta)-\nabla R_N(\theta)\big\|\cdot\|\hat\theta_N-\theta_N^*\|.
\end{equation}
On the other hand, the quadratic growth \eqref{eq:quad-growth} of $R_N$ applies on
$\mathcal N_N$, since $\theta_N^*$ is an interior minimizer of $R_N$. Hence, on
$\mathcal E_N$,
\begin{equation}\label{eq:growth-lower}
\frac{\mu_N}2\|\hat\theta_N-\theta_N^*\|^2\ \le\ R_N(\hat\theta_N)-R_N(\theta_N^*).
\end{equation}
Combining \eqref{eq:growth-upper} and \eqref{eq:growth-lower}, on $\mathcal E_N$,
\[
\frac{\mu_N}2\|\hat\theta_N-\theta_N^*\|^2
\ \le\
\sup_{\theta\in\mathcal N_N}\big\|\nabla\widehat R_N^\circ(\theta)-\nabla R_N(\theta)\big\|\cdot\|\hat\theta_N-\theta_N^*\|,
\]
If $\|\hat\theta_N-\theta_N^*\|>0$, divide by this norm. If it is zero, the bound below is trivial. Thus
\begin{equation}
\|\hat\theta_N-\theta_N^*\|
\ \le\
\frac2{\mu_N}
\sup_{\theta\in\mathcal{N}_N}\big\|\nabla \widehat R_N^\circ(\theta)-\nabla R_N(\theta)\big\|,
\label{eq:theta-by-grad}
\end{equation}
on $\mathcal E_N$. Fix $\delta\in(0,1)$. By Assumption~\ref{ass:grad-conc}, the event
\[
\begin{aligned}
\mathcal F_N(\delta):=&\bigg\{\sup_{\theta\in\mathcal{N}_N}\big\|\nabla \widehat R_N^\circ(\theta)-\nabla R_N(\theta)\big\|
\\
&\le C_g\,(1+\sigma)L_N\sqrt{\tfrac{p_N}{N\delta}}\bigg\}
\end{aligned}
\]
has $\mathbb P(\mathcal F_N(\delta)^c)\le\delta+\zeta_N$. On $\mathcal E_N\cap\mathcal F_N(\delta)$,
\eqref{eq:theta-by-grad} gives
\begin{equation}\label{eq:param-rate}
\|\hat\theta_N-\theta_N^*\|
\ \le\
2C_g\,(1+\sigma)\frac{L_N}{\mu_N}\sqrt{\frac{p_N}{N\delta}}
\ =\
2C_g\,(1+\sigma)\kappa_N\sqrt{\frac{p_N}{N\delta}}.
\end{equation}
By the union bound, $\mathbb P\big((\mathcal E_N\cap\mathcal F_N(\delta))^c\big)
\le\mathbb P(\mathcal E_N^c)+\mathbb P(\mathcal F_N(\delta)^c)\le\eta_N+\delta+\zeta_N$, so
\eqref{eq:param-rate} holds with probability at least $1-\delta-\eta_N-\zeta_N$.

Substituting \eqref{eq:param-rate} into \eqref{eq:prebound}, on the same event
$\mathcal E_N\cap\mathcal F_N(\delta)$,
\[
\begin{aligned}
\|\widehat T_N-T\|_{L^2(\pi_0;\mathcal H)}
&\le
L_N\|\hat\theta_N-\theta_N^*\|+\|T_N^*-T\|_{L^2(\pi_0;\mathcal H)}
\\
&\le
2C_g\,(1+\sigma)\kappa_NL_N\sqrt{\frac{p_N}{N\delta}}+\|T_N^*-T\|_{L^2(\pi_0;\mathcal H)},
\end{aligned}
\]
with probability at least $1-\delta-\eta_N-\zeta_N$. Taking $C:=2C_g$ gives the theorem.
\end{proof}

\section{Proof of the minimax lower bound}
\label{sup:proofs-minimax}

\subsection{Proof of Lemma~\ref{lem:hypothesis}} \label{app:proof-hypothesis}

\begin{proof}
Throughout index the band by $k=d+1,\dots,2d$.

(i) Since $\omega_k\in\{0,1\}$ implies
$\omega_k^2=\omega_k$, Definition~\ref{def:reg-class} together with $\|v_k\|_{L^2(\pi_0)}=1$
gives
\[
\|T_\omega\|_{\mathcal W^s}^2
=\tau^2\sum_{k=d+1}^{2d}k^{2s}\omega_k
\le\tau^2(2d)^{2s}d=2^{2s}\tau^2d^{2s+1}.
\]
using $k\le2d$ on the band and $\sum_{k=d+1}^{2d}\omega_k\le d$. There are $d$ terms, each taking a value in $\{0,1\}$. Taking square roots gives $\|T_\omega\|_{\mathcal W^s}\le2^s\tau\,d^{s+1/2}$,
i.e.\ (i) holds with $C_1=2^s$, for every $\omega\in\{0,1\}^d$, hence in particular
for $\omega\in\Omega$.

(ii) By the Varshamov--Gilbert bound
\citep[Lemma~2.9]{tsybakov2009introduction}, for $m\ge8$ there is a subset
$\Omega\subseteq\{0,1\}^d$ containing $\underline0:=(0,\ldots,0)$ with
\[
|\Omega|\ge2^{m/8}
\qquad\text{and}\qquad
d_H(\omega,\omega')\ge d/8\quad\text{for all distinct }\omega,\omega'\in\Omega,
\]
$d_H$ denoting Hamming distance. In particular $\log|\Omega|\ge(d/8)\log2=:cd$ with
$c=\log2/8$, which is the claimed cardinality bound. Since $\{e_k\}$ is orthonormal in
$\mathcal H$ and $(\omega_k-\omega_k')^2=\mathbb 1\{\omega_k\ne\omega_k'\}$ for
$\omega_k,\omega_k'\in\{0,1\}$,
\[
\|T_\omega-T_{\omega'}\|_{L^2(\pi_0;\mathcal H)}^2
=\mathbb E_{\pi_0}\Big[\Big\|\tau\!\!\sum_{k=d+1}^{2d}\!\!(\omega_k-\omega_k')v_k\,e_k\Big\|_{\mathcal H}^2\Big]
=\tau^2\!\!\sum_{k=d+1}^{2d}\!\!(\omega_k-\omega_k')^2\|v_k\|_{L^2(\pi_0)}^2
=\tau^2\,d_H(\omega,\omega').
\]
For distinct $\omega,\omega'\in\Omega$ this is $\ge\tau^2m/8=\tau^2d/8$, giving (ii)
with $c_2=1/8$.

(iii) Fix $\omega,\omega'\in\{0,1\}^d$ and consider
one pair $(f_0,f_1)$ generated according to Assumption~\ref{ass:data} with population map
$T_\omega$ or $T_{\omega'}$. Conditionally on $f_0=f$, write the standardized coordinates of $f_1-f$ as in
\eqref{eq:Dik},
\[
  D_k:=\frac{\langle f_1-f,\varphi_k\rangle_{\mathcal F}}{\sqrt{\lambda_k}},
  \qquad k\ge1,
\]
which are ordinary real random variables; the Cameron--Martin bracket
$\langle f_1-f,e_k\rangle_{\mathcal H}$ is not available here, since
$f_1-f=T_\omega(f)-f+\xi$ and $\xi\notin\mathcal H$ almost surely. For $k$ outside the band
$\{d+1,\ldots,2d\}$, $T_\omega(f)=T_{\omega'}(f)=f$, so this coordinate reduces to the noise
coordinate $\xi_k$ of \eqref{eq:noise-coords}, which is $N(0,\sigma^2)$ under both
hypotheses, independently of $f$ and of the band coordinates. On the band, writing
$\xi_{\mathrm{band}}:=(\xi_{d+1},\ldots,\xi_{2d})\sim N(0,\sigma^2I_d)$,
\[
\big(D_k\big)_{k=d+1}^{2d}
=\tau\,\omega\odot v(f)+\xi_{\mathrm{band}}
\quad\big(\text{resp.\ }\tau\,\omega'\odot v(f)+\xi_{\mathrm{band}}\big)
\]
under $T_\omega$ and $T_{\omega'}$, respectively, where $v(f):=(v_{d+1}(f),\ldots,v_{2d}(f))$ and
$\odot$ is the coordinatewise product. Thus, conditionally on $f_0=f$, the two hypotheses
induce $N(\tau\,\omega\odot v(f),\sigma^2I_d)$ and $N(\tau\,\omega'\odot v(f),\sigma^2I_d)$
on the band coordinates, and the same law on the complementary coordinates,
independently of the band. Since $\mathrm{KL}(P_1\otimes Q\,\|\,P_2\otimes Q)=\mathrm{KL}(P_1\|P_2)$
for any common factor $Q$, and since the shift $\tau(\omega-\omega')\odot v(f)$ lies in the
finite-dimensional, hence Cameron--Martin, span of $e_{d+1},\ldots,e_{2d}$, so that the two
conditional laws are mutually absolutely continuous, the conditional KL divergence reduces to
the standard Gaussian shift formula,
\[
\mathrm{KL}\big(P_{T_\omega}(\cdot\mid f)\,\|\,P_{T_{\omega'}}(\cdot\mid f)\big)
=\frac{1}{2\sigma^2}\big\|\tau(\omega-\omega')\odot v(f)\big\|_2^2
=\frac{\tau^2}{2\sigma^2}\sum_{k=d+1}^{2d}(\omega_k-\omega_k')^2v_k(f)^2 .
\]
Since the $f_0$-marginal is $\pi_0$ under both hypotheses, integrating over $f\sim\pi_0$
gives the unconditional KL divergence between the one-pair laws,
\begin{align*}
\mathrm{KL}\big(P_{T_\omega}\|P_{T_{\omega'}}\big)
&=\mathbb E_{\pi_0}\!\left[
\frac{\tau^2}{2\sigma^2}
\sum_{k=d+1}^{2d}(\omega_k-\omega_k')^2v_k(f_0)^2
\right]\\
&=\frac{\tau^2}{2\sigma^2}
\sum_{k=d+1}^{2d}(\omega_k-\omega_k')^2
\|v_k\|_{L^2(\pi_0)}^2\\
&=\frac{\tau^2}{2\sigma^2}\,d_H(\omega,\omega')
\le\frac{\tau^2d}{2\sigma^2}.
\end{align*}
using $\|v_k\|_{L^2(\pi_0)}=1$ and $d_H(\omega,\omega')\le d$. Finally, since the $N$
pairs $(f_0^{(i)},f_1^{(i)})$ are i.i.d.\ under Assumption~\ref{ass:data}, KL divergence
tensorizes,
\[
\mathrm{KL}\big(P_{T_\omega}^{\otimes N}\|P_{T_{\omega'}}^{\otimes N}\big)
=N\,\mathrm{KL}\big(P_{T_\omega}\|P_{T_{\omega'}}\big)
\le\frac{N\tau^2d}{2\sigma^2},
\]
which is (iii).
\end{proof}

\subsection{Proof of Theorem~\ref{thm:minimax}} \label{app:proof-minimax}

\begin{proof}
Fix $\alpha:=1/16\in(0,1/8)$. For $d\ge16$, the Varshamov--Gilbert bound
$|\Omega|\ge2^{d/8}$ underlying Lemma~\ref{lem:hypothesis}(ii) gives, for $M:=|\Omega|-1$,
the number of non-reference hypotheses with $\Omega\ni\underline0$,
\[
M\ge2^{d/8}-1\ge2^{d/8-1}\ge2^{d/16},
\qquad\text{so}\qquad
\log M\ge\frac{d}{16}\log2=:c_0\,d,
\qquad c_0:=\frac{\log2}{16},
\]
using $2^x-1\ge2^{x-1}$ for $x\ge1$ and $d/8-1\ge d/16$ for $d\ge16$.

Set
\[
\tau^2:=\tau_N^2:=\frac{2\alpha c_0\,\sigma^2}{N},
\qquad
d:=d_N:=\Big\lfloor\big(B/(C_1\tau_N)\big)^{1/(s+1/2)}\Big\rfloor,
\]
with $C_1=2^s$ from Lemma~\ref{lem:hypothesis}(i). By construction $C_1\tau_Nd_N^{s+1/2}\le B$,
so Lemma~\ref{lem:hypothesis}(i) places every $T_\omega$, $\omega\in\Omega$, in
$\mathcal W^s(B)$. Since $\tau_N\asymp\sigma N^{-1/2}$ and
$(s+1/2)/(2s+1)=\tfrac12$ exactly, $d_N\asymp(B/\tau_N)^{2/(2s+1)}\asymp(B^2N/\sigma^2)^{1/(2s+1)}\to\infty$
as $N\to\infty$. Fix $N_1$ such that $d_N\ge16$ for all $N\ge N_1$, so that $\log M\ge c_0d_N$
holds from $N_1$ on.

By Lemma~\ref{lem:hypothesis}(iii) with
$\omega'=\underline0$. In this case $T_{\underline0}=I$. Therefore every $\omega\in\Omega\setminus\{\underline0\}$
satisfies
\[
\mathrm{KL}\big(P_{T_\omega}^{\otimes N}\,\|\,P_{T_{\underline0}}^{\otimes N}\big)
\le\frac{N\tau_N^2d_N}{2\sigma^2}=\alpha c_0d_N\le\alpha\log M,
\]
so the average over $\omega\in\Omega\setminus\{\underline0\}$ satisfies the same
bound, verifying condition (ii) of the generalized Fano method
\citep[Theorem~2.5]{tsybakov2009introduction} with this $\alpha\in(0,1/8)$. By
Lemma~\ref{lem:hypothesis}(ii), for all distinct $\omega,\omega'\in\Omega$,
\[
\|T_\omega-T_{\omega'}\|_{L^2(\pi_0;\mathcal H)}\ge\tau_N\sqrt{d_N/8}=:2\rho_N,
\qquad
\rho_N:=\tfrac12\tau_N\sqrt{d_N/8},
\]
which verifies condition (i) of the same method with separation $2\rho_N$ in the distance
$d(T,T'):=\|T-T'\|_{L^2(\pi_0;\mathcal H)}$. Enumerating $\Omega=\{\omega^{(0)},\ldots,\omega^{(M)}\}$
with $\omega^{(0)}=\underline0$, Theorem~2.5 of \citet{tsybakov2009introduction} gives
\[
\inf_{\widehat T}\ \max_{0\le j\le M}\
\mathbb P_{T_{\omega^{(j)}}}\Big(\big\|\widehat T-T_{\omega^{(j)}}\big\|_{L^2(\pi_0;\mathcal H)}\ge\rho_N\Big)
\ \ge\
\frac{\sqrt M}{1+\sqrt M}\Big(1-2\alpha-\sqrt{\tfrac{2\alpha}{\log M}}\Big),
\]
the infimum over all estimators measurable in the $N$ paired samples. Since $M\to\infty$ as
$N\to\infty$ and $\alpha=1/16$, the right side tends to $1-2\alpha=7/8$. Fix $N_2\ge N_1$ so
that it exceeds $q:=1/2$ for all $N\ge N_2$.

This testing bound converts into a bound on the estimation risk as follows. For every estimator $\widehat T$ and
every $j$, Markov's inequality gives
\[
\mathbb E_{T_{\omega^{(j)}}}\big\|\widehat T-T_{\omega^{(j)}}\big\|_{L^2(\pi_0;\mathcal H)}
\ge\rho_N\,\mathbb P_{T_{\omega^{(j)}}}\Big(\big\|\widehat T-T_{\omega^{(j)}}\big\|_{L^2(\pi_0;\mathcal H)}\ge\rho_N\Big),
\]
and since every $T_{\omega^{(j)}}\in\mathcal W^s(B)$,
\[
\sup_{T\in\mathcal W^s(B)}\mathbb E_T\|\widehat T-T\|_{L^2(\pi_0;\mathcal H)}
\ge\max_{0\le j\le M}\mathbb E_{T_{\omega^{(j)}}}\big\|\widehat T-T_{\omega^{(j)}}\big\|_{L^2(\pi_0;\mathcal H)}
\ge\rho_N\max_{0\le j\le M}\mathbb P_{T_{\omega^{(j)}}}\big(\cdots\ge\rho_N\big).
\]
Taking $\inf_{\widehat T}$ and combining with the Fano bound above,
\[
\inf_{\widehat T}\sup_{T\in\mathcal W^s(B)}\mathbb E_T\|\widehat T-T\|_{L^2(\pi_0;\mathcal H)}
\ge q\,\rho_N,\qquad N\ge N_2 .
\]

It remains to evaluate $\rho_N$. By construction $\tau_N^2=2\alpha c_0\sigma^2/N$
and $d_N\asymp(B/(C_1\tau_N))^{2/(2s+1)}$, using $1/(s+1/2)=2/(2s+1)$; the floor changes
$d_N$ by at most one unit and hence $\rho_N$ by at most a bounded factor once
$d_N\ge16$, which the symbol $\asymp$ absorbs. Working with the unfloored value, hence
\[
\rho_N^2=\frac{\tau_N^2d_N}{32}
=\frac{\tau_N^2}{32}\Big(\frac{B}{C_1\tau_N}\Big)^{2/(2s+1)}
=\frac{1}{32}\Big(\frac{B}{C_1}\Big)^{2/(2s+1)}\tau_N^{2-2/(2s+1)}
=\frac{1}{32}\Big(\frac{B}{C_1}\Big)^{2/(2s+1)}\tau_N^{4s/(2s+1)} ;
\]
since $\tau_N^2=2\alpha c_0\sigma^2/N$, $\tau_N^{4s/(2s+1)}=(2\alpha c_0)^{2s/(2s+1)}(\sigma^2/N)^{2s/(2s+1)}$,
so
\[
\rho_N^2=\frac{(2\alpha c_0)^{2s/(2s+1)}}{32}\Big(\frac{B}{C_1}\Big)^{2/(2s+1)}\Big(\frac{\sigma^2}{N}\Big)^{2s/(2s+1)},
\]
i.e.\ $\rho_N=c_1\big(\sigma^2/N\big)^{s/(2s+1)}$ for a constant
$c_1=c_1(s,B)>0$ not depending on $N$ or $\sigma$. The floor correction is absorbed into a bounded factor for $N\ge N_2$. Combining with the previous display,
\[
\inf_{\widehat T}\sup_{T\in\mathcal W^s(B)}\mathbb E_T\|\widehat T-T\|_{L^2(\pi_0;\mathcal H)}
\ge q\,c_1\Big(\frac{\sigma^2}{N}\Big)^{s/(2s+1)},\qquad N\ge N_2 .
\]
This proves the theorem with $c:=qc_1$ and $N_0:=N_2$.

For the diagonal subclass, take $\pi_0=\gamma$ as in Definition~\ref{def:diag-class}, so
that $\|\hat e_k\|_{L^2(\pi_0)}=1$ and $c_k=1$. The hypothesis family then already lies
inside the subclass, since the coordinate fields are
$u_k^\omega=\tau\omega_kv_k=\tau\omega_k\hat e_k$, so
$a_k=\tau\omega_k\ge0\ge\underline a$ for $k\in\{d+1,\ldots,2d\}$ and $a_k=0$
elsewhere, so the floor in Definition~\ref{def:diag-class} is slack for every
hypothesis, whence
$\{T_\omega:\omega\in\Omega\}\subset\mathcal W_{\mathrm{diag}}^s(B)\subset\mathcal W^s(B)$
by Lemma~\ref{lem:hypothesis}(i). The argument above uses only the norm bound, pairwise
separation, and Kullback--Leibler bound of Lemma~\ref{lem:hypothesis}, never a property
of the ambient class beyond its containing $\{T_\omega\}$, so it lower-bounds the
diagonal-subclass minimax risk verbatim and at the same rate.
\end{proof}

\section{Proofs for the worked examples}
\label{sup:proofs-examples}

\subsection{Conjugate Gaussian inverse problem}
\label{sup:gaussian-proofs}

\subsubsection{Transport map and bias in the conjugate Gaussian example} \label{app:proof-gaussian-bias}

\begin{proof}
Let $X\sim\pi_0$ and let $\{\varphi_k\}$ be the common eigenbasis of the commuting
operators $C_0,C_1$, so that $X_k:=\langle X,\varphi_k\rangle_{\mathcal F}$ are
independent $N(0,\lambda_k^{(0)})$. If $T(x)=Ax$ is diagonal in this basis with
$A\varphi_k=a_k\varphi_k$, then $\mathrm{Var}(\langle T(X),\varphi_k\rangle_{\mathcal F})
=a_k^2\lambda_k^{(0)}$, which must equal $\lambda_k^{(1)}$ for $T$ to push $\pi_0$
forward to $\pi_1$. This gives $a_k=\pm(\lambda_k^{(1)}/\lambda_k^{(0)})^{1/2}$, and
matching variances alone does not select the sign. The positive root is the optimal one:
each coordinate is a one-dimensional transport between $N(0,\lambda_k^{(0)})$ and
$N(0,\lambda_k^{(1)})$, whose $W_2$-optimal map is the nondecreasing one,
$x\mapsto+(\lambda_k^{(1)}/\lambda_k^{(0)})^{1/2}x$. Summing the coordinatewise $W_2^2$ lower bounds by Tonelli, as in the proof of Proposition~\ref{prop:two-groups-population}, shows that the coordinatewise monotone map attains the infimum over all couplings. Hence \eqref{eq:gauss-OT-diagonal} is the optimal map. By definition of $T_d$,
$T_d(x)-T(x)=\sum_{k>d}(1-a_k)\langle x,\varphi_k\rangle_{\mathcal F}\varphi_k$, which is
\eqref{eq:Td-minus-T}. Since $\gamma=\pi_0$ gives
$\langle\varphi_j,\varphi_k\rangle_{\mathcal H}=\delta_{jk}/\lambda_k^{(0)}$, the factors
$\lambda_k^{(0)}$ from $\mathbb E[\langle X,\varphi_k\rangle_{\mathcal F}^2]$ and
$|\varphi_k|_{\mathcal H}^2$ cancel term by term, leaving
$\|T_d-T\|_{L^2(\pi_0;\mathcal H)}^2=\sum_{k>d}(1-a_k)^2$, which proves
\eqref{eq:explicit-bias}.

The sieve class contains the canonical truncation map. By
\eqref{eq:sieve-map-linear}, for $B=B^\ast=\mathrm{diag}(a_1-1,\ldots,a_{d_N}-1)$,
\[
\big(B^\ast\Phi_N(f)\big)_k=(a_k-1)\,\hat e_k(f)=(a_k-1)\,\frac{\langle f,\varphi_k\rangle_{\mathcal F}}{\sqrt{\lambda_k^{(0)}}},
\]
so
\[
T_{B^\ast,N}(f)
=
f+\sum_{k=1}^{d_N}(a_k-1)\,\frac{\langle f,\varphi_k\rangle_{\mathcal F}}{\sqrt{\lambda_k^{(0)}}}\;\sqrt{\lambda_k^{(0)}}\,\varphi_k
=
f+\sum_{k=1}^{d_N}(a_k-1)\langle f,\varphi_k\rangle_{\mathcal F}\,\varphi_k :
\]
the factor $\sqrt{\lambda_k^{(0)}}$ from $e_k=\sqrt{\lambda_k^{(0)}}\varphi_k$ in
\eqref{eq:sieve-map-linear} exactly cancels the $1/\sqrt{\lambda_k^{(0)}}$ from the
normalization of $\Phi_N$, so the $\lambda_k^{(0)}$-dependence drops out and
\[
T_{B^\ast,N}(f)
=
\sum_{k=1}^{d_N}a_k\langle f,\varphi_k\rangle_{\mathcal F}\,\varphi_k
+
\sum_{k>d_N}\langle f,\varphi_k\rangle_{\mathcal F}\,\varphi_k
=
T_{d_N}(f).
\]
Therefore,
\[
\inf_{\theta\in\Theta_N}\|T_{\theta,N}-T\|_{L^2(\pi_0;\mathcal H)}
\le
\|T_{B^\ast,N}-T\|_{L^2(\pi_0;\mathcal H)}
=
\|T_{d_N}-T\|_{L^2(\pi_0;\mathcal H)}.
\]
Combining this with \eqref{eq:explicit-bias} gives \eqref{eq:explicit-bias-final}.
\end{proof}

\subsubsection{Proof of Proposition~\ref{prop:gaussian-verify}} \label{app:proof-gaussian-verify}

\begin{proof}
For $\mu_N=1$: here $\partial_kg_\theta(z)=\theta_kz_k$, so
$J_\theta(z)=\mathrm{diag}(z_1,\ldots,z_{d_N})$ does not depend on $\theta$ at all, the
zero-curvature case of Assumption~\ref{ass:sieve-primitive}(B2), and
\[
G_N(\theta)=\mathbb E_{\pi_0}\big[J_\theta(\Phi_N(f_0))^\top J_\theta(\Phi_N(f_0))\big]
=\mathrm{diag}\big(\mathbb E[\hat e_1(f_0)^2],\ldots,\mathbb E[\hat e_{d_N}(f_0)^2]\big)=I_{d_N},
\]
since each $\hat e_k(f_0)\sim N(0,1)$ under $\pi_0$. Since $g_\theta$ is affine in $\theta$,
the curvature term of Assumption~\ref{ass:sieve-primitive}(B2) vanishes identically, so
$H_{R_N}(\theta)=2G_N(\theta)=2I_{d_N}$ at every $\theta$ and the quadratic growth
\eqref{eq:quad-growth} holds on all of $\Theta_N$ with $\mu_N=1$.

For $L_N=1$, the quadratic sieve gives
$T_{\theta,N}(f)-T_{\theta',N}(f)=\sum_{k\le d_N}(\theta_k-\theta_k')\hat e_k(f)e_k$, and hence
\[
\|T_{\theta,N}-T_{\theta',N}\|_{L^2(\pi_0;\mathcal H)}^2
=\sum_{k\le d_N}(\theta_k-\theta_k')^2\,\mathbb E[\hat e_k(f_0)^2]
=\|\theta-\theta'\|^2.
\]
Thus Assumption~\ref{ass:lipschitz-map} holds with $L_N=1$. The pointwise envelope in
Assumption~\ref{ass:sieve-primitive}(B3) is unbounded for this quadratic sieve, so we use
the direct $L^2$ calculation. For the block sieve of Section~\ref{subsec:block},
$\nabla\psi$ is bounded and the corresponding envelope is finite.

We next use the exact ERM solution to bound the parameter error and verify
Assumption~\ref{ass:localization}. The loss is quadratic and separable across coordinates.
Conditionally on the design, the coordinate errors are Gaussian.

Write
\[
Z_{ik}:=\hat e_k(f_0^{(i)})
=\frac{\langle f_0^{(i)},\varphi_k\rangle_{\mathcal F}}{\sqrt{\lambda_k}}
\stackrel{\mathrm{i.i.d.}}{\sim}N(0,1),
\qquad
\xi_{ik}:=\frac{\langle \xi^{(i)},\varphi_k\rangle_{\mathcal F}}{\sqrt{\lambda_k}}
\stackrel{\mathrm{i.i.d.}}{\sim}N(0,\sigma^2),
\]
with $\{\xi_{ik}\}$ independent of $\{Z_{ik}\}$. Since $g_\theta$ is quadratic and the
coordinates separate, $\widehat R_N^\circ(\theta)$ is, up to an
additive constant not depending on $\theta$, a sum over $k\le d_N$ of independent
one-dimensional least-squares criteria, in which coordinate $k$ minimizes
$\sum_i[(1+\theta_k)Z_{ik}-Y_{ik}]^2$ over $\theta_k\in[-2,2]$, where the normalized
ambient-space coordinate of the observed response is
\[
Y_{ik}:=\frac{\langle f_1^{(i)},\varphi_k\rangle_{\mathcal F}}{\sqrt{\lambda_k}}
=a_kZ_{ik}+\xi_{ik}.
\] The unconstrained minimizer is the ordinary
least-squares slope,
\[
\widetilde\theta_k:=\frac{\sum_iZ_{ik}Y_{ik}}{\sum_iZ_{ik}^2}-1
=\theta_{N,k}^*+\frac{\sum_iZ_{ik}\xi_{ik}}{\sum_iZ_{ik}^2},
\]
using $Y_{ik}=(1+\theta_{N,k}^*)Z_{ik}+\xi_{ik}$ with $\theta_{N,k}^*=a_k-1$. Since
$\theta_{N,k}^*\in[-1,0]\subset[-2,2]$ and projection onto the convex set $[-2,2]$ is
nonexpansive, the constrained minimizer $\widehat\theta_k:=\Pi_{[-2,2]}(\widetilde\theta_k)$ satisfies
\[
|\widehat\theta_k-\theta_{N,k}^*|
\ \le\ |\widetilde\theta_k-\theta_{N,k}^*|
\ =\ \Big|\frac{U_k}{V_k}\Big|,
\qquad
U_k:=\tfrac1N\textstyle\sum_iZ_{ik}\xi_{ik},
\qquad
V_k:=\tfrac1N\textstyle\sum_iZ_{ik}^2.
\]
By the standard chi-square lower tail bound, $\mathbb P(V_k<\tfrac12)\le e^{-cN}$ for an
absolute constant $c>0$, so by a union bound over $k\le d_N$,
\[
\mathbb P\Big(\min_{k\le d_N}V_k<\tfrac12\Big)\ \le\ d_Ne^{-cN}\ \to\ 0
\]
whenever $\log d_N=o(N)$, which holds under every sieve-dimension choice used in this
paper. On the event $\{\min_kV_k\ge\tfrac12\}$: conditionally on $\{Z_{ik}\}$, $U_k$ is
Gaussian with $\mathrm{Var}(U_k\mid Z)=\sigma^2V_k/N$, and $U_k/V_k$, $k=1,\ldots,d_N$,
are conditionally independent because the $\xi_{ik}$ are independent across $k$. Their
conditional variance $\sigma^2/(NV_k)\le2\sigma^2/N$. So, conditionally on this event,
$\sum_{k\le d_N}(U_k/V_k)^2$ is a weighted sum of independent squared Gaussians with
$\mathbb E\big[\sum_k(U_k/V_k)^2\mid Z\big]\le2\sigma^2d_N/N$. Since the conditional
covariance is deterministic given $Z$ on this event, Gaussian concentration for the
$1$-Lipschitz map $x\mapsto\|x\|$ applies, giving, for every
$\delta\in(0,1)$,
\begin{equation}
\label{eq:gaussian-exact-rate}
\mathbb P\Big(\|\widehat\theta_N-\theta_N^*\|>2\sigma\sqrt{\tfrac{2(d_N+\log(2/\delta))}N}\ \Big|\ Z\Big)
\ \le\ \delta
\qquad\text{on }\{\min_kV_k\ge\tfrac12\},
\end{equation}
and hence unconditionally, with probability $\ge1-\delta-d_Ne^{-cN}$,
\[
\|\widehat\theta_N-\theta_N^*\|\ \le\ C_5\,\sigma\sqrt{\tfrac{d_N+\log(1/\delta)}N}
\]
for an absolute constant $C_5$. This gives an unconditional bound on the parameter error of the global ERM over
$\Theta_N=[-2,2]^{d_N}$.

For the expected-risk bound \eqref{eq:gaussian-expected-risk}, conditionally on
the design $\widetilde\theta_k-\theta_{N,k}^*$ is centered Gaussian with variance
$\sigma^2/\sum_iZ_{ik}^2$, and $\sum_iZ_{ik}^2\sim\chi_N^2$ independently of the noise. Thus, for $N>2$,
\[
\mathbb E\big[(\widetilde\theta_k-\theta_{N,k}^*)^2\big]
=\sigma^2\,\mathbb E\big[1/\chi_N^2\big]=\frac{\sigma^2}{N-2}.
\]
Projection onto the parameter box is nonexpansive, so
$\mathbb E\|\widehat\theta_N-\theta_N^*\|^2\le\sigma^2d_N/(N-2)$. The same argument applies to the widened box of Section~\ref{subsec:gaussian}. The two error terms are orthogonal because $\widehat T_N-T^{(d_N)}$ is supported on coordinates $k\le d_N$ and $T^{(d_N)}-T$ on coordinates $k>d_N$. Since $T_{\theta_N^*,N}=T^{(d_N)}$,
\[
\mathbb E\big\|\widehat T_N-T\big\|_{L^2(\pi_0;\mathcal H)}^2
=\mathbb E\|\widehat\theta_N-\theta_N^*\|^2+\sum_{k>d_N}(a_k-1)^2
\le\frac{\sigma^2d_N}{N-2}+B^2d_N^{-2s},
\]
uniformly over $T\in\mathcal W^s_{\mathrm{diag}}(B)$. This proves
\eqref{eq:gaussian-expected-risk}. Since $L_N=1$,
$\|\widehat T_N-T_N^*\|_{L^2(\pi_0;\mathcal H)}
=\|\widehat\theta_N-\theta_N^*\|$.

For Assumption~\ref{ass:localization}, let
$r_N=\upsilon_N\sigma\sqrt{d_N/N}$ with $\upsilon_N\to\infty$ slowly enough that
$r_N<1/2$. When $d_N/N\to0$, such a sequence can be chosen so that
$\mathcal N_N=B(\theta_N^*,r_N)$ remains inside $\Theta_N$. Applying
\eqref{eq:gaussian-exact-rate} with, for example, $\delta_N=1/d_N$ gives
$\mathbb P(\widehat\theta_N\notin\mathcal N_N)\to0$.

Consequently $\kappa_N=L_N/\mu_N=1$ exactly, and, since $L_N=1$, $\kappa_NL_N=1$.
\end{proof}

\subsection{Nonlinear block-interaction class}
\label{app:block-proofs}

\subsubsection{Proof of Proposition~\ref{prop:block-population}}

\begin{proof}
The centered summands $a_j[\psi(\hat e_{B_j})-\mu_\psi]$ are independent with summable
variances. Their partial sums converge in $L^2(\gamma)$, and their gradients converge
in $L^2(\gamma;\mathcal H)$ because the squared norm of block $j$ is
$\varsigma^2a_j^2$. Closedness of the Cameron--Martin gradient gives
$\phi_a\in\mathbb D^{1,2}(\gamma)$ and the coefficient formula. The weighted bound
follows from $(2j-1)^{2s},(2j)^{2s}\le2^{2s}j^{2s}$.

The centered series also converges almost surely. On its convergence set, Taylor's
formula bounds the remainder in the increment on block $j$ by
$|a_j|\,|h_{B_j}|^2/2$. The linear increments are summable by Cauchy--Schwarz, since
$\sum_j a_j^2\|\nabla\psi(z_{B_j})\|^2<\infty$ and $h\in\mathcal H$.
Thus this set is invariant under Cameron--Martin shifts. The finite partial potentials
plus $|h|_{\mathcal H}^2/2$ are convex and converge pointwise on every such slice.
Their limit proves that $\phi_a$ has the required $1$-convex version.

Write $S_a(z)=z+a\nabla\psi(z)$. For $|a|\le1/2$, this is a smooth strongly monotone
bi-Lipschitz map and is optimal from $\gamma_2=N(0,I_2)$ to its pushforward $\nu_a$.
For any coupling of the infinite-dimensional marginals, Tonelli's theorem bounds the
cost below by the sum of the blockwise optimal costs. The map $T_a$ attains each of
these bounds. Uniqueness of the blockwise optimal couplings also gives uniqueness
of the full optimal coupling. Its cost is $\varsigma^2\sum_j a_j^2<\infty$.
The displacement is uniformly bounded in $\mathcal H$ by $(\sum_j a_j^2)^{1/2}$;
the continuous embedding into $\mathcal F$ also gives finite second moments.

Each $\nu_a$ has a positive smooth density. Change of variables and Gaussian integration
by parts give
\begin{align*}
\mathrm{KL}(\nu_a\|\gamma_2)
&=\mathbb E\left[a\langle Z,\nabla\psi(Z)\rangle
 +\tfrac{a^2}2\|\nabla\psi(Z)\|^2-\log\det(I_2+a\nabla^2\psi(Z))\right]\\
&\le \tfrac{a^2}2+2a^2\le3a^2,
\end{align*}
since $\mathbb E\langle Z,\nabla\psi(Z)\rangle=\mathbb E\operatorname{tr}\nabla^2\psi(Z)$
and $|\log\det(I+M)-\operatorname{tr}M|\le\|M\|_F^2$ for symmetric
$\|M\|_{\rm op}\le1/2$. Thus $\sum_jH^2(\nu_{a_j},\gamma_2)<\infty$.
Kakutani's criterion, with equivalence of every factor pair, gives $\pi_1\sim\gamma$.
The coordinate products are supported on $\{z:\sum_k\lambda_kz_k^2<\infty\}$ by the
second-moment bound, so this equivalence transfers to $\mathcal F$.

A nonzero block has a non-Gaussian target. The quadratic-cost
optimal map between nondegenerate Gaussian measures is affine. If $S_a$ were affine,
its bounded displacement would be constant; oddness would make that constant zero,
contrary to $a\ne0$ and $\nabla\psi\not\equiv0$.
The two coordinates of $S_a(Z)$ also have covariance
\[
-\frac{2a e^{-1}}3+a^2 C_{12},\qquad
C_{12}=\frac{(e^{-1/2}-e^{-5/2})/2+(1-e^{-4})/2}{9}.
\]
Here $0<C_{12}<1/9$. For $0<|a|\le1/2$ the displayed covariance is nonzero
(the linear coefficient $2/(3e)$ exceeds $|a|C_{12}$), proving dependence.
\end{proof}

\subsubsection{Proof of Theorem~\ref{thm:block-upper}}

\begin{proof}
Every sieve map with $m$ blocks agrees with the identity beyond coordinate $2m$, and the
Cameron--Martin basis is orthonormal, so exactly as in \eqref{eq:exact-split}
\[
  \big\|\widehat T_m-T_a\big\|_{L^2(\gamma;\mathcal H)}^2
  =\varsigma^2\sum_{j\le m}\big(\widehat a_j-a_j\big)^2
  +\varsigma^2\sum_{j>m}a_j^2 ,
\]
using \eqref{eq:block-coef} and the linearity of $a\mapsto T_{a,m}$, with
$\|T_{a,m}-T_{a',m}\|_{L^2(\gamma;\mathcal H)}^2=\varsigma^2\|a-a'\|^2$. The bias term is
at most $\varsigma^2m^{-2s}\sum_{j>m}j^{2s}a_j^2\le\varsigma^2B^2m^{-2s}$.

For the variance term fix $j\le m$ and write $S_j:=\sum_{i\le N}\|\Psi_{ij}\|^2$. By
\eqref{eq:block-regression} the unclipped least-squares coefficient satisfies
$\widetilde a_j-a_j=S_j^{-1}\sum_i\langle\Psi_{ij},\xi_{i,B_j}\rangle$, which conditionally
on the design is centered Gaussian with variance $\sigma^2/S_j$. The summands
$\|\Psi_{ij}\|^2$ are
i.i.d., take values in $[0,1]$, and have mean $\varsigma^2$, so Hoeffding's inequality
gives $\mathbb P(S_j<\tfrac12N\varsigma^2)\le e^{-cN}$ with $c:=\tfrac12\varsigma^4$. On the
complementary event the conditional variance is at most $2\sigma^2/(N\varsigma^2)$.
Projection onto $[-a_0,a_0]$ is nonexpansive and $a_j$ lies in that interval, so
$|\widehat a_j-a_j|\le|\widetilde a_j-a_j|$, while $|\widehat a_j-a_j|\le2a_0$ always.
Splitting on the two events,
\[
  \mathbb E\big[(\widehat a_j-a_j)^2\big]
  \ \le\ \frac{2\sigma^2}{N\varsigma^2}+4a_0^2e^{-cN}.
\]
Multiplying by $\varsigma^2$ and summing over $j\le m$ gives the first display. Balancing
$2\sigma^2m/N$ against $\varsigma^2B^2m^{-2s}$ at $m\asymp(B^2N/\sigma^2)^{1/(2s+1)}$ makes
both of order $B^{2/(2s+1)}(\sigma^2/N)^{2s/(2s+1)}$, while
$me^{-cN}\to0$ faster than any power of $N$; Jensen's inequality applied to the square
root gives the second display.
\end{proof}

\subsubsection{Proof of Theorem~\ref{thm:block-lower}}

\begin{proof}
The argument is that of Theorem~\ref{thm:minimax}, with the diagonal perturbation replaced
by a block perturbation, so we only record the three quantities the Fano method needs.

Fix a resolution $m\ge1$, perturb the band of blocks $j\in\{m+1,\dots,2m\}$, and for
$\omega\in\{0,1\}^m$ set $a^\omega_j:=\tau\omega_j$ on the band and $0$ elsewhere, with
$\tau\in(0,a_0]$, writing $T_\omega:=T_{a^\omega}$. Let
$\Omega\subseteq\{0,1\}^m$ be the Varshamov--Gilbert subset supplied by the proof of
Lemma~\ref{lem:hypothesis}(ii) with the band size $d$ there replaced by $m$, so that
$\log|\Omega|\ge cm$ and $d_H\ge m/8$ on $\Omega$.

By \eqref{eq:block-coef},
$\sum_jj^{2s}(a^\omega_j)^2\le\tau^2(2m)^{2s}m$, so $a^\omega\in\mathcal A_s(B,a_0)$ as
soon as $2^s\tau m^{s+1/2}\le B$ and $\tau\le a_0$.

Since the blocks are orthogonal in $\mathcal H$ and
$\|T_{a}-T_{a'}\|^2=\varsigma^2\|a-a'\|^2$,
$\|T_\omega-T_{\omega'}\|_{L^2(\gamma;\mathcal H)}^2=\varsigma^2\tau^2d_H(\omega,\omega')
\ge\varsigma^2\tau^2m/8$ for distinct $\omega,\omega'\in\Omega$.

Conditionally on the design, the observation on block $j$ is
$N(a_j\Psi_{ij},\sigma^2I_2)$ by \eqref{eq:block-regression}, and the laws agree off the
band. The Gaussian shift formula and integration over the design give
\[
  \mathrm{KL}\big(P_{T_\omega}\|P_{T_{\omega'}}\big)
  =\frac{\tau^2}{2\sigma^2}\sum_{j\ \rm band}(\omega_j-\omega'_j)^2\,
  \mathbb E\|\nabla\psi(Z_{B_j})\|^2
  =\frac{\tau^2\varsigma^2}{2\sigma^2}\,d_H(\omega,\omega')
  \ \le\ \frac{\tau^2\varsigma^2m}{2\sigma^2},
\]
and tensorization over the $N$ pairs multiplies this by $N$.

These are the three bounds of Lemma~\ref{lem:hypothesis}, with $\tau^2$ replaced
by $\varsigma^2\tau^2$ in the separation and in the divergence, and with $d$ replaced by
$m$. Choosing $\tau^2\asymp\sigma^2/(\varsigma^2N)$ and
$m\asymp(B^2N\varsigma^2/\sigma^2)^{1/(2s+1)}$ exactly as in the proof of
Theorem~\ref{thm:minimax}, and noting that $\tau\to0$ so that $\tau\le a_0$ for all large
$N$, the same Fano and Markov steps give the stated bound, with $c'$ absorbing
$\varsigma$.
\end{proof}

\subsection{Continuous two-groups model}
\label{sup:two-groups-proofs}

\subsubsection{Proof of Theorem~\ref{thm:logodds-fast}} \label{app:proof-logodds-fast}

\begin{proof}
Throughout, $\delta\in\mathcal D_\rho(c_-,c_+)$ is fixed, $K=L_0r_0$ is the envelope
of \eqref{eq:g-envelope}, and $C,C',\dots$ denote constants depending only on $\sigma$,
$r_0$ and $L_0$. All statements are uniform in $\delta$ over the class, since $K$,
$r_0$ and the mesh are.

Every sieve map acts as the
identity on the coordinates $k>d_N$, and the Cameron--Martin basis is orthonormal, so for
any $\eta$
\begin{equation}
\label{eq:exact-split}
  \big\|T_{\ell(\eta),N}-T_\delta\big\|_{L^2(\gamma;\mathcal H)}^2
  =\sum_{k\le d_N}\big\|M_{\ell(\eta_k)}-M_{\delta_k}\big\|_{L^2(\Phi)}^2
  +\sum_{k>d_N}\|u_k\|_{L^2(\gamma)}^2 ,
\end{equation}
 The second sum is at most
$c_+W_0^2\sum_{k>d_N}k^{-2\rho}\le c_+W_0^2d_N^{-(2\rho-1)}/(2\rho-1)$ by
\eqref{eq:two-groups-bias}, which is the bias term of
\eqref{eq:logodds-fast-risk}. It remains to bound
$V(\widehat\eta_N):=\sum_{k\le d_N}v_k(\widehat\eta_k)$, where
$v_k(\eta):=\|M_{\ell(\eta)}-M_{\delta_k}\|_{L^2(\Phi)}^2$.

Fix $k\le d_N$ and
$\eta\in[B_--2\rho\log k,B_+-2\rho\log k]$, and write $g=g_\eta$ as in
\eqref{eq:excess-loss}, so $\|g\|_\infty\le K$ by \eqref{eq:g-envelope} and
$v:=v_k(\eta)=\|g\|_{L^2(\Phi)}^2$. Because the contrast \eqref{eq:def-Rhat-circ} is
separable across coordinates, its $k$th block equals, up to an additive term free of
$\eta$, the empirical average of $L_\eta(Z_{ik},\xi_{ik})=g(Z_{ik})^2-2\xi_{ik}g(Z_{ik})$
over $i\le N$, with $Z_{ik}\sim N(0,1)$ and $\xi_{ik}\sim N(0,\sigma^2)$ independent.
Since $\mathbb E[\xi\mid Z]=0$, $\mathbb E L_\eta=v$.

We compute the moment generating function exactly. Conditioning on $Z$ and using
$\mathbb E[e^{-2\varkappa\xi g}\mid Z]=e^{2\varkappa^2\sigma^2g^2}$,
\[
  \mathbb E\big[e^{\varkappa L_\eta}\big]
  =\mathbb E_Z\big[e^{u g(Z)^2}\big],
  \qquad u:=\varkappa+2\varkappa^2\sigma^2 .
\]
Set $\varkappa_0:=\min\{(4\sigma^2)^{-1},(2K^2)^{-1}\}$ and let $|\varkappa|\le\varkappa_0$, so
that $|u|\le\tfrac32|\varkappa|$ and $|u|K^2\le\tfrac34$. For $x\in[0,K^2]$ and
$|ux|\le\tfrac34$ one has $|e^{ux}-1-ux|\le u^2x^2$, whence, using
$\mathbb E[g^4]\le K^2v$,
\[
  \mathbb E_Z\big[e^{ug^2}\big]\ \le\ 1+uv+u^2K^2v\ \le\ \exp\big(uv+u^2K^2v\big),
\]
and therefore
\[
  \log\mathbb E\big[e^{\varkappa(L_\eta-v)}\big]
  \ \le\ \big(u-\varkappa\big)v+u^2K^2v
  \ \le\ \varkappa^2v\big(2\sigma^2+\tfrac94K^2\big)=:\bar C\varkappa^2v ,
  \qquad|\varkappa|\le\varkappa_0 .
\]
This is the Bernstein condition \eqref{eq:var-mean} in exponential form. The variance
proxy is proportional to the mean $v$, with a constant free of $d_N$ and of $k$.

Write
$\mathbb P_NL_\eta:=\tfrac1N\sum_{i\le N}L_\eta(Z_{ik},\xi_{ik})$. By the Chernoff bound
with the exponential-moment estimate above, optimizing over $\varkappa\in(0,\varkappa_0]$, for every $t>0$
\[
  \mathbb P\Big(\big|\mathbb P_NL_\eta-v\big|
  \ \ge\ 2\sqrt{\tfrac{\bar Cvt}{N}}+\tfrac{2t}{\varkappa_0N}\Big)\ \le\ 2e^{-t}.
\]
By $2\sqrt{\bar Cvt/N}\le\tfrac v2+2\bar Ct/N$, the event
\begin{equation}
\label{eq:bernstein-pointwise}
  \tfrac12 v-\tfrac{C_\ast t}{N}\ \le\ \mathbb P_NL_\eta\ \le\ \tfrac32v+\tfrac{C_\ast t}{N},
  \qquad C_\ast:=2\bar C+2/\varkappa_0,
\end{equation}
has probability at least $1-2e^{-t}$.

For each fixed $k$, choose
$\widetilde\eta_k\in\mathcal G_k$ with $|\widetilde\eta_k-\eta_k^*|\le h$.
Then $v_k(\widetilde\eta_k)\le L_0^2h^2$.
A union bound over $\mathcal G_k$, with $t=\log(2|\mathcal G_k|)+u$, shows that
\eqref{eq:bernstein-pointwise} holds at every grid point with probability at least
$1-e^{-u}$. On this event the coordinatewise grid minimizer satisfies
\[
\tfrac12v_k(\widehat\eta_k)-C_*t/N
\le\mathbb P_NL_{\widehat\eta_k}
\le\mathbb P_NL_{\widetilde\eta_k}
\le\tfrac32L_0^2h^2+C_*t/N.
\]
Consequently
\[
\mathbb P\left(v_k(\widehat\eta_k)>3L_0^2h^2+
\frac{4C_*[\log(2(r_0/h+2))+u]}N\right)\le e^{-u}.
\]
Integrating this tail bound and summing expectations over $k\le d_N$ gives
\begin{equation}
\label{eq:V-highprob}
\mathbb E V(\widehat\eta_N)
\le3L_0^2d_Nh^2+
\frac{4C_*d_N[1+\log(2(r_0/h+2))]}N.
\end{equation}
No union over coordinates is needed for this expectation bound.
Combining with the split \eqref{eq:exact-split} proves \eqref{eq:logodds-fast-risk}.

Define the nonnegative coordinatewise excess
$\Delta_k=\mathbb P_NL_{\widetilde\eta_k^{\rm fit}}-
\min_{\eta\in\mathcal G_k}\mathbb P_NL_\eta$, where
$\widetilde\eta^{\rm fit}\in\mathcal G_N$ is an approximate minimizer.
The assumed total objective gap gives $\sum_k\Delta_k\le\varepsilon_N$.
The preceding comparison adds $2\Delta_k$ to its right side. Applying the same
coordinatewise tail integration to
$(v_k(\widetilde\eta_k^{\rm fit})-2\Delta_k)_+$ and summing proves the stated
additional $2\varepsilon_N$ bound. The approximate coordinate fits need not be independent.

The specified mesh range implies
$h^2\le C_0^2\log N/N$ and $\log(2+r_0/h)\le C_A\log N$.
Balancing $d_N\log N/N$ with $d_N^{-(2\rho-1)}$ gives
$d_N\asymp(N/\log N)^{1/(2\rho)}$. Jensen's inequality gives
\eqref{eq:logodds-fast-rate}. All bounds are uniform over the class because the box
contains every $\eta^*$ and its width depends only on $c_-,c_+$.
\end{proof}

\section{Additional Gaussian calculations}
\label{sup:gaussian-details}

We consider centered Gaussian measures with commuting covariances, including the conjugate Gaussian inverse problem. The transport map and truncation bias are available in closed form, so the regularity exponent and approximation rate follow directly from the covariance spectrum. Let
\[
\pi_0 = N(0,C_0),
\qquad
\pi_1 = N(0,C_1),
\]
where $C_0$ and $C_1$ are self-adjoint, strictly positive, trace-class covariance
operators on $\mathcal{F}$ that commute, so that some orthonormal basis
$\{\varphi_k\}_{k\ge1}$ of $\mathcal{F}$ satisfies
\[
C_0 \varphi_k = \lambda_k^{(0)} \varphi_k,
\qquad
C_1 \varphi_k = \lambda_k^{(1)} \varphi_k,
\qquad
\lambda_k^{(0)},\lambda_k^{(1)} > 0.
\]

We assume the commuting pair satisfies \textnormal{(A1)}--\textnormal{(A3)}. With the Cameron--Martin reference
taken to be the source, $\gamma=\pi_0$, as fixed below, the transport map
\eqref{eq:gauss-OT-diagonal} has displacement
$T-I=\sum_k(a_k-1)\hat e_k\,e_k$ with
$a_k=(\lambda_k^{(1)}/\lambda_k^{(0)})^{1/2}$, and the coordinates $\hat e_k$ are
standard normal under $\pi_0$, so
$\|T-I\|_{L^2(\pi_0;\mathcal H)}^2=\sum_k(1-a_k)^2$. The finite Cameron--Martin
transport cost \textnormal{(A2)} therefore requires in particular
\[
\sum_{k\ge1}(1-a_k)^2<\infty ,
\]
a restriction on the eigenvalue ratios. The polynomial regime of Remark~\ref{rem:gaussian-rates}, where
$1-a_k\asymp k^{-\rho}$, satisfies it for $\rho>\tfrac12$.

Consider the conjugate Gaussian inverse problem. Let $\{\varphi_k\}$ diagonalize both the prior
covariance $C_\gamma=C_0$ and the forward heat semigroup, and let
$u\sim\gamma=N(0,C_0)$ be the unknown initial condition, with independent coefficients
$u_k\sim N(0,\lambda_k^{(0)})$. We observe
$Y_k=\kappa_ku_k+\varepsilon\xi_k$, where $\xi_k$ are i.i.d.\ $N(0,1)$,
$\kappa_k\ge0$ are known forward-operator eigenvalues, and $\varepsilon>0$ is the noise level.
Gaussian conjugacy gives
$u_k\mid Y\sim N(m_k,v_k)$ with
$v_k=\big(1/\lambda_k^{(0)}+\kappa_k^2/\varepsilon^2\big)^{-1}$ and
$m_k=v_k\kappa_kY_k/\varepsilon^2$. The centered prior-to-posterior transport is diagonal with coefficients
$a_k=\sqrt{v_k/\lambda_k^{(0)}}=\big(1+\lambda_k^{(0)}\kappa_k^2/\varepsilon^2\big)^{-1/2}$.
The posterior mean is an explicit affine shift. We study the centered transport, whose
covariance operators $C_0$ and $C_1=C_Y$ are available in closed form. This gives a
direct calculation of the transport coefficients and the corresponding convergence rate
\citep{stuart2010inverse,knapikvdvaartvzanten2011,knapikvdvaartvzanten2013}.

We take the Cameron--Martin reference to be the source, $\gamma=\pi_0$, so that
$\mathcal H=C_0^{1/2}(\mathcal F)$, $\langle h,k\rangle_{\mathcal H}=\langle C_0^{-1/2}h,C_0^{-1/2}k\rangle_{\mathcal F}$, and,
as in Section~\ref{subsec:prelim}, the induced Cameron--Martin orthonormal basis of
$\mathcal H$ is $e_k:=C_0^{1/2}\varphi_k=\sqrt{\lambda_k^{(0)}}\,\varphi_k$. Hence $|e_k|_{\mathcal H}=1$, while $|\varphi_k|_{\mathcal H}^2=(\lambda_k^{(0)})^{-1}$. Under the commuting
assumption, the quadratic-cost optimal transport map $T$ from $\pi_0$ to $\pi_1$ is
linear and diagonal in $\{\varphi_k\}$,
\begin{equation}
\label{eq:gauss-OT-diagonal}
T(x) = \sum_{k\ge 1} a_k \langle x,\varphi_k\rangle_{\mathcal{F}}\, \varphi_k,
\qquad
a_k=\sqrt{\frac{\lambda_k^{(1)}}{\lambda_k^{(0)}}}.
\end{equation}

For each $d\in\mathbb{N}$, the canonical cylindrical truncation of $T$ is
\begin{equation}
\label{eq:Td-definition}
T_d(x)
:=
\sum_{k=1}^d a_k \langle x,\varphi_k\rangle_{\mathcal{F}}\, \varphi_k
+
\sum_{k>d} \langle x,\varphi_k\rangle_{\mathcal{F}}\, \varphi_k,
\end{equation}
so that
\begin{equation}
\label{eq:Td-minus-T}
(T_d-T)(x)
=
\sum_{k>d} (1-a_k)\langle x,\varphi_k\rangle_{\mathcal{F}}\, \varphi_k,
\end{equation}
and
\begin{equation}
\label{eq:explicit-bias}
\|T_d-T\|_{L^2(\pi_0;\mathcal H)}^2=\sum_{k>d}(1-a_k)^2
=\sum_{k>d}\Bigg(1-\sqrt{\tfrac{\lambda_k^{(1)}}{\lambda_k^{(0)}}}\Bigg)^{2}.
\end{equation}

Now fix $N$ and set $d:=d_N$. For the cylindrical-sieve family of Definition~\ref{def:cylindrical}, we use the Cameron--Martin basis $\{e_k\}$ and the
associated first-chaos coordinates, not the raw $\mathcal F$-inner products against
$\{\varphi_k\}$: the feature map is
\[
\Phi_N(f)
:=
\big(\hat e_1(f),\ldots,\hat e_{d_N}(f)\big)
=
\Big(\frac{\langle f,\varphi_1\rangle_{\mathcal{F}}}{\sqrt{\lambda_1^{(0)}}},\ldots,
\frac{\langle f,\varphi_{d_N}\rangle_{\mathcal{F}}}{\sqrt{\lambda_{d_N}^{(0)}}}\Big)
\in\mathbb{R}^{d_N},
\]
which has i.i.d.\ $N(0,1)$ entries under $\pi_0$, as required by
Definition~\ref{def:cylindrical}. Let
$z=\Phi_N(f)$ and consider the quadratic parametric family
\[
g_B(z):=\frac12 z^\top B z,
\qquad
B=B^\top\in\mathbb{R}^{d_N\times d_N}.
\]
Define the cylindrical potential
\[
\phi_{B,N}(f):=g_B(\Phi_N(f)).
\]
The associated sieve map is, by the chain rule \eqref{eq:chain} with $h_k=e_k$,
\begin{equation}
\label{eq:sieve-map-form}
T_{B,N}(f)
:=
f+\sum_{k=1}^{d_N}\partial_k g_B(\Phi_N(f))\,e_k.
\end{equation}
Since $\nabla g_B(z)=Bz$, we have $\partial_k g_B(z)=(Bz)_k$, so, writing $e_k$ in terms
of $\varphi_k$,
\begin{equation}
\label{eq:sieve-map-linear}
T_{B,N}(f)
=
f+\sum_{k=1}^{d_N}(B\Phi_N(f))_k\,e_k
=
f+\sum_{k=1}^{d_N}(B\Phi_N(f))_k\,\sqrt{\lambda_k^{(0)}}\,\varphi_k.
\end{equation}

Let $A$ be the diagonal operator in \eqref{eq:gauss-OT-diagonal}, and define
\[
B^\ast:=\mathrm{diag}(a_1-1,\ldots,a_{d_N}-1)\in\mathbb{R}^{d_N\times d_N}.
\]
Then
\[
T_{B^\ast,N}(f)=T_{d_N}(f).
\]
Hence, in this benchmark, the sieve class generated by quadratic potentials contains the canonical truncation map exactly. In particular,
\begin{equation}
\label{eq:explicit-bias-final}
\inf_{\theta\in\Theta_N}\|T_{\theta,N}-T\|_{L^2(\pi_0;\mathcal H)}
\le
\|T_{B^\ast,N}-T\|_{L^2(\pi_0;\mathcal H)}
=
\|T_{d_N}-T\|_{L^2(\pi_0;\mathcal H)},
\end{equation}
and therefore the sieve approximation bias is explicitly controlled by the tail eigenstructure through \eqref{eq:explicit-bias}. The derivations of \eqref{eq:gauss-OT-diagonal}, \eqref{eq:explicit-bias}, and \eqref{eq:explicit-bias-final} are collected in Section~\ref{app:proof-gaussian-bias}.

Section~\ref{subsec:spectral} gives $a_L,a_\mu,s$ explicitly for this benchmark under a spectral-decay model.

\begin{rem}
\label{rem:gaussian-rates}
Suppose $\lambda_k^{(0)}\asymp k^{-\alpha}$ for some $\alpha>1$. This restriction is needed for $C_0$ to be trace class. In the conjugate parametrization of the previous display,
\[
\frac{\lambda_k^{(0)}\kappa_k^2}{\varepsilon^2}=ck^{-\rho}+o(k^{-\rho}),
\qquad \rho>\tfrac12,\ c>0,
\]
i.e.\ the forward operator's information about coordinate $k$, relative to the prior
variance and noise level, decays polynomially, the mildly ill-posed regime of
Bayesian linear inverse problems \citep{knapikvdvaartvzanten2011,knapikvdvaartvzanten2013}.
Here $a_k=(1+\lambda_k^{(0)}\kappa_k^2/\varepsilon^2)^{-1/2}$, so a first-order expansion
gives $a_k=1-\tfrac12ck^{-\rho}+o(k^{-\rho})$, and $a_k<1$ whenever $\kappa_k\ne0$. By the chain rule,
$T-I=\nabla_{\mathcal H}\phi$ for the centered potential
\[
\phi(f)=\tfrac12\sum_k(a_k-1)\big(\hat e_k(f)^2-1\big),
\]
where the $-1$ leaves the gradient unchanged and makes the series converge in
$L^2(\gamma)$. The uncentered series does not converge there for
$\rho\in(\tfrac12,1]$. The
coefficient fields of Section~\ref{sec:approx-section} are $u_k(f)=(a_k-1)\hat e_k(f)$
with $\|u_k\|_{L^2(\pi_0)}^2\asymp k^{-2\rho}$, so $\rho-\tfrac12$ is $T$'s boundary
regularity exponent in the sense of Remark~\ref{rem:supremal}, and the direct tail sum
$\|T_d-T\|^2_{L^2(\pi_0;\mathcal H)}=\sum_{k>d}(a_k-1)^2\asymp d^{-2(\rho-1/2)}$, already
recorded in \eqref{eq:explicit-bias}, matches \eqref{eq:truncation} at that exponent.

This polynomial decay is an assumption, not a consequence. At a fixed observation time
the heat semigroup gives $\kappa_k=e^{-t\nu_k}$, so $1-a_k$ decays faster than any power
of $k$, the severely ill-posed case, for which $T$ lies in $\mathcal W^s$ for every
finite $s$ and the regularity exponent is not the operative measure of difficulty
\citep{knapikvdvaartvzanten2011,knapikvdvaartvzanten2013}. The mildly ill-posed case
worked out here is the regime in which the scale $\mathcal W^s$ is informative.
\end{rem}

\begin{rem}
\label{rem:supremal}
Where the coefficient fields obey $\|u_k\|_{L^2(\pi_0)}^2\asymp k^{-2\rho}$ on both sides,
$T\in\mathcal W^s$ for every $s<\rho-\tfrac12$, with $\|T\|_{\mathcal W^s}\uparrow\infty$
as $s\uparrow\rho-\tfrac12$ and no membership at the boundary itself, while the direct
truncation sum $\sum_{k>d}\|u_k\|_{L^2(\pi_0)}^2\asymp d^{-2(\rho-1/2)}$ attains the
boundary exponent. Statements written with $s=\rho-\tfrac12$ therefore refer to the
truncation calculation, which is carried out at the boundary exponent directly, and not to
a limit of statements at $s<\rho-\tfrac12$. In Section~\ref{subsec:bayesian-example} only
the one-sided bound $\|u_k\|_{L^2(\pi_0)}^2\lesssim k^{-2\rho}$ is available, so
membership at the boundary is not excluded there and the rates quoted are upper bounds.
\end{rem}

\begin{proposition}
\label{prop:gaussian-verify}
Take $\Theta_N:=[-2,2]^{d_N}$ with $B=\mathrm{diag}(\theta_1,\ldots,\theta_{d_N})$, so
that $\theta_N^*=(a_1-1,\ldots,a_{d_N}-1)\in(-1,0]^{d_N}$ lies in the interior of
$\Theta_N$, and $p_N=d_N$, $q=1$ in the notation of
Section~\ref{subsec:spectral}, and assume $a_k\le1$ for every $k$, as holds in the
conjugate model of Remark~\ref{rem:gaussian-rates}. Then
Assumption~\ref{ass:sieve-primitive}(B2) holds with $\mu_N=1$,
Assumption~\ref{ass:lipschitz-map} holds with $L_N=1$, and
Assumption~\ref{ass:localization} holds; hence $a_L=a_\mu=0$ and $\kappa_N=\kappa_NL_N=1$.
For every $\delta\in(0,1)$, with probability at least $1-\delta-d_Ne^{-cN}$
for an absolute constant $c>0$,
\[
\|\widehat\theta_N-\theta_N^*\|\ \le\ C_5\,\sigma\sqrt{\tfrac{d_N+\log(1/\delta)}N}.
\]
Both statements hold verbatim on any compact box $\Theta_N=[\ell,u]^{d_N}$ containing
$\theta_N^*$ in its interior, since only convexity of the box and interiority of the
oracle are used. On the widened box
\begin{equation}
\label{eq:widened-box}
\Theta_N^{(B)}:=\big[\underline a-2,\ B+1\big]^{d_N},
\end{equation}
which contains the coefficient range of the whole of $\mathcal W^s_{\mathrm{diag}}(B)$ in
its interior, one obtains in addition, for $N>2$,
\begin{equation}
\label{eq:gaussian-expected-risk}
\sup_{T\in\mathcal W^s_{\mathrm{diag}}(B)}
\mathbb E_T\big\|\widehat T_N-T\big\|_{L^2(\pi_0;\mathcal H)}^2
\ \le\ \frac{\sigma^2d_N}{N-2}+B^2d_N^{-2s} ,
\end{equation}
where $\widehat T_N$ is the empirical risk minimizer over $\Theta_N^{(B)}$. 
\end{proposition}

Section~\ref{app:proof-gaussian-verify} solves the separable quadratic problem in closed form. A direct $L^2$ calculation gives $L_N=1$, and the risk Hessian is $2I_{d_N}$ at every $\theta$, so $\mu_N=1$ is a valid strong-convexity constant. The parameter-error bound holds for the global minimizer over $\Theta_N$.

Combining the stochastic term $\sigma\sqrt{d_N/N}$ with the bias $d_N^{-s}$ gives
\[
d_N^\star\asymp (N/\sigma^2)^{1/(2s+1)},
\qquad
\|\widehat T_N-T\|_{L^2(\pi_0;\mathcal H)}
=O_p\big(\sigma^{2s/(2s+1)}N^{-s/(2s+1)}\big).
\]
For the expected norm, Jensen's inequality and
$d_N\asymp(N/\sigma^2)^{1/(2s+1)}$ give
\[
\sup_{T\in\mathcal W^s_{\mathrm{diag}}(B)}
\mathbb E_T\|\widehat T_N-T\|_{L^2(\pi_0;\mathcal H)}
\le\Big(\frac{\sigma^2d_N}{N-2}+B^2d_N^{-2s}\Big)^{1/2}
\lesssim\big(\sigma^2/N\big)^{s/(2s+1)}.
\]
This has the same order as the lower bound in Theorem~\ref{thm:minimax} on the aligned
diagonal subclass. The mixture class is treated by the direct risk bound in
Theorem~\ref{thm:logodds-fast}.

For every $s<\rho-\tfrac12$, Remark~\ref{rem:gaussian-rates} gives
$\|T\|_{\mathcal W^s}<\infty$. The coefficients $a_k-1$ lie in $(-1,0]$, are bounded
away from $-1$, and converge to zero under the spectral model. Hence
$T\in\mathcal W_{\mathrm{diag}}^s(B_s)$ with
$B_s=\|T\|_{\mathcal W^s}$. On the widened box
$\Theta_N^{(B_s)}$ in \eqref{eq:widened-box}, the diagonal quadratic sieve represents
every member of this subclass up to truncation, so Assumption~\ref{ass:expressivity}
holds uniformly with $C'=0$. The lower bound in Theorem~\ref{thm:minimax} applies to the
same subclass. Therefore the estimator is minimax rate-optimal over
$\mathcal W_{\mathrm{diag}}^s(B_s)$ for every $s<\rho-\tfrac12$.


\section{Numerical integration and coordinate tails}
\label{sup:numerical-integration}

For $x\ge0$, the implementation solves
\[
(1-\theta)\overline\Phi(t)+\tfrac\theta2e^{-t}=\overline\Phi(x)
\]
in log-survival form, with a bracket between $x$ and
$-\log(2\overline\Phi(x))$. Negative arguments use oddness. The bracket follows
because a mixture quantile lies between the corresponding component quantiles.

Write $b(\theta)=\|M_\theta-\mathrm{id}\|_2^2$.
For fixed $x$, $t_\theta=M_\theta(x)$ is continuous and differentiable on $(0,1)$,
with derivative $(\Phi(t_\theta)-F_\nu(t_\theta))/f_\theta(t_\theta)$.
If this derivative vanishes, the defining quantile equation implies
$\Phi(t_\theta)=\Phi(x)$, hence $t_\theta=x$ and $F_\nu(x)=\Phi(x)$.
In that case $M_\theta(x)=x$ for every $\theta$.
Otherwise the derivative never vanishes, and $|M_\theta(x)-x|$ increases from zero.
Thus $b$ is nondecreasing and $b(ck^{-2\rho})$ is nonincreasing in $k$.
For any partition of $\{d+1,\ldots,K\}$ into consecutive integer blocks $[a_j,b_j]$,
\[
\sum_j(b_j-a_j+1)b(cb_j^{-2\rho})
\le \sum_{k>d}b(ck^{-2\rho})
\le \sum_j(b_j-a_j+1)b(ca_j^{-2\rho})
+\frac{cW_0^2}{2\rho-1}K^{1-2\rho}.
\]
Geometrically increasing blocks allow a large cutoff without evaluating every
coordinate. The implementation uses $W_0^2\le(1+\sqrt2)^2$ for the remainder bound.
Endpoint integrals are evaluated numerically at two quadrature orders. The resulting
sensitivity measure is reported separately from the block interval and the analytic
remainder bound; these quantities are reported separately.

\section{Numerical records}
\label{sup:numerical-records}

\subsection{Detailed grid-experiment results}
\label{sup:grid-results}
\begin{table}[H]
\centering\small
\begin{tabular}{rrrrrrrr}
\hline
$\rho$ & $N$ & $d_N$ & Reps & Estimation risk & SE & Computed bias & Endpoint fraction\\
\hline
$0.75$ & 250 & 13 & 40 & $5.4945\times10^{-3}$ & $3.52\times10^{-4}$ & $3.1925\times10^{-3}$ & $0.383$\\
$0.75$ & 500 & 19 & 40 & $4.1504\times10^{-3}$ & $2.38\times10^{-4}$ & $2.2785\times10^{-3}$ & $0.361$\\
$0.75$ & 1000 & 28 & 30 & $2.9593\times10^{-3}$ & $1.14\times10^{-4}$ & $1.6197\times10^{-3}$ & $0.354$\\
$0.75$ & 2000 & 41 & 25 & $2.1881\times10^{-3}$ & $1.03\times10^{-4}$ & $1.1625\times10^{-3}$ & $0.381$\\
$0.75$ & 4000 & 61 & 20 & $1.6560\times10^{-3}$ & $6.70\times10^{-5}$ & $8.2639\times10^{-4}$ & $0.413$\\
$1.0$ & 500 & 9 & 40 & $1.9433\times10^{-3}$ & $1.42\times10^{-4}$ & $6.0887\times10^{-4}$ & $0.261$\\
$1.0$ & 1000 & 12 & 40 & $1.3967\times10^{-3}$ & $1.09\times10^{-4}$ & $3.9901\times10^{-4}$ & $0.344$\\
$1.0$ & 2000 & 16 & 30 & $8.9311\times10^{-4}$ & $6.06\times10^{-5}$ & $2.6179\times10^{-4}$ & $0.335$\\
$1.0$ & 4000 & 22 & 30 & $5.6253\times10^{-4}$ & $2.77\times10^{-5}$ & $1.6460\times10^{-4}$ & $0.356$\\
$1.0$ & 8000 & 30 & 20 & $3.7723\times10^{-4}$ & $2.09\times10^{-5}$ & $1.0511\times10^{-4}$ & $0.387$\\
$1.0$ & 16000 & 41 & 20 & $2.7516\times10^{-4}$ & $1.26\times10^{-5}$ & $6.7171\times10^{-5}$ & $0.401$\\
$1.5$ & 500 & 4 & 40 & $9.2475\times10^{-4}$ & $1.18\times10^{-4}$ & $1.5290\times10^{-4}$ & $0.188$\\
$1.5$ & 1000 & 5 & 40 & $5.2089\times10^{-4}$ & $4.89\times10^{-5}$ & $8.7569\times10^{-5}$ & $0.255$\\
$1.5$ & 2000 & 6 & 40 & $3.3376\times10^{-4}$ & $2.95\times10^{-5}$ & $5.5341\times10^{-5}$ & $0.237$\\
$1.5$ & 4000 & 8 & 30 & $1.7469\times10^{-4}$ & $1.49\times10^{-5}$ & $2.6794\times10^{-5}$ & $0.258$\\
$1.5$ & 8000 & 10 & 25 & $1.2522\times10^{-4}$ & $1.19\times10^{-5}$ & $1.5326\times10^{-5}$ & $0.308$\\
$1.5$ & 16000 & 12 & 25 & $9.5377\times10^{-5}$ & $1.17\times10^{-5}$ & $9.7759\times10^{-6}$ & $0.310$\\
\hline
\end{tabular}
\caption{Global grid estimator with class bounds $[1/4,1/3]$ and $h_N=\sqrt{\log N/N}$. The last column is the fraction of fitted coordinates, pooled over replicates, at which the minimizer is an endpoint of the coordinate interval.}
\label{tab:grid-details}
\end{table}

\subsection{Score-root implementation}
\label{sup:score-root}

Each fit
was obtained by bisection of a score equation when the endpoint scores had opposite
signs, and by comparison of the endpoints otherwise. This procedure did not certify a
global minimum of the nonconvex criterion. The records contain 155, 180 and 200
replicates at $\rho=0.75,1,1.5$, respectively. No repeated seed within a fixed
$(\rho,N)$ was found. The recorded tail cutoff is 150,000 for $\rho=0.75$ and $1$ and 40,000 for
$\rho=1.5$. Figure~\ref{fig:rate} and
Table~\ref{tab:rate} report these records. They differ from
Section~\ref{sec:numerical-illustration} in the estimator and in the search box, which was
built there from the class bounds $[c_-,c_+]$ and here from the single value $c=1/3$, so
the two sets of risks are not directly comparable. The computed biases differ by at most
$2.5\times10^{-6}$, which is below $0.6\%$ of the total risk in every setting.

\begin{figure}[H]
\centering
\includegraphics[width=\textwidth]{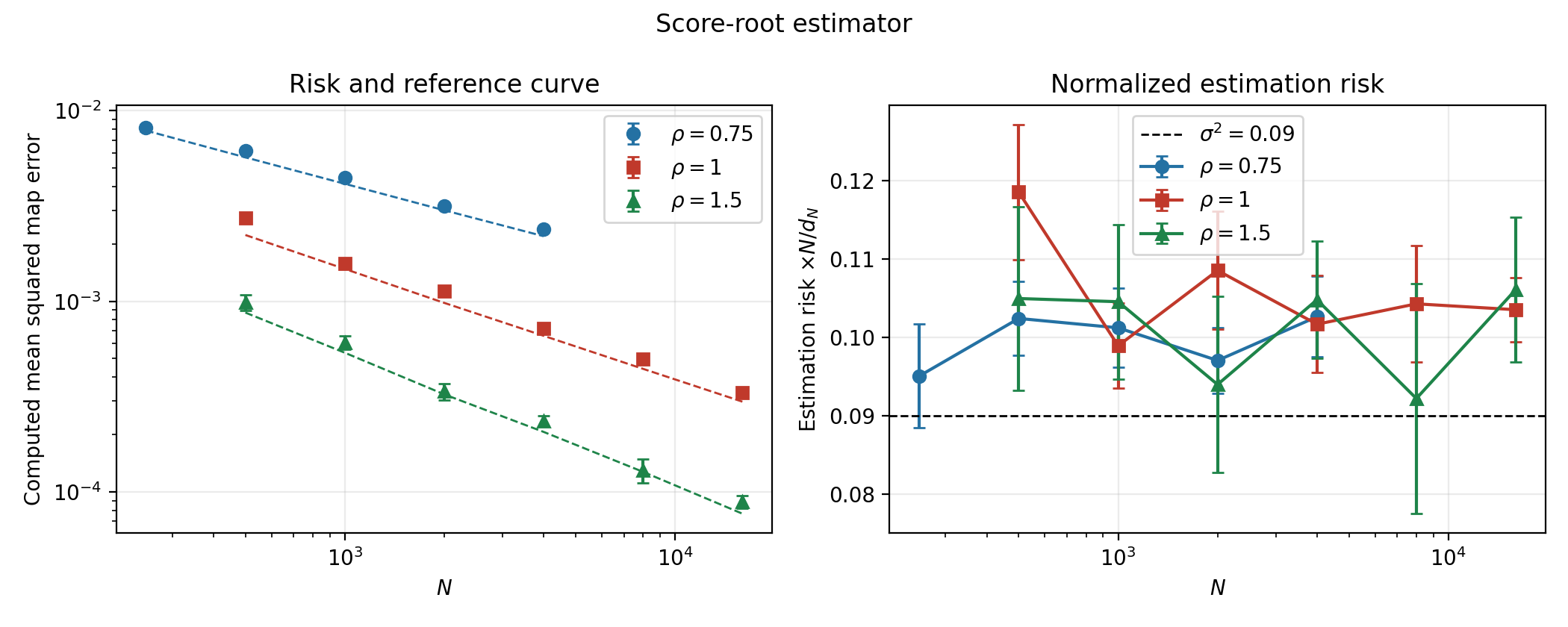}
\caption{Score-root estimator. Left: mean estimation risk plus the stored
finite-sum bias, with the reference $\sigma^2d_N/N$ plus the same bias. Right:
normalized estimation risk, with Monte Carlo standard-error bars. Numerical integration
and omitted-tail errors are not included in the error bars.}
\label{fig:rate}
\end{figure}

\begin{table}[H]
\centering
\small
\begin{tabular}{ccrcccc}
\hline
$\rho$ & $N$ & $d_N$ & replicates & estimation risk & computed bias &
$\text{estimation}\times N/d_N$\\
\hline
$0.75$ &   250 & 13 & 40 & $4.942\times10^{-3}$ & $3.190\times10^{-3}$ & $0.095$\\
       &   500 & 19 & 40 & $3.890\times10^{-3}$ & $2.276\times10^{-3}$ & $0.102$\\
       &  1000 & 28 & 30 & $2.833\times10^{-3}$ & $1.617\times10^{-3}$ & $0.101$\\
       &  2000 & 41 & 25 & $1.989\times10^{-3}$ & $1.160\times10^{-3}$ & $0.097$\\
       &  4000 & 61 & 20 & $1.565\times10^{-3}$ & $8.240\times10^{-4}$ & $0.103$\\
\hline
$1.0$  &   500 &  9 & 40 & $2.133\times10^{-3}$ & $6.084\times10^{-4}$ & $0.118$\\
       &  1000 & 12 & 40 & $1.187\times10^{-3}$ & $3.985\times10^{-4}$ & $0.099$\\
       &  2000 & 16 & 30 & $8.680\times10^{-4}$ & $2.613\times10^{-4}$ & $0.108$\\
       &  4000 & 22 & 30 & $5.590\times10^{-4}$ & $1.641\times10^{-4}$ & $0.102$\\
       &  8000 & 30 & 20 & $3.909\times10^{-4}$ & $1.046\times10^{-4}$ & $0.104$\\
       & 16000 & 41 & 20 & $2.652\times10^{-4}$ & $6.667\times10^{-5}$ & $0.103$\\
\hline
$1.5$  &   500 &  4 & 40 & $8.393\times10^{-4}$ & $1.524\times10^{-4}$ & $0.105$\\
       &  1000 &  5 & 40 & $5.226\times10^{-4}$ & $8.707\times10^{-5}$ & $0.105$\\
       &  2000 &  6 & 40 & $2.818\times10^{-4}$ & $5.484\times10^{-5}$ & $0.094$\\
       &  4000 &  8 & 30 & $2.095\times10^{-4}$ & $2.629\times10^{-5}$ & $0.105$\\
       &  8000 & 10 & 25 & $1.152\times10^{-4}$ & $1.483\times10^{-5}$ & $0.092$\\
       & 16000 & 12 & 25 & $7.955\times10^{-5}$ & $9.277\times10^{-6}$ & $0.106$\\
\hline
\end{tabular}
\caption{Score-root estimator, $\sigma=0.3$. The last column is the normalized estimation risk; $\sigma^2=0.09$ is a reference level.}
\label{tab:rate}
\end{table}

Using identical fixed weights for the observed and reference log-risk curves gives
slopes $(-0.4550,-0.5872,-0.6954)$ and $(-0.4602,-0.5796,-0.6988)$.
Within-$N$ bootstrap resampling gives 95\% intervals for their differences of
$[-0.0251,0.0356]$, $[-0.0365,0.0202]$ and $[-0.0547,0.0633]$.
These intervals describe Monte Carlo uncertainty in finite-range slopes, conditional
on the recorded biases and the original weights. They do not estimate uncertainty in
an asymptotic rate exponent. The curves are compatible with a reference estimation
scale $\sigma^2d_N/N$ over this range.

\end{document}